\documentclass[11pt]{amsart}
\usepackage[T1]{fontenc}
\usepackage{lmodern,microtype,amsmath,amssymb,amsthm,mathrsfs,esint}
\usepackage[a4paper,margin=1in]{geometry}
\usepackage{booktabs}
\usepackage[hidelinks]{hyperref}
\hypersetup{
bookmarksdepth=2,
pdftitle={On Critical Dimensions for Compactness in the Boundary Yamabe Problem, II},
pdfauthor={Liuwei Gong, Seunghyeok Kim, Monica Musso, Juncheng Wei},
pdfsubject={Compactness and critical dimensions for the boundary Yamabe equation},
pdfkeywords={boundary Yamabe problem, compactness, blow-up, critical dimension, manifolds with boundary}
}
\allowdisplaybreaks
\numberwithin{equation}{section}
\newtheorem{theorem}{Theorem}[section]
\newtheorem{maintheorem}{Theorem}
\newtheorem{masstheorem}{Theorem}

\newtheorem{maincorollary}{Corollary}
\newtheorem{proposition}[theorem]{Proposition}
\newtheorem{lemma}[theorem]{Lemma}
\newtheorem{corollary}[theorem]{Corollary}
\theoremstyle{definition}
\newtheorem{definition}[theorem]{Definition}
\newcommand{\R}{\mathbb R}
\newcommand{\dd}{\,\mathrm d}
\newcommand{\tr}{\operatorname{tr}}
\newcommand{\Ric}{\operatorname{Ric}}
\newcommand{\Div}{\operatorname{div}}
\newcommand{\A}{\mathcal A}
\newcommand{\B}{\mathcal B}
\newcommand{\N}{\mathcal N}
\newcommand{\E}{\mathcal E}
\newcommand{\norm}[1]{\left\lVert#1\right\rVert}
\makeatletter
\def\paragraph{\@startsection{paragraph}{4}%
  \z@\z@{-\fontdimen2\font}{\normalfont\scshape}}
\def\l@section{\@tocline{1}{.15em \@plus\p@}{0pt}{}{\normalfont}}
\def\l@subsection{\@tocline{2}{0pt}{1.5em}{}{\normalfont}}
\makeatother

\title[Critical Dimensions for the Boundary Yamabe Problem, II]{On Critical Dimensions for Compactness\\ in the Boundary Yamabe Problem, II}
\author{Liuwei Gong}
\address[Liuwei Gong]{Department of Mathematics, Chinese University of Hong Kong, Shatin, NT, Hong Kong}
\email{lwgong@math.cuhk.edu.hk}

\author{Seunghyeok Kim}
\address[Seunghyeok Kim]{Department of Mathematics and Research Institute for Natural Sciences, College of Natural Sciences, Hanyang University, 222 Wangsimni-ro Seongdong-gu, Seoul 04763, Republic of Korea}
\email{shkim0401@hanyang.ac.kr}
\email{shkim0401@gmail.com}

\author{Monica Musso}
\address[Monica Musso]{Department of Mathematical Sciences, University of Bath, Bath BA2 7AY, United Kingdom}
\email{mm2683@bath.ac.uk}

\author{Juncheng Wei}
\address[Juncheng Wei]{Department of Mathematics, Chinese University of Hong Kong, Shatin, NT, Hong Kong}
\email{wei@math.cuhk.edu.hk}

\date{\today}
\begin{document}
\begin{abstract}
We determine the sharp compactness ranges for the scalar-flat and minimal-boundary Yamabe problems on smooth compact manifolds of positive conformal type, excluding the conformal round hemisphere.
For zero scalar curvature and positive constant boundary mean curvature, compactness holds through dimension $14$ for general boundary and dimension $21$ for umbilic boundary.
For positive scalar curvature and zero boundary mean curvature, the corresponding upper dimensions are $14$ and $20$.
Together with the noncompactness examples in Part~I, these results identify the transition dimensions in both boundary classes.
For positive scalar curvature, we also prove compactness through dimension eight for every fixed real boundary mean curvature, and obtain higher-dimensional ranges when this curvature is near zero or sufficiently large and positive.
The proof combines scalar-correction estimates for the full conformal Fermi metric expansion with a geometric formula expressing the logarithmic coefficient of the corrected energy as a negative sum of squares.
\end{abstract}

\maketitle
\tableofcontents

\section{Introduction}\label{sec:introduction}
Let $(M^N,g)$ be a smooth compact connected Riemannian manifold with nonempty boundary, $N\ge3$.
The boundary Yamabe problem asks for a metric in the conformal class of $g$ with constant scalar curvature and constant boundary mean curvature.
Its variational formulation builds on Cherrier~\cite{Cherrier} and Escobar~\cite{EscobarAnnals,EscobarJDG}.
We study compactness of the set of positive solutions for a fixed background metric and fixed curvature constants: every sequence of solutions has a subsequence converging in $C^k(M)$, for every fixed $k$, to a positive solution.

We use the outward unit normal $\nu$, the Laplacian $\Delta_g=\operatorname{div}_g\nabla_g$, and boundary mean curvature $H_g$.
Set $c_N:=(N-2)/(4(N-1))$ and define
\[
 L_g:=-\Delta_g+c_NR_g,\qquad B_g:=\partial_\nu+\frac{N-2}{2}H_g.
\]
For constants $(\sigma,\kappa)$, we use the coefficients of Part~I~\cite[(1.1)]{PartI}:
\begin{equation}\label{eq:PDE}
 \begin{cases}
 L_gu=\dfrac{N(N-2)}4\sigma u^{(N+2)/(N-2)} &\text{on }M,\\
 u>0 &\text{on }M,\\
 B_gu=\dfrac{N-2}2\kappa u^{N/(N-2)} &\text{on }\partial M.
 \end{cases}
\end{equation}
Thus $\hat g=u^{4/(N-2)}g$ has $R_{\hat g}=\sigma N(N-1)$ and $H_{\hat g}=\kappa$.
We consider $(\sigma,\kappa)=(0,2)$ and $(\sigma,\kappa)=(1,\kappa)$ with $\kappa\in\R$ fixed, the Type-II and Type-I equations of Part~I, respectively.
For any fixed $\kappa>0$, multiplication by $(\kappa/2)^{(N-2)/2}$ identifies the scalar-flat solution set with that for mean curvature $2$.
Positive conformal type means that the first eigenvalue of $L_g$ with $B_gv=0$ is positive.
The boundary is umbilic when its trace-free second fundamental form vanishes identically.

For the scalar-flat problem, Escobar~\cite{EscobarAnnals} established existence under several geometric hypotheses by minimizing the quotient associated with the Sobolev trace inequality.
Marques~\cite{Marques2005,Marques2007} and Almaraz~\cite{AlmarazExistence} treated further cases using refined test functions and conformal Fermi coordinates.
S.-Y.~S.~Chen~\cite{ChenScalar} extended the higher-dimensional analysis and proved the remaining minimization results conditional on the corresponding positive mass theorem.
Mayer--Ndiaye~\cite[Theorem~1.1]{MayerNdiayeScalar} obtained existence in the remaining positive-type umbilic cases by the barycenter method.
Together with the preceding results, this yields existence for manifolds with finite Sobolev quotient, as stated in \cite[Theorem~1.2]{MayerNdiayeScalar}.

For constant scalar curvature and minimal boundary, Escobar~\cite{EscobarJDG} proved existence in the low-dimensional and nonumbilic cases and in several umbilic cases.
Brendle--Chen~\cite{BrendleChen} treated the remaining higher-dimensional umbilic geometry and established existence of a minimizer whenever the associated Green metric has positive mass.
Mayer--Ndiaye~\cite[Theorem~1.1]{MayerNdiayeMinimal} supplied a barycenter argument for existence in those cases; their combined existence conclusion is \cite[Theorem~1.2]{MayerNdiayeMinimal}.

The problem with nonzero constant scalar and boundary mean curvatures was studied by Escobar~\cite{Escobar1996} and Han--Li~\cite{HanLi,HanLi2000}.
The Han--Li problem asks for a conformal metric of positive constant scalar curvature and any fixed real boundary mean curvature.
Chen--Sun~\cite{ChenSun} developed a variational existence theory for these two curvature conditions, and Chen--Ruan--Sun~\cite[Theorem~1.1]{ChenRuanSun} proved the Han--Li assertion in dimensions $3\le N\le7$ and under several higher-dimensional geometric hypotheses.

For closed manifolds, the compactness theorem of Khuri--Marques--Schoen~\cite{KMS} combines local Pohozaev inequalities with positive mass.
In the scalar-flat boundary problem, Felli--Ould Ahmedou~\cite{FelliAhmedou,FelliAhmedou2005} proved compactness for locally conformally flat manifolds with umbilic boundary and for three-dimensional manifolds with umbilic boundary, respectively.
Almaraz~\cite[Theorem~1.2]{AlmarazCompactness} obtained compactness for $N\ge7$ when the trace-free second fundamental form is nowhere zero on the boundary.
Almaraz--de Queiroz--Wang~\cite{AQW} established three-dimensional compactness for general boundary.
Kim--Musso--Wei~\cite[Theorem~1.1]{KMW} proved compactness for $N=4,5$ without a boundary-geometry restriction and for $N=6$ when the trace-free second fundamental form is nowhere zero on the boundary.
Ghimenti--Micheletti~\cite{GhimentiMicheletti2020,GhimentiMicheletti2021} proved compactness for umbilic boundary under Weyl-curvature nonvanishing assumptions.
For positive scalar curvature, Disconzi--Khuri~\cite[Section~3]{DK} developed the ambient conformal-normal-coordinate analysis of boundary geometry, while Almaraz--Wang~\cite{AW} proved three-dimensional compactness for arbitrary constant boundary mean curvature.
Almaraz's scalar-flat noncompactness examples~\cite{AlmarazBlowup} and the constructions in Part~I~\cite{PartI} show why the dimension ranges matter.

The positive mass argument originates in the minimal-hypersurface proof of Schoen--Yau~\cite{SchoenYau} and the spinorial proof of Witten~\cite{Witten}.
The recent theorem of Brendle--Wang~\cite{BW} removes the dimension and spin restrictions from the smooth boundaryless result.
For nonnegative scalar curvature, the required decay and rigidity conclusions follow from the smooth case of Khuri--Wang--Wang~\cite[Theorem~1.1]{KWW}.
The doubling and smoothing arguments of Almaraz--Barbosa--de Lima~\cite[Section~4]{ABdL} and Miao~\cite[Theorem~1]{Miao} then yield the boundary version used here.

\begin{masstheorem}[Positive mass and rigidity; informal statement]\label{in:PMT}
In every dimension $N\ge3$, a complete connected scalar-flat manifold with minimal boundary and asymptotically Euclidean or half-space ends has nonnegative mass at each end, under the decay assumptions specified in Section~\ref{sec:prelim-mass}.
If one end has zero mass, the manifold is isometric to Euclidean space or Euclidean half-space, according to the type of that end.
\end{masstheorem}
The precise statement and references are given in Theorem~\ref{thm:positive-mass-full}.
This establishes the positive-mass and rigidity statement assumed in the announcement of our compactness results in Part~I~\cite[Remark~1.1]{PartI}.
For the Green metrics used below, Lemma~\ref{lem:auxmass} and the proof of Theorem~\ref{thm:main} verify the required decay and completeness; scalar- and mean-curvature integrability follow from their vanishing.

\begin{maintheorem}\label{thm:main}
Assume that $(M,g)$ has positive conformal type and is not conformally equivalent to the round hemisphere.
Positive solutions of \eqref{eq:PDE} satisfy
\[
 C^{-1}\le u\le C,\qquad \|u\|_{C^{2,\alpha}(M)}\le C,
 \qquad 0<\alpha<1,
\]
in each of the following cases:
\begin{center}
\begin{tabular}{@{}lcc@{}}\toprule
Curvature pair & General boundary & Umbilic boundary\\\midrule
$(0,2)$ & $3\le N\le14$ & $3\le N\le21$\\
$(1,0)$ & $3\le N\le14$ & $3\le N\le20$\\
$(1,\kappa)$, any fixed $\kappa\in\R$ & $3\le N\le8$ & $3\le N\le8$\\\bottomrule
\end{tabular}
\end{center}
Here $C$ depends on $(M,g)$, the curvature pair, and $\alpha$.
In the last row it is uniform for $\kappa$ in a compact interval.
The solution set is compact in $C^k(M)$ for every fixed $k$.
\end{maintheorem}

\begin{maincorollary}\label{thm:continuation}
Under the assumptions of Theorem~\ref{thm:main}, the same compactness conclusion holds for $(\sigma,\kappa)=(1,\kappa)$ in the following additional ranges.
There are positive constants $\delta_N$ and $K_N$, depending also on the boundary class, such that
\[
\begin{array}{c|cc}
 &\text{General boundary}&\text{Umbilic boundary}\\\hline
 |\kappa|<\delta_N&3\le N\le14&3\le N\le20\\
 \kappa>K_N&3\le N\le14&3\le N\le21.
\end{array}
\]
The estimates are uniform on compact subsets of the indicated parameter intervals.
\end{maincorollary}

Theorem~\ref{thm:main} and Part~I~\cite[Theorem~1]{PartI} give complementary dimension ranges for general and umbilic boundary at the scalar-flat and minimal-boundary endpoints.
Part~I~\cite[Theorem~1]{PartI} also gives noncompact examples in every dimension $N\ge9$ for sufficiently negative mean curvature.
Thus dimension eight is the largest range in which compactness holds for every fixed real mean curvature under the hypotheses of Theorem~\ref{thm:main}.
The parameter intervals in Corollary~\ref{thm:continuation} follow from the continuity argument in Proposition~\ref{prop:local-continuation}.

Cordero-Misteli~\cite[Theorem~1.1]{CM} proves compactness for positive scalar curvature and minimal boundary in dimensions $3\le N\le14$ for general boundary and $3\le N\le20$ for umbilic boundary, together with noncompactness in the complementary dimensions.
These ranges agree with Theorem~\ref{thm:main} and Part~I~\cite[Theorem~1]{PartI}.
His theorem also establishes compactness through dimension $18$ and noncompactness from dimension $19$ for locally conformally flat metrics.
The present paper additionally treats zero scalar curvature with positive mean curvature, and positive scalar curvature with nonzero constant mean curvature, in the ranges of Theorem~\ref{thm:main} and Corollary~\ref{thm:continuation}.
We obtained our compactness results and the noncompactness constructions of Part~I independently of~\cite{CM}.

For positive scalar curvature and minimal boundary, the noncompactness constructions use different polynomial metric profiles and coordinate choices.
Part~I~\cite[Lemma~4.5]{PartI} uses tangential trace-free exponential metrics in Gaussian normal coordinates, with exactly minimal background boundary.
Cordero-Misteli~\cite[Section~4]{CM} uses ambient conformal normal coordinates, represents the boundary as a graph over its tangent hyperplane, and then flattens this graph vertically.
At the first nonumbilic dimension $N=15$, the profile in Part~I~\cite[(4.13)]{PartI} has degree six, whereas the corresponding degree in \cite[(9.50)]{CM} is five.
At the first umbilic dimension $N=21$, these degrees are nine in Part~I~\cite[(4.14)]{PartI} and eight in \cite[(9.50)]{CM}.
Part~I uses the same respective degrees throughout $N\ge15$ and $N\ge21$; Cordero-Misteli also uses lower degrees in some subsequent dimension ranges, as listed in \cite[(9.50)]{CM}.
The degree in \cite[(9.50)]{CM} is that of the first-order metric tensor after flattening; the polynomial defining the boundary graph may have degree one higher because its first derivatives enter that tensor.
For compactness, we impose the additional boundary geodesic-coordinate condition in \eqref{eq:fermi-normalization} and work on the full conformal Fermi tensor space \eqref{crit:eq:fermi}.

Both compactness calculations retain the scalar correction and the cross terms between distinct polynomial components of the same tangential harmonic type.
Here $h$ denotes a polynomial metric perturbation, represented by the model metric $g=\exp(h)$.
Within a fixed tangential harmonic type, write $h=\sum_i a_i h_i$, where the $h_i$ form a basis of homogeneous polynomial tensors satisfying \eqref{crit:eq:fermi}, and set $a=(a_i)$.
If $E=(E_{ij})$ is the corresponding symmetric matrix of the corrected quadratic energy, then $a^{\mathsf T}Ea=\sum_i a_i^2E_{ii}+2\sum_{i<j}a_i a_jE_{ij}$.
The terms with $i<j$ pair distinct basis tensors, which may have different total degrees.
See \eqref{eq:quadratic-action}, \eqref{lc:eq:scale}, and \cite[Section~A.5]{CM}.

For even $N$, the homogeneous degree $d=(N-2)/2$ is called \emph{critical}: its quadratic radial integral diverges logarithmically.
It equals the last retained degree $d_N$ in \eqref{eq:degrees}; for odd $N$, $d_N$ remains in the strictly integrable range.
At the critical degree, Cordero-Misteli~\cite[Lemmas~A.23--A.24]{CM} determines the leading scalar correction on the hemisphere and inserts its angular energy into the full logarithmic coefficient matrices.
For a homogeneous polynomial metric tensor $h$ of this degree, Theorem~\ref{crit:thm:critical} uses Kelvin inversion and a boundary-preserving Bianchi gauge to obtain \eqref{crit:eq:residue}, a negative multiple of $\int_{S^{N-1}_+}(|\nabla_ST_dh|^2+d^2|T_dh|^2)$.
If this integral vanishes, then $T_dh=0$; Lemma~\ref{lc:lem:kernel} shows that the original homogeneous polynomial tensor $h$ is identically zero, provided it satisfies the conformal Fermi conditions \eqref{crit:eq:fermi}.
The proof includes the planar boundary terms and estimates the spherical boundary and cutoff terms.

\medskip
\paragraph{The local quadratic form.}
Part~I constructs noncompact examples by choosing polynomial metric perturbations for which the corrected quadratic energy has a suitable local extremum.
Compactness requires a sign on the entire space of polynomial tensors allowed by conformal Fermi coordinates.
This space contains scalar-, vector-, and tensor-derived polynomials multiplied by powers $t^j$ of the normal coordinate.
The total degree and boundary conditions are specified in \eqref{eq:degrees} and \eqref{crit:eq:fermi}: at $t=0$ the boundary radial identity must hold, and umbilicity removes the coefficient of $t$ but not the higher odd powers.
Within each tangential harmonic type, the quadratic form retains the pairings between all these components.
The scalar-derived terms generally have nonzero divergence, so their correction contributes to the quadratic energy.

We evaluate that contribution by the decomposition into a particular solution and a boundary correction used in Part~I~\cite[Section~5.1]{PartI}.
A rational multiple of the bubble solves the interior linearized equation exactly and leaves a Robin or Neumann residual on the boundary.
The solution with this residual determines the remaining inverse pairing through Green's identity.
For the scalar-flat equation, the resulting matrix entries are rational after normalization.
For minimal boundary, they are rational polynomials in $1/\pi$, evaluated by exact telescoping of the boundary spectral series.
The finite positivity checks are implemented in the self-contained program \texttt{verify\_compactness.py}, included with the arXiv submission and described in the proof of Theorem~\ref{thm:ordinary-endpoints}.
For arbitrary fixed mean curvature through dimension eight, the homogeneous degrees satisfying $2d<N-2$ are at most two; direct bounds on the boundary correction prove the required sign.

\medskip
\paragraph{The logarithmic term and comparison with Green functions.}
The Taylor expansion \eqref{eq:degrees} must be controlled through $d_N=\lfloor(N-2)/2\rfloor$ to obtain the vanishing in Corollary~\ref{cor:jets} and the Green-metric decay in Lemma~\ref{lem:auxmass}.
In even dimensions, $d_N$ is the critical degree, and its quadratic energy has radial factor $\int_1^R r^{2d_N+1-N}\dd r=\log R$.
Thus the convergent whole-half-space integrals for $2\deg h<N-2$ in Part~I~\cite[Proposition~3.1]{PartI} do not cover this last coefficient.
Theorem~\ref{crit:thm:critical} supplies this coefficient; together with the matrices for $2d<N-2$, it controls the negative scale derivative in \eqref{eq:Qcoercive}, uniformly for coefficient vectors varying along the blow-up sequence.

To apply this inequality to solutions, we use the localized metric \eqref{eq:localized-metric}, obtained by cutting off $h_p$ from \eqref{eq:fermi-normalization}, including its Taylor remainder; this metric agrees with the original metric near $p$ and is Euclidean outside a small half-ball.
Theorem~\ref{thm:tame} constructs the correction $w(k)$ in \eqref{eq:nonlinear-projected}: $U+w(k)$ satisfies the weak equation for every test function in $Y$, the complement of the Jacobi kernel defined after \eqref{eq:dual}.
Its differentiated remainder estimate yields the scale inequality in Proposition~\ref{prop:nonlinear}.
Theorem~\ref{thm:matching} identifies the scale derivative with the regular Green-function difference.
Corollary~\ref{cor:jets} proves the joint conformal Fermi-coordinate formulation of Weyl vanishing and vanishing of the trace-free second fundamental form at a boundary blow-up point: $J_d(p)=0$ for $1\le d\le d_N$, or equivalently $g_p=\delta+O(r^{d_N+1})$.
In particular, $\nabla^jW_g(p)=0$ for $0\le j\le d_N-2$ when $N\ge6$, and $\bar\nabla^j\mathring{\mathrm{II}}_g(p)=0$ for $0\le j\le d_N-1$ when $N\ge4$, where $\bar\nabla$ is the boundary covariant derivative.
The corresponding interior Weyl-vanishing result is recalled in Proposition~\ref{prop:interiorlocal}.
Lemma~\ref{lem:auxmass} then bounds the localized Green metric's mass by a quantity tending to zero with the localization radius.
Propositions~\ref{prop:simple} and~\ref{prop:separate} then show that every blow-up point is isolated simple and that distinct selected maxima remain a fixed positive distance apart.
Theorem~\ref{thm:positive-mass-full} excludes blow-up in the final proof of Theorem~\ref{thm:main}.

\medskip
\paragraph{Organization.}
Section~\ref{sec:preliminaries} collects the standard geometric and analytic results, including the positive mass theorem, and fixes the notation used below.
Sections~\ref{sec:ordinary-endpoints} and~\ref{sec:ordinary} prove the quadratic coefficient inequalities for the scalar-flat and minimal-boundary problems and for arbitrary constant mean curvature, respectively.
Section~\ref{sec:critical} computes the logarithmic coefficient and proves coercivity of the negative scale derivative.
Sections~\ref{sec:nonlinear-family} and~\ref{sec:local-comparison} establish the nonlinear estimates, the comparison with Green functions, and the mass estimate for the localized metric.
Section~\ref{sec:compactness} excludes accumulation of blow-up points and proves the compactness results.

\medskip
\noindent\textsc{Statement on the Use of AI.}
ChatGPT and Codex (OpenAI) were used to assist with the organization and wording of the manuscript, notation checks, and preparation and checking of the submission files.
The analytic arguments and computer-assisted calculations are documented in the text and verification materials.
The authors are responsible for the mathematical content and references.

\medskip
\noindent\textsc{Acknowledgement.}
L. Gong and J. Wei gratefully acknowledge support from the Research Grants Council of Hong Kong (RGC) through the project: \emph{On Fujita equation in the critical or supercritical regime}.
S. Kim was supported by the Basic Science Research Program through the National Research Foundation of Korea (NRF), funded by the Ministry of Science and ICT (RS2025-00558417), and by Samsung Science and Technology Foundation under Project Number SSTF-BA2601-01.

\section{Preliminaries}\label{sec:preliminaries}
We write $y=(x,t)\in\R^{N-1}\times[0,\infty)$, with inward normal coordinate $t$ and $m=N-1$.
Indices $i,j$ range from $1$ to $N$, and $a,b$ are tangential; repeated indices are summed.
The outward normal on the half-space is $-\partial_t$.
Polynomial degree means total degree in $(x,t)$.
The spaces of homogeneous coefficients carry fixed Euclidean norms.

For a symmetric matrix field $h$ in these coordinates, $\exp(h)$ denotes the metric whose coefficient matrix is the matrix exponential of $h$.
All metric exponentials below use this convention.
Choose a fixed smooth radial cutoff $\chi$ equal to one on $[0,1]$ and zero on $[2,\infty)$, and set $\chi_R(y):=\chi(|y|/R)$.

\subsection{Bubbles and the linearized equation}\label{sec:source}
In the energy and linearized equations, we use $p:=(N+2)/(N-2)$ and $q:=N/(N-2)$.
For the curvature pairs in Section~\ref{sec:introduction}, define the bubbles as in Part~I~\cite[Section~2.2]{PartI}:
\begin{equation}\label{eq:bubble}
 U(x,t):=
 \begin{cases}
 \bigl(|x|^2+(t+1)^2\bigr)^{-(N-2)/2}, &(\sigma,\kappa)=(0,2),\\
 \displaystyle\left(\frac{2}{1+|x|^2+(t+\kappa)^2}\right)^{(N-2)/2}, &\sigma=1.
 \end{cases}
\end{equation}
For the two cases $\sigma\in\{0,1\}$ in \eqref{eq:bubble}, $U(y)=2^{\sigma(N-2)/2}|y|^{2-N}+O(|y|^{1-N})$ as $|y|\to\infty$.
The model equation is
\[
 \begin{cases}
 -\Delta U=\dfrac{N(N-2)}4\sigma U^p &\text{in }\R^N_+,\\
 \partial_\nu U=\dfrac{N-2}2\kappa U^q &\text{on }\partial\R^N_+.
 \end{cases}
\]
The translated and dilated bubbles are
\[
 U_{\xi,\varepsilon}(x,t):=\varepsilon^{-(N-2)/2}U\left(\frac{x-\xi}{\varepsilon},\frac{t}{\varepsilon}\right).
\]
For $(\sigma,\kappa)=(0,2)$, the bubble is harmonic and its Jacobi boundary operator is $\partial_\nu-NU^{2/(N-2)}$.
For $(1,0)$, the Jacobi boundary condition is Neumann.
Every scale derivative keeps the curvature pair fixed.

We use the energy functional $\A_g$, related by $\A_g=c_N^{-1}E_g$ to the energy $E_g$ in Part~I~\cite[(2.1)]{PartI} for the same curvature pair.
This normalization keeps the scalar-curvature coefficient in the geometric expansions equal to one.
Explicitly,
\begin{equation}
 \begin{aligned}
 \A_g(v):={}&\int_{\R^N_+}(c_N^{-1}|\nabla v|_g^2+R_gv^2)\dd v_g
 +2(N-1)\int_{\partial\R^N_+} H_gv^2\dd\sigma_g\\
 &-\sigma(N-1)(N-2)\int_{\R^N_+}(v_+)^{p+1}\dd v_g
 -2(N-2)\kappa\int_{\partial\R^N_+}(v_+)^{q+1}\dd\sigma_g.
 \end{aligned}
\end{equation}
Its scalar derivative is $2/c_N$ times the weak equation associated with \eqref{eq:PDE}.
At $U$, the scalar Hessian multiplied by $c_N/2$ is
\begin{equation}
 \B(v,w):=\int_{\R^N_+}\nabla v\cdot\nabla w
 -\frac{\sigma N(N+2)}4\int_{\R^N_+} U^{4/(N-2)}vw
 -\frac{N\kappa}2\int_{\partial\R^N_+} U^{2/(N-2)}vw.
\end{equation}

Set $\E:=\dot H^1(\R^N_+)$, with norm $\|v\|_{\E}:=\|\nabla v\|_{L^2(\R^N_+)}$.
The classification of positive half-space solutions is due to Li--Zhu~\cite{LiZhu}.
The following energy-space formulation of Jacobi nondegeneracy is recorded in Part~I~\cite[Section~2.2]{PartI}, using \cite[Lemmas~2.1--2.2]{AW} and \cite[Section~2.2]{KMW}.
\begin{lemma}[Classification and Jacobi kernel]\label{lem:model-kernel}
Every positive solution of the model half-space equation is one of the bubbles $U_{\xi,\varepsilon}$.
The kernel of $\B$ on $\E$ is spanned by
\[
 Z_0:=\frac{N-2}{2}U+y\cdot\nabla U,\qquad Z_a:=\partial_aU\quad(1\le a<N).
\]
\end{lemma}
Choose a nonnegative $\zeta\in C_c^\infty(\overline{\R^N_+})$ that is positive on a nonempty open set, and define
\[
 \mathsf G_{jk}:=\int_{\R^N_+}\zeta Z_jZ_k,\qquad
 \theta_j:=\zeta\sum_{\ell=0}^{N-1}(\mathsf G^{-1})_{j\ell}Z_\ell.
\]
The Gram matrix is positive definite, and
\begin{equation}\label{eq:dual}
 \int_{\R^N_+}\theta_jZ_k=\delta_{jk},\qquad 0\le j,k<N.
\end{equation}
We use the closed complement $Y:=\{w\in\E:\int_{\R^N_+}\theta_jw=0\text{ for }0\le j<N\}$.
By the projected inverse construction in \cite[Section~2.2]{PartI}, $\B:Y\to Y^*$ is invertible, with inverse norm locally uniform in $\kappa$.
The negative scalar direction is retained; only the Jacobi kernel is removed.
Sobolev, trace, and Hardy inequalities apply on $\E$; in particular,
\begin{equation}
 \int_{\R^N_+}\frac{|w|^2}{(1+|y|)^2}+\int_{\partial\R^N_+}\frac{|w|^2}{1+|x|}\le C\|w\|_{\E}^2.
\end{equation}
See \cite{EscobarTrace} and \cite[Appendix~C.1]{PartI} for the trace and weighted forms used here.

\medskip
\paragraph{The minimal-boundary model on the hemisphere.}
For $(\sigma,\kappa)=(1,0)$, stereographic compactification by $U^{4/(N-2)}\delta$ is the unit round hemisphere.
If $w=Uv$, conformal covariance gives
\[
 \int_{\R^N_+}|\nabla w|^2
 =\int_{S^N_+}\left(|\nabla v|^2+\frac{N(N-2)}4v^2\right),\qquad
 \B(w,w)=\int_{S^N_+}(|\nabla v|^2-Nv^2).
\]
The boundary operator is Neumann.
Constants give the negative direction.
The degree-one Neumann harmonics correspond to the $N$ Jacobi fields.
All remaining Neumann harmonics have a positive spectral gap.
Thus the constrained inverse exists on the complement defined by \eqref{eq:dual}.

The separated boundary response is given in \cite[Lemma~5.4]{PartI}:
\begin{lemma}\label{min:responselemma}
Let $Y_L$ be a spherical harmonic of degree $L$ on the equator $S^{N-1}$.
The regular solution of $(-\Delta_{S^N}-N)v=0$ in the hemisphere with normal derivative $Y_L$ has boundary value $r_LY_L$, where
\begin{equation}
 r_L=\frac{\Gamma((L-1)/2)\Gamma((L+N)/2)}
 {2\Gamma(L/2)\Gamma((L+N+1)/2)},\qquad L\ge2.
\end{equation}
Degree $L=1$ is excluded by the Jacobi compatibility condition.
At $L=0$ the response is zero, not a projection of the negative mode.
\end{lemma}

\medskip
\paragraph{Radial integrals.}
After angular integration in the tangential variables, the scalar-flat and minimal-boundary coefficients reduce to the following integrals of monomials in $r=|x|$ and $t$ against the corresponding bubble weights.
Euler's beta integral \cite[(5.12.1)]{DLMF} yields these formulas for $m=N-1$, $R,j_1\ge0$, and $2\nu>m+R+j_1+1$; see also \cite[Lemma~5.3]{PartI}:
\begin{equation}\label{lc:eq:moment}
 \begin{aligned}M_{R,j_1}^{\nu}&=\int_0^\infty\!\int_0^\infty
 r^{m-1+R}t^{j_1}\big(r^2+(t+1)^2\big)^{-\nu}\dd r\dd t
\\&=\tfrac12 B(a,\nu-a)B(j_1+1,2\nu-2a-j_1-1),\quad a=(m+R)/2.\end{aligned}
\end{equation}
\begin{equation}\label{min:moment}
\begin{aligned}
 M^I_{R,j_1,\nu}
 &=\int_0^\infty\!\int_0^\infty
 r^{m-1+R}t^{j_1}(1+r^2+t^2)^{-\nu}\,dr\,dt\\
 &=\frac{\Gamma((m+R)/2)\Gamma((j_1+1)/2)
 \Gamma(\nu-(m+R+j_1+1)/2)}{4\Gamma(\nu)}.
\end{aligned}
\end{equation}
At $t=0$, the corresponding boundary integral for either model is $M_{R,0}^{\nu,\partial}:=\int_0^\infty r^{m-1+R}(1+r^2)^{-\nu}\dd r=\tfrac12 B((m+R)/2,\nu-(m+R)/2)$, for $2\nu>m+R$.

\subsection{Conformal Fermi coordinates}\label{sec:quadratic}
Fix an integer $K>N+4$.
Let $\Psi_p$ denote the conformal Fermi coordinate map from a half-ball to a neighborhood of a boundary center $p$.
The conformally normalized metric can be chosen so that
\begin{equation}\label{eq:fermi-normalization}
 g_p=\exp(h_p),\qquad (h_p)_{Ni}=0,\qquad
 \tr h_p=O(|y|^K),\qquad x_a(h_p)_{ab}(x,0)=0.
\end{equation}
The conformal Fermi normalization is due to Marques~\cite[Proposition~3.1]{Marques2005}; see \cite[Proposition~2.2]{AlmarazExistence} and the formulation in \cite[Lemma~2.2]{KMW}.
The last identity is the boundary Gauss lemma, applied also to the matrix logarithm.
It is the additional boundary geodesic-coordinate condition relative to the Gaussian-coordinate ansatz of Part~I~\cite[Section~4.1]{PartI}, which does not require the boundary coordinates to be geodesic.
Normalize the conformal factor to be one with zero tangential gradient at $p$; its normal derivative normalizes mean curvature.
The finite-order normalization and the boundary exponential maps can be chosen smoothly in the center $p$.
On a compact boundary chart, these choices have common ellipticity and finite derivative bounds.

Put
\begin{equation}\label{eq:degrees}
 d_N=\left\lfloor\frac{N-2}{2}\right\rfloor,
 \qquad d_N^{\rm ord}=\left\lfloor\frac{N-3}{2}\right\rfloor.
\end{equation}
The ordinary degrees are $1\le d\le d_N^{\rm ord}$, equivalently $2d<N-2$.
For even $N$, the remaining degree $d_N=(N-2)/2$ is critical; for odd $N$, $d_N=d_N^{\rm ord}$.
Let $\mathscr J_d$ be the space of symmetric degree-$d$ homogeneous polynomial tensors satisfying
\begin{equation}\label{crit:eq:fermi}
 h_{Ni}=0,\qquad\tr h=0,\qquad x_a h_{ab}(x,0)=0.
\end{equation}
Its umbilic subspace $\mathscr J_d^{\rm u}$ also satisfies $\partial_t h(x,0)=0$.
This removes the normal power $t^1$, not odd powers $t^3,t^5,\ldots$.
The jets of \eqref{eq:fermi-normalization} satisfy
\[
 h_p=\sum_{d=1}^{d_N}J_d(p)+q_p,\qquad
 J_d(p)\in\mathscr J_d,\qquad q_p=O(|y|^{d_N+1}),
\]
with differentiated remainders; in the umbilic case $J_d(p)\in\mathscr J_d^{\rm u}$.
Indeed, umbilicity and the high-order trace condition give the vanishing of the normal-linear coefficient at every retained order.
At degrees one and two,
\[
 J_1=tA,\qquad J_2=H^W(x)+tC(x)+t^2B.
\]
The tangential degree-two normal-coordinate tensor is the boundary curvature tensor; trace freeness removes its Ricci contraction, leaving the Weyl tensor.
The remaining terms are the normal Taylor coefficients.
Fix coefficient norms on every $\mathscr J_d$.

For families of background metrics on a fixed coordinate half-ball, all local constants below are uniform under common ellipticity and $C^{K+2,\alpha}$ bounds, with $\kappa$ in a fixed compact interval.
These assumptions are imposed whenever varying or rescaled metrics are considered.

\subsection{First and second metric variations}\label{sec:action}
The first-variation formula \cite[Proposition~2.1]{ABdL}, applied to the compactified round cap or ball, gives the following stationarity statement; its tangential trace-free case is \cite[Proposition~3.2]{PartI}.
\begin{lemma}[First variation with respect to the metric]\label{lem:stationarity}
For every smooth compactly supported metric variation $k$, $D_g\A_\delta(U)[k]=0$.
Its mixed scalar--metric derivative annihilates the Jacobi kernel.
For $\hat g:=U^{4/(N-2)}g$, conformal covariance and integration by parts give
\[
 \A_g(U)={}\int_{\R^N_+} \bigl(R_{\hat g}-\sigma(N-1)(N-2)\bigr)\dd v_{\hat g} +2\int_{\partial\R^N_+} \bigl((N-1)H_{\hat g}-(N-2)\kappa\bigr)\dd\sigma_{\hat g}.
\]
\end{lemma}

The source and quadratic expansion are those of \cite[Propositions~3.1--3.2]{PartI} and \cite[Lemma~3.3]{PartI}.
\begin{proposition}[Metric expansion and scalar correction]
Let $h$ be smooth and compactly supported, with $h_{Ni}=0$ and $\tr h=0$.
The path $\exp(s h)$ has determinant one and Gaussian normal block form, so
\[
 H_{\exp(s h)}
 =-\frac{1}{2(N-1)}\partial_t\log\det(\exp(s h^T))=0
 \quad\text{on }t=0.
\]
The scalar-curvature coefficient of $s^2$ and the first scalar source are
\begin{equation}\label{lc:eq:rawsource}
 \begin{aligned}
 R^{(2)}(h)&=-\tfrac14|\partial h|^2+\tfrac12|\Div h|^2
              -\partial_i(h_{ij}(\Div h)_j),\\
 S_h&=h_{ij}U_{ij}+(\partial_i h_{ij})U_j
                  +c_N(\partial_i\partial_jh_{ij})U.
 \end{aligned}
\end{equation}
The scalar-curvature expansion is $R_{\exp(s h)}=s\partial_i\partial_jh_{ij} +s^2R^{(2)}(h)+O(s^3)$.
The quadratic energy after this correction is
\begin{equation}\label{eq:quadratic-action}
 \begin{aligned}
  F(h)&=\mathcal Q(h)+c_N^{-1}\int_{\R^N_+} S_hv_h,\qquad
 \B(v_h,\psi)=-\int_{\R^N_+} S_h\psi,\\
 \mathcal Q(h)&=\frac{1}{2c_N}\int_{\R^N_+}|h\nabla U|^2+\int_{\R^N_+} R^{(2)}(h)U^2.
 \end{aligned}
\end{equation}
The corrector belongs to $Y$, and its equation holds for all $\psi\in\E$.
The coefficient $F$ is independent of the Jacobi complement and of the second derivative of the metric path.
For the same bubble and curvature pair, $F(h)=c_N^{-1}F_h(0,1)$, where $F_h$ is the corrected energy in Part~I~\cite[Proposition~3.1]{PartI}.
This identity applies to compactly supported tensors and to polynomial tensors satisfying $2\deg h<N-2$.

Polarizing the raw coefficient gives
\begin{equation}\label{lc:eq:rawpolar}
 \begin{aligned}
 \mathcal Q_{ij}={}&\frac{1}{2c_N}\int_{\R^N_+}(h_i\nabla U)\cdot(h_j\nabla U)
 -\frac14\int_{\R^N_+}\partial h_i\cdot\partial h_jU^2\\
 &+\frac12\int_{\R^N_+}(\Div h_i)\cdot(\Div h_j)U^2
 +\frac12\int_{\R^N_+}[h_i\Div h_j+h_j\Div h_i]\cdot\nabla(U^2).
 \end{aligned}
\end{equation}
\end{proposition}

For every fixed curvature pair considered here, define
\begin{equation}\label{eq:Nnorm}
 \N(k)^2:=\int_{\R^N_+} U^2\left(|\partial k|^2+\frac{|k|^2}{(1+|y|)^2}\right).
\end{equation}
The source estimate follows from the Sobolev--Hardy inequalities and the projected inverse, as in \cite[Lemma~3.3]{PartI} and \cite[Lemma~C.3]{PartI}; the same estimate applies to the raw pairing \eqref{lc:eq:rawpolar}.
\begin{lemma}[Bound for the corrected quadratic form]
For tangential trace-free compactly supported tensors $k,\ell$, the form \eqref{eq:quadratic-action} satisfies
\begin{equation}\label{eq:quadratic-bounded}
 | F(k,\ell)|\le C\N(k)\N(\ell).
\end{equation}
The same bound holds for polynomial tensors with finite $\N$-norm by cutoff approximation.
\end{lemma}

For a smooth domain $\Omega$, define $\mathcal S_\Omega(g):=\int_\Omega R_g\dd v_g+2\int_{\partial\Omega}H_{\tr,g}\dd\sigma_g$, where $H_{\tr,g}:=(N-1)H_g$.
For compactly supported variations on $\R^N_+$, write $\mathcal S:=\mathcal S_{\R^N_+}$.
For the smooth-boundary variation formula \cite[Proposition~2.1]{ABdL}, take a covariant variation $L$ and set $X=L(\nu,\cdot)^T$, with the tangential one-form identified with a vector by the induced metric $\gamma$.
The identities
\begin{align*}
 R'(L)&=-\langle\Ric,L\rangle+
                  \Div(\Div L-\dd\tr L),\\
 2H'_{\tr}(L)&=(\dd\tr L-\Div L)(\nu)
                  -\Div_{\partial}X-\langle\mathrm{II},L^T\rangle
\end{align*}
yield, for a compactly supported variation on the smooth half-space,
\begin{equation}\label{crit:eq:EHfirst}
 \begin{aligned}
 D\mathcal S_g[L]={}&
 -\int_{\R^N_+}\langle\Ric-\tfrac12Rg,L\rangle\dd v_g\\
 &-\int_{\partial\R^N_+}
       \langle\mathrm{II}-H_{\tr}\gamma,L^T\rangle\dd\sigma_g.
 \end{aligned}
\end{equation}

\subsection{The Pohozaev identity}
We use the identity in the normalization of \eqref{eq:PDE}, with outward normal $-\partial_t$ on the flat disk; compare \cite[Proposition~2.4]{AW}.
For a half-ball, write $S_r^+=\partial B_r\cap\{t>0\}$ and $D_r=B_r\cap\{t=0\}$.
If $-\Delta u=\sigma N(N-2)u^p/4+f$ and $\partial_\nu u=\frac{N-2}{2}\kappa u^q+b$ on the flat disk, define
\begin{equation}\label{eq:Pohozaev}
 \begin{aligned}
 \mathcal P(u,r)={}&\int_{S_r^+}
 \left(r(\partial_ru)^2-\frac r2|\nabla u|^2+\frac{N-2}{2} u\partial_ru\right)
 +\frac{\sigma(N-2)^2r}{8}\int_{S_r^+}u^{p+1}\\
 &+\frac{\kappa(N-2)^2r}{4(N-1)}\int_{\partial D_r}u^{q+1}.
 \end{aligned}
\end{equation}
Then, with $X_u=\frac{N-2}{2} u+y\cdot\nabla u$,
\begin{equation}
 \mathcal P(u,r)=-\int_{B_r^+}X_uf-\int_{D_r}X_ub.
\end{equation}
For a harmonic Neumann function $G=a r^{2-N}+b+O(r)$, the quadratic boundary flux therefore satisfies
\begin{equation}
 \lim_{r\downarrow0}\mathcal P_0(G,r)
 =-\frac{(N-2)^2}{2}|S^{N-1}_+|ab,
\end{equation}
where $\mathcal P_0$ denotes the first integral in \eqref{eq:Pohozaev}.

\subsection{Local blow-up estimates}
We use the standard bubble-selection result in \cite[Proposition~3.5]{AW} and \cite[Proposition~3.2]{KMW}.
For the scalar-flat equation, positive conformal type permits a fixed representative with $R_g>0$ and $H_g=0$, obtained from the positive first eigenfunction of $L_g$ with homogeneous $B_g$ condition.
In that representative, $\Delta_gu=c_NR_gu>0$, so every local maximum lies on the boundary.

\begin{proposition}[Local maxima and bubble neighborhoods]\label{prop:basiclocal}
Fix the approximation accuracy and a sufficiently large bubble-neighborhood parameter $L$.
For a sequence of solutions with unbounded supremum, one can choose finitely many local maxima $P_i=\{a_{i,j}\}$, with scales $\varepsilon_{i,j}$, whose $L\varepsilon_{i,j}$ neighborhoods are disjoint and whose rescalings approximate the corresponding whole-space or half-space bubbles.
Their heights are comparable to $\varepsilon_{i,j}^{-(N-2)/2}$, and
\begin{equation}\label{eq:selectionbound}
 u_i(x)\,d_g(x,P_i)^{(N-2)/2}\le C.
\end{equation}
For a boundary-type bubble, its boundary center $p_{i,j}$ satisfies $d_g(a_{i,j},p_{i,j})\le C_\kappa\varepsilon_{i,j}$.
The constants are uniform under the local bounds in Section~\ref{sec:quadratic}; the number of selected points may depend on $i$.
\end{proposition}
The isolated-simple condition uses the following fixed-coordinate convention, as in \cite[Definition~3.1]{KMW}.
\begin{definition}\label{def:fermisimple}
For an isolated boundary-type blow-up sequence, fix its boundary centers, the smoothly chosen conformal Fermi charts, and the associated conformally transformed solutions $u_i$.
The isolation bound is $u_i(z)\le C|z|^{-(N-2)/2}$ on a fixed punctured half-ball; the bubble scale is normalized so that $\varepsilon_i^{(N-2)/2} u_i(\varepsilon_i y)$ tends to $U_{\kappa_i}$.
Define
\[
 \overline u_i(r)=\frac1{|S^{N-1}_+|}
                 \int_{S^{N-1}_+}u_i(r\omega)\,\dd\omega.
\]
The blow-up point is \emph{isolated simple} in these conformal Fermi coordinates if, for some fixed $r_0>0$ and all sufficiently large $i$, $r^{(N-2)/2}\overline u_i(r)$ has exactly one critical point in $(0,r_0)$.
The spheres and measure are Euclidean in this fixed Fermi chart.
The critical-point condition is imposed in this fixed chart; after a change of center or conformal representative, we use the pointwise estimates justified in Section~\ref{sec:comparison}.
\end{definition}
The standard interior estimates near an isolated simple point are \cite[Propositions~2.3--2.5]{KMS}.
The Weyl-vanishing theorem and the local sign restriction are \cite[Theorem~6.1]{KMS} and \cite[Theorem~7.1]{KMS}, respectively:
\begin{proposition}[Weyl vanishing and local sign at interior blow-up points]\label{prop:interiorlocal}
Let $3\le N\le24$, $\sigma>0$ fixed, and let $x_i\to x$ be an isolated simple interior blow-up point for smooth metrics $g_i\to g$ on a fixed open ball.
If a fixed positive multiple of $u_i(x_i)u_i$ converges to $H$ away from $x$, then in the limiting conformal normal coordinates
\begin{equation}
 \liminf_{r\downarrow0}\mathcal P^{\rm int}_0(H,r)\ge0,\qquad
 \mathcal P^{\rm int}_0(H,r)=\int_{|z|=r}
 \left(\frac{N-2}{2} H\partial_rH-\tfrac r2|\nabla H|^2+r(\partial_rH)^2\right).
\end{equation}
For $6\le N\le24$, with $\varepsilon_i$ the interior bubble scale,
\[
 |\nabla^\ell W_{g_i}|^2(x_i)
 \le C\varepsilon_i^{N-6-2\ell}
       |\log\varepsilon_i|^{-\mathbf 1_{\{2\ell=N-6\}}},
 \qquad 0\le\ell\le\left\lfloor\frac{N-6}{2}\right\rfloor.
\]
These Weyl derivatives vanish at $x$ in the limit.
For a flat limiting metric, $H=a r^{2-N}+b+O(r)$, $a>0$, implies $b\le0$, without any rate assumption on metric flattening.
\end{proposition}
The Green expansion \cite[(9.2)]{KMS}, its mean-zero identities \cite[(9.3)]{KMS}, and the local sign above imply the following interior consequence.
The argument is local near the pole and is unaffected by a boundary disjoint from that coordinate neighborhood.
\begin{proposition}\label{prop:interiormass}
Suppose $p\in M^\circ$ is an isolated simple blow-up point on a positive-type manifold in dimension $3\le N\le24$, and its normalized outer limit is $G=G_p+H$, where $G_p$ has singular coefficient one, $H$ is regular near $p$, and $H(p)\ge0$.
Then the scalar-flat Green metric $G_p^{4/(N-2)}g$ has an asymptotically Euclidean end of order strictly greater than $(N-2)/2$ and its unnormalized mass is nonpositive.
Here $\sigma>0$; compact minimal boundary components of the Green metric are allowed.
\end{proposition}
\subsection{Mass of Green metrics}\label{sec:prelim-mass}
We use the following precise form of Theorem~\ref{in:PMT}.
An asymptotically Euclidean or half-space end is diffeomorphic to $\R^N\setminus\overline B_1$ or $\R^N_+\setminus\overline B_1^+$, respectively, and in the corresponding coordinates satisfies
\begin{equation}\label{eq:appendix-AF}
 |g_{ij}-\delta_{ij}|+r|\partial g_{ij}|+r^2|\partial^2g_{ij}|=O(r^{-\tau}),
 \qquad \tau>\frac{N-2}{2},\qquad r=|x|\to\infty.
\end{equation}
This is the decay condition in \cite[(1.1)]{ABdL} and \cite[Definition~2.2]{KWW}.
We allow finitely many ends of either type and finitely many compact boundary components, as in \cite[Remark~1.5]{ABdL}.

For a half-space end, the unnormalized mass in its asymptotic chart is
\begin{equation}\label{eq:appendix-mass}
 m(g):=\lim_{r\to\infty}\left[
 \int_{S_r^+}(\partial_jg_{ij}-\partial_i g_{jj})\frac{x_i}{r}\dd\sigma
 +\int_{\partial S_r^+}g_{aN}\frac{x_a}{r}\dd\ell\right].
\end{equation}
Here $S_r^+:=\{x\in\R^N_+:|x|=r\}$, the measures are Euclidean, $1\le i,j\le N$, and $1\le a<N$.
For a Euclidean end, replace $S_r^+$ by $S_r$ and omit the equatorial integral; see \cite[Definition~1.1]{ABdL} and \cite[(2.6)]{ABdL}.
Under \eqref{eq:appendix-AF} and integrability of the scalar and mean curvatures, these limits are coordinate invariant by Bartnik~\cite{Bartnik} and \cite[Section~3]{ABdL}.
Scalar flatness and minimality below ensure the required integrability.

The boundary positive mass argument is \cite[Section~4]{ABdL}, including the Euclidean ends and compact minimal boundary components in \cite[Remark~1.5]{ABdL}.
Its smooth positive mass input is now supplied in every dimension by \cite[Theorem~1.1]{KWW}, with empty singular set.
The approximation and corner-smoothing steps are \cite[Proposition~4.1]{ABdL} and \cite[Theorem~1]{Miao}; rigidity is obtained for the original zero-mass metric by \cite[Lemmas~4.3--4.4]{ABdL}, not by assuming that the original boundary has a smooth reflection double.
\begin{theorem}[Positive mass and rigidity]\label{thm:positive-mass-full}
Let $(M^N,g)$ be a complete connected smooth scalar-flat manifold, $N\ge3$, with minimal boundary and finitely many asymptotically Euclidean or half-space ends satisfying \eqref{eq:appendix-AF}.
The mass of each end is nonnegative.
If one end has zero mass, $M$ has only that end and is isometric to Euclidean space or Euclidean half-space, according to its type.
In particular, a compact boundary component excludes zero mass.
\end{theorem}
The conformal change of the flux in \eqref{eq:appendix-mass} gives the standard comparison formula below, in the same unnormalized convention.
\begin{lemma}\label{lem:masschange}
Let $G$ and $G_\rho$ have the same singular coefficient one and solve the same local conformal Robin equation near $p$.
Assume the backgrounds agree near $p$, have the expansion $g=\delta+O(r^{d_N+1})$ with differentiated bounds, and their Green metrics have asymptotically half-space ends with the decay in \eqref{eq:appendix-AF}.
If $(G-G_\rho)(p)=b_\rho$, their end masses satisfy
\begin{equation}
 m(G^{4/(N-2)}g)-m(G_\rho^{4/(N-2)}g_\rho)
 =4(N-1)|S^{N-1}_+|b_\rho.
\end{equation}
The difference $G-G_\rho$ is assumed regular at $p$.
\end{lemma}
\section{The scalar-flat and minimal-boundary quadratic forms}
\label{sec:ordinary-endpoints}
In this section, $(\sigma,\kappa)=(0,2)$ or $(1,0)$, $N\ge5$, and $m=N-1$.
Put $s=|x|^2$.
We compute on the full polynomial spaces \eqref{crit:eq:fermi}, with degrees at most $d_N^{\rm ord}$.
These are the ordinary degrees, for which the quadratic integrals converge.
The calculation has three parts.
We first describe the scalar-, vector-, and tensor-derived tensors satisfying \eqref{crit:eq:fermi} and compute their contractions in the source and raw quadratic coefficient.
We then evaluate the scalar correction with the Robin or Neumann inverse.
Finally, we assemble the matrices of the negative scale derivative and prove their positivity.

Tangential spherical integration makes distinct harmonic degrees and the three tensor types orthogonal.
Within each type, all copies of a given harmonic degree, including different normal powers, are retained in one matrix.
Thus the computation includes their mixed terms.
\subsection{Decomposition of polynomial metric tensors}
A nonzero homogeneous polynomial tensor of degree $d\ge1$ has finite $\N$-norm precisely when $2d<N-2$, since the radial integrand at infinity has order $r^{2d+1-N}$.

For ordinary polynomials, \eqref{eq:quadratic-action} and \eqref{eq:quadratic-bounded} extend by cutoff convergence in $\N$.
Critical polynomials are kept cut off and treated in Section~\ref{sec:critical}.

\begin{proposition}[Polynomial decomposition and contractions]\label{prop:polynomial-decomposition}
For $1\le d\le d_N^{\rm ord}$, the full Fermi spaces $\mathscr J_d$ in \eqref{crit:eq:fermi} decompose into the scalar-, vector-, and tensor-derived families described below.
Different tangential harmonic degrees and tensor types are orthogonal for the raw and corrected quadratic forms.
Within a fixed harmonic type, all multiplicity components are retained; their contractions and Gram matrices are the formulas in the proof.
\end{proposition}
\begin{proof}
The scalar-derived tensors are built from a tangential harmonic polynomial, its first two derivatives, and radial polynomial factors.
The following formulas describe the full family, including tensors with nonzero divergence.
For a homogeneous harmonic polynomial $H=H_\ell$ of degree $\ell$, write $H_a=\partial_aH$, $H_{ab}=\partial_a\partial_bH$ and use
\begin{equation}\label{lc:eq:Sgen}
 s^kH_{ab},\quad s^{k-1}(x_aH_b+x_bH_a),\quad s^{k-1}H\delta_{ab},\quad
 s^{k-2}Hx_ax_b.
\end{equation}
A coefficient vector $v=(a,b,c,e)$ is trace free precisely when $2\ell b+mc+e=0$.
The tangential degree is $r=\ell-2+2k$.
A polynomial independent trace-free basis consists of
\[
 (1,0,0,0)\ (\ell\ge2,k\ge0),\quad
 (0,1,-2\ell/m,0)\ (\ell\ge1,k\ge1),\quad
 (0,0,-1/m,1)\ (\ell\ge0,k\ge2).
\]
Multiply by $t^{j_1}$ for $j_1\ge1$ and $r+j_1\le d_N^{\rm ord}$; in the umbilic case exclude $j_1=1$.
At $j_1=0$ the radial constraint leaves, for $\ell\ge2,k\ge2$, the single vector
\begin{equation}\label{lc:eq:Sradial}
 \left(1,-(\ell-1),\frac{\ell(\ell-1)}{m-1},
                 \frac{\ell(\ell-1)(m-2)}{m-1}\right).
\end{equation}
There are no $j_1=0$ scalar modes with $\ell=0,1$.

The radial contraction, divergence, and double divergence determine the raw coefficient and the scalar source.
To compute them, set
\begin{align*}
 B_x&=(\ell-1)a+b,& B_y&=\ell b+c+e,\\
 D_x&=2k(\ell-1)a+(m+\ell+2k-2)b+c,&
 D_y&=2\ell(k-1)b+2(k-1)c+(m+\ell+2k-3)e,\\
 D_2&=2\ell(k-1)D_x+(m+\ell+2k-4)D_y.
\end{align*}
Then
\[
 hx=B_xs^k\nabla H+B_ys^{k-1}Hx,\quad
 \Div h=D_xs^{k-1}\nabla H+D_ys^{k-2}Hx,\quad
 \Div\Div h=D_2s^{k-2}H.
\]
For the tensor multiplied by $t^{j_1}$, the polynomial on the right of the source equation is
\begin{align*}
 F=t^{j_1}\biggl[{}&N(N-2)(\ell B_x+B_y)s^k
 -(N-2)(\ell D_x+D_y)s^{k-1}\left(\frac{2^{\sigma(N-2)/2}}{U}\right)^{2/(N-2)}\\
 &+c_ND_2s^{k-2}\left(\frac{2^{\sigma(N-2)/2}}{U}\right)^{4/(N-2)}\biggr].
\end{align*}
Here $(2^{\sigma(N-2)/2}/U)^{2/(N-2)}$ equals $s+(t+1)^2$ in the scalar-flat case and $1+s+t^2$ in the minimal-boundary case, so the displayed expression is polynomial in either case.
Terms with zero coefficients are omitted, before any negative power could be interpreted.

The remaining angular integrals are determined by the Gram matrix and Laplacian of the generators.
With $\fint_{S^{m-1}} H^2=1$, write $u=\ell(2\ell+m-2)$ and $v=\ell(\ell-1)(2\ell+m-2)(2\ell+m-4)$.
The angular Gram matrix of \eqref{lc:eq:Sgen} is
\[
 G_S=\begin{pmatrix}
 v&2(\ell-1)u&0&\ell(\ell-1)\\
 2(\ell-1)u&2u+2\ell^2&2\ell&2\ell\\
 0&2\ell&m&1\\
 \ell(\ell-1)&2\ell&1&1
 \end{pmatrix}.
\]
The Laplacian maps the coefficient vector at $k$ to the following vector at $k-1$:
\begin{equation}
 \left(\begin{array}{l}
 2k(m+2\ell+2k-6)a+4b\\
 2(k-1)(m+2\ell+2k-4)b+2e\\
 2(k-1)(m+2\ell+2k-4)c+2e\\
 2(k-2)(m+2\ell+2k-2)e
 \end{array}\right).
\end{equation}
\medskip
\paragraph{Vector-derived tensors.}
The vector-derived family is obtained in the same way from a tangential harmonic vector polynomial.
Its contractions below show that the scalar source vanishes, so only the raw quadratic coefficient is needed.
Let $V_\ell$ be a harmonic homogeneous vector polynomial with $x\cdot V=0$, $\Div V=0$, and $\fint_{S^{m-1}}|V|^2=1$.
Use
\[
 s^k(\partial_aV_b+\partial_bV_a),\qquad s^{k-1}(x_aV_b+x_bV_a),
\]
of tangential degree $r=\ell-1+2k$.
The first exists for $\ell\ge2,k\ge0$, the second for $\ell\ge1,k\ge1$.
At $j_1\ge1$ and $r+j_1\le d_N^{\rm ord}$, use both available vectors (excluding $j_1=1$ in the umbilic case).
At $j_1=0$, $\ell\ge2,k\ge1$, the radial vector is $(1,-(\ell-1))$.
For coefficients $(a,b)$,
\[
 hx=((\ell-1)a+b)s^kV,\qquad
 \Div h=[2k(\ell-1)a+(m+\ell+2k-1)b]s^{k-1}V.
\]
The scalar source is identically zero.
The angular Gram matrix and Laplacian coefficients are
\[
 G_V=\begin{pmatrix}2(\ell-1)(m+2\ell-2)&2(\ell-1)\\2(\ell-1)&2\end{pmatrix},
\]
\[
 (a,b)\mapsto
 \big(2k(m+2\ell+2k-4)a+2b,\ 2(k-1)(m+2\ell+2k-2)b\big).
\]

\medskip
\paragraph{Tensor-derived tensors.}
For the tensor-derived family, both the radial contraction and divergence vanish.
Its quadratic coefficient therefore reduces to the derivative term.
Let $T_\ell$ be harmonic homogeneous of degree $\ell\ge2$, trace free, divergence free, with $T_\ell x=0$ and $\fint_{S^{m-1}}|T_\ell|^2=1$.
Use $s^kt^{j_1}T_\ell$ with $2k+j_1+\ell\le d_N^{\rm ord}$, again excluding $j_1=1$ for umbilic jets.
Both $hx$ and the scalar source vanish.
For the scalar-flat equation, the normalized energy entry for $(k,j_1)$ and $(h,j_2)$ is
\begin{align*}
 E_{(k,j_1),(h,j_2)}={}&-\frac{(2k+\ell)(2h+\ell)+\ell(\ell+m-2)}{4\mu_N}
 M_{2\ell+2k+2h-2,j_1+j_2}^{N-2}\\
 &-\frac{j_1j_2}{4\mu_N}M_{2\ell+2k+2h,j_1+j_2-2}^{N-2},
\end{align*}
with the second term omitted when $j_1j_2=0$.
For the minimal-boundary equation, replace $M_{R,j_1}^{\nu}/\mu_N$ by $M^I_{R,j_1,\nu}/\mu_N^I$ in this formula.

\medskip
\paragraph{Completeness of the tensor decomposition.}
We check that these families exhaust the polynomial space before imposing the boundary radial constraint.
The harmonic dimensions are
\begin{align*}
 h_\ell&=\binom{m+\ell-1}{\ell}-\binom{m+\ell-3}{\ell-2},\qquad
 v_\ell=mh_\ell-h_{\ell+1}-h_{\ell-1},\\
 t_\ell&=\frac{m(m-3)(\ell+m-1)(\ell-1)(2\ell+m-2)(\ell+m-4)!}
 {2(m-2)!(\ell+1)!}.
\end{align*}
Negative-index binomial terms and $h_{-1}$ are zero.
Harmonic polynomial decomposition, spherical Hodge decomposition, and tangential TT decomposition give these three types.
Their displayed polynomial reconstructions are injective after exactly the low-degree omissions above.
Nonzero blocks of different types are orthogonal, while the Gram matrices give independence of the multiplicity vectors.
For each tangential polynomial degree $r$, the dimensions add to
\[
 \left(\frac{m(m+1)}2-1\right)\binom{m+r-1}{r}.
\]
This identity also follows by inserting the three dimension formulas and their multiplicities.
Boundary radial contraction is imposed separately at $j_1=0$ by \eqref{lc:eq:Sradial} and its vector analogue.

For scalar and vector modules, one obtains the tangential gradient pairing from
\[
 \partial h_i\cdot\partial h_j
 =\tfrac12\Delta(h_i\cdot h_j)-\tfrac12(\Delta h_i\cdot h_j+h_i\cdot\Delta h_j).
\]
Its spherical average uses the displayed Gram and Laplacian formulas.
Normal derivatives contribute $j_1j_2\,t^{j_1+j_2-2}$.
\end{proof}

For a basis $\{h_i\}$, let $E_{ij}:= F(h_i,h_j)/C_N^{\rm alg}$ in the scalar-flat case and $E_{ij}:= F(h_i,h_j)/C_N^I$ in the minimal-boundary case.
The positive factors are defined in \eqref{lc:eq:normalization} and \eqref{min:eq:normalization} below.
If $d_i$ is the total polynomial degree of $h_i$, set $\mathsf D:=\operatorname{diag}(d_i)$.
The normalized ordinary scale matrix is
\begin{equation}\label{lc:eq:scale}
 P=-(\mathsf D E+E\mathsf D).
\end{equation}
This is the homogeneous scale-derivative identity \cite[(4.28)]{PartI}, with the negative sign used here.

\subsection{The linearized scalar-flat problem}\label{lc:sec:moments}
Only the scalar-derived tensors require a correction.
We first construct a particular solution that solves the interior linearized equation exactly.
It generally leaves a Robin residual on the boundary.
The solution of the homogeneous interior equation with this boundary data then gives the remaining contribution to the inverse pairing in \eqref{eq:quadratic-action}.

\medskip
\paragraph{Bulk and boundary integrals.}
The raw coefficient and the particular-solution pairings use Euler's beta integral \cite[(5.12.1)]{DLMF}.
Here $U=(s+(t+1)^2)^{-(N-2)/2}$, and we set $z:=t+1$.
For $N>4$, divide out the common positive factor
\begin{equation}\label{lc:eq:normalization}
 C_N^{\rm alg}:=|S^{N-2}|\mu_N,\qquad
 \mu_N:=\frac{B((N-1)/2,(N-3)/2)}{2(N-4)}.
\end{equation}
Angular averages are normalized to have total mass one.
The radial integrals are \eqref{lc:eq:moment}; a boundary integral omits the last beta factor.
If $R=2j$ and $\nu=N-2+j_2$, the normalized tangential beta factor is
\[
 (N-4)\frac{((N-1)/2)_j((N-3)/2)_{j_2-j}}{(N-2)_{j_2}}.
\]
For the bulk, multiply by
\[
 \frac{j_1!\,(N-5+2j_2-R-j_1)!}{(N-4+2j_2-R)!}.
\]
Every argument used by the ordinary computations lies in the convergent range.
These ratios are rational, for odd as well as even $N$.

\medskip
\paragraph{Particular solutions and boundary residuals.}
Let $H_\ell(x)$ be a harmonic homogeneous polynomial of degree $\ell$, and put $\Phi_\ell:=H_\ell U/(s+z^2)$.
For a polynomial $G(s,z)$, define $\mathcal T_\ell$ on the basis $(s+z^2)^az^b$ by
\[
 \mathcal T_\ell\bigl((s+z^2)^az^b\bigr) ={}2(a+b+\ell-1)(N-2a)(s+z^2)^az^b -b(b-1)(s+z^2)^{a+1}z^{b-2}.
\]
Let $B_J:=\partial_\nu-NU^{2/(N-2)}$.
The following identity extends the polynomial calculation in \cite[Lemma~5.2]{PartI} to every harmonic degree used here.
\begin{lemma}
For every polynomial $G(s,z)$,
\begin{equation}
 \begin{cases}
 -\Delta(G\Phi_\ell)=\dfrac{\Phi_\ell}{s+z^2}\mathcal T_\ell G &\text{in }\R^N_+,\\
 B_J(G\Phi_\ell)=-(G_z)_{z=1}\Phi_\ell &\text{on }\partial\R^N_+.
 \end{cases}
\end{equation}
\end{lemma}
\begin{proof}
Apply $\Delta(H_\ell f)=H_\ell\{4s f_{ss}+2(m+2\ell)f_s+f_{zz}\}$ to $f=(s+z^2)^{a-N/2}z^b$.
On $z=1$, the terms differentiating $U/(s+z^2)$ cancel the Robin coefficient, leaving $-G_z\Phi_\ell$.
\end{proof}

Thus for a source $S=(\Phi_\ell/(s+z^2))F$, expand $F$ in the same basis and solve $\mathcal T_\ell G=-F$ by triangular elimination.
On each monomial, divide by $2(a+b+\ell-1)(N-2a)$ and then pass to $(a+1,b-2)$.
For every basis element used in Table~\ref{tab:endpoint-certificates}, each coefficient requiring elimination has a nonzero denominator.
This is checked exactly when assembling the finite matrices.
The resulting particular solution $w=G\Phi_\ell$ is in the energy space, but generally has a nonzero Robin defect $b=B_Jw$.

If $\Psi$ is harmonic with $B_J\Psi=b$, then $v=w-\Psi$ is the corrector with zero Robin boundary data.
The pairing identity \cite[Lemma~5.1]{PartI} reads
\begin{equation}\label{lc:eq:schur}
  F(h_i,h_j)=\mathcal Q_{ij}+c_N^{-1}\left(\int_{\R^N_+} S_iw_j+\int_{\partial\R^N_+}b_jw_i
                                      -\int_{\partial\R^N_+}b_i\Psi_j\right).
\end{equation}
Dividing \eqref{lc:eq:schur} by $C_N^{\rm alg}$ produces $E_{ij}$.
The first two pairings in parentheses are evaluated by \eqref{lc:eq:moment}, and the last by \eqref{lc:eq:response}.
Symmetry also follows from the self-adjointness of the inverse.

\medskip
\paragraph{The harmonic Robin problem on the ball.}
\begin{proposition}[Scalar-flat boundary pairings]\label{prop:scalar-flat-response}
For each scalar component in Proposition~\ref{prop:polynomial-decomposition}, the normalized boundary-response pairing is \eqref{lc:eq:response}, with the kernel terms and constant mode specified below.
Its value is rational for the integer parameters in the ordinary range.
\end{proposition}
\begin{proof}
To evaluate the last pairing in \eqref{lc:eq:schur}, we compactify the half-space and decompose the boundary data into spherical harmonics.
The conformal map $\zeta=2(y+e_N)/(|x|^2+(t+1)^2)-e_N$ sends the half-space to the unit ball and pulls its Euclidean metric back to $4U^{4/(N-2)}\delta$.
Writing $\Psi(y)=U(y)\hat\Psi(\zeta(y))$ gives
\[
 B_J\Psi=2U^{N/(N-2)}(\partial_\nu\hat\Psi-\hat\Psi)\circ\zeta.
\]
Thus a boundary defect $b_a=(1+s)^a\Phi_\ell$ becomes
\[
 f_a(\zeta)=\frac{b_a}{2U^{N/(N-2)}}=
 \frac{2^{a-1}H_\ell(\zeta^T)}{(1+\zeta_N)^{a+\ell}}.
\]
Here $\nu$ is the outward normal of the unit ball; the inverse is that of $\partial_\nu-1$ on harmonic functions.
The half-space pairing $\int_{\partial\R^N_+}b_i\Psi_j$ is $2^{2-N}$ times the corresponding unit-sphere pairing.
A spherical harmonic of degree $j+\ell$ therefore has denominator $j+\ell-1$.
Gegenbauer expansion in $\varpi=\zeta_N$, with parameter $\Lambda=\ell+(N-2)/2$ and the orthogonality normalization of \cite[Section~18.3]{DLMF}, gives the normalized response
\begin{equation}\label{lc:eq:response}
 \mathcal R_{\ell}(a,b)=K_{N,\ell,a,b}
 \sum_{j\in\mathcal I_\ell}
 \frac{(N+2\ell-2)_j(a+\ell)_j(b+\ell)_j}
 {j!(N+\ell-1-a)_j(N+\ell-1-b)_j}
 \frac{2j+N+2\ell-2}{(N+2\ell-2)(j+\ell-1)},
\end{equation}
where, with $A=(N-1)/2$ and $B=A+\ell$,
\[
 K_{N,\ell,a,b}=\frac{N-4}{2}
 \frac{B(A-a,B)B(A-b,B)}{B(B,B)B(A,A-1)}.
\]
This assumes $\fint_{S^{m-1}} H_\ell^2=1$.
For $H_2=Q_A$ with $|A|=1$, multiply by $2/[m(m+2)]$.
For $\ell\ge2$, $\mathcal I_\ell=\{0,1,\ldots\}$.
For $\ell=1$, omit $j=0$ (translations).
For $\ell=0$, omit $j=1$ (dilation), but retain $j=0$ with its negative denominator.
If $\ell=0$ and $a=0$ or $b=0$, only the constant term remains and the response is $-K_{N,0,a,b}$.
The assembled defects annihilate the omitted kernel: test the particular equation against the corresponding Jacobi field and use Lemma~\ref{lem:stationarity}.
This cancellation is imposed before evaluating a singular degree-one response.

To justify the normalization, the zeroth Gegenbauer projection of $f_a$ is proportional to $2^{N+\ell-3}B(A-a,B)$, while the zeroth Gegenbauer norm is $2^{N+2\ell-2}B(B,B)$.
Multiplication by $2^{2-N}$ and division by $\mu_N$ gives $K$.
The ratio of the $j$th and zeroth coefficients is $(-1)^j(2\Lambda)_j(a+\ell)_j/[j!(N+\ell-1-a)_j]$, and the Gegenbauer norm ratio is $(2\Lambda)_j\Lambda/[j!(j+\Lambda)]$.
These identities give \eqref{lc:eq:response}.

For the integer parameters appearing here, each summand is a rational function of $j$ with simple poles, after cancellation.
Write it as $\sum_x c_x/(j+x)$; absolute convergence gives $\sum_xc_x=0$.
For $\ell\ge2$, its sum is
\[
 -\sum_xc_x H_{x-1},
\]
where $H_k=\sum_{v=1}^k1/v$.
For $\ell=1$, summation starts at one and gives $-\sum_xc_xH_x$.
For $\ell=0$, after omitting $j=1$, the value is
\[
 -c_{-1}+\sum_{x>0}c_x\left(-H_{x-1}-\frac1{x+1}\right).
\]
Thus \eqref{lc:eq:response} is computed exactly using rational arithmetic.
The constant-mode check is particularly simple:
\[
 B_J(-U/2)=U^{N/(N-2)},\qquad
 \mathcal R_0(0,0)=-\tfrac12\frac{M_{0,0}^{N-1,\partial}}{\mu_N}.
\]
\end{proof}

\subsection{The linearized problem with minimal boundary}
For minimal boundary, the same decomposition leaves a Neumann residual.
The compactification is now a hemisphere, and the homogeneous interior operator differs from the harmonic operator in the scalar-flat case.
The hemisphere inverse is recalled in Lemma~\ref{min:responselemma}; here we compute the particular solutions and boundary pairings for the full tensor decomposition.

\medskip
\paragraph{Particular solutions and boundary residuals.}
Use the source and raw form in \eqref{lc:eq:rawsource}--\eqref{lc:eq:rawpolar} with $U=(2/(1+s+t^2))^{(N-2)/2}$.
Write $J:=-\Delta-\tfrac{N(N+2)}4U^{4/(N-2)}$ and $B_J:=\partial_\nu$.
Let $H_\ell(x)$ be homogeneous harmonic and $\Phi_\ell:=H_\ell U/(1+s+t^2)$.
\begin{lemma}
For every polynomial $G(s,t)$,
\begin{equation}
 \begin{cases}
 J(G\Phi_\ell)=\dfrac{\Phi_\ell}{1+s+t^2}\mathcal T^I_\ell G &\text{in }\R^N_+,\\
 B_J(G\Phi_\ell)=-(G_t)_{t=0}\Phi_\ell &\text{on }\partial\R^N_+,
 \end{cases}
\end{equation}
where
\begin{equation}\label{min:T}
 \begin{aligned}
 \mathcal T^I_\ell\bigl((1+s+t^2)^at^b\bigr)
 ={}&2(a+b+\ell-1)(N-2a)(1+s+t^2)^at^b\\
 &+4a(a-N-1)(1+s+t^2)^{a-1}t^b
 -b(b-1)(1+s+t^2)^{a+1}t^{b-2}.
 \end{aligned}
\end{equation}
\end{lemma}
\begin{proof}
For a radial coefficient $f(s,t)$,
\[
 \Delta(H_\ell f)=H_\ell\{4s f_{ss}+2(m+2\ell)f_s+f_{tt}\}.
\]
Insert $f=(1+s+t^2)^{a-N/2}t^b$ and subtract $N(N+2)(1+s+t^2)^{-2}f$.
Collection of the three monomials gives \eqref{min:T}.
At $t=0$, both $\partial_t(1+s+t^2)$ and $\partial_tU$ vanish, proving the boundary identity.
\end{proof}

Write $S_h=\Phi_\ell F(s,t)/(1+s+t^2)$, expand $F$ in the basis $(1+s+t^2)^at^b$, and solve $\mathcal T^I_\ell G=-F$.
Process the weights $2a+b$ in descending order, and within a weight process $b$ in descending order.
The two off-diagonal monomials in \eqref{min:T} occur later in this ordering.
The factor $2(a+b+\ell-1)(N-2a)$ is nonzero whenever an actual retained source coefficient has to be eliminated.
The program \texttt{verify\_compactness.py} verifies both this property and the exact zero remainder for every generated source.

Set $w_h=G\Phi_\ell$ and $b_h=B_Jw_h$.
The particular solution is in the energy space in the ordinary range.
Solve $J\Psi_h=0$, $B_J\Psi_h=b_h$, and put $v_h:=w_h-\Psi_h$.
The unnormalized pairing again satisfies \eqref{lc:eq:schur}; division by $C_N^I$ gives the minimal-boundary entries $E_{ij}$.
The assembled boundary defects annihilate the Jacobi modes: test $Jw=-S$ against a Jacobi field and use Lemma~\ref{lem:stationarity}.
The verifier also checks these beta-integral compatibility identities independently.

\medskip
\paragraph{Integrals for the scalar correction.}
\begin{proposition}[Minimal-boundary pairings]\label{prop:minimal-response}
The scalar-correction pairing is given by \eqref{lc:eq:schur} and \eqref{min:series}, with the normalization \eqref{min:eq:normalization}.
For every parameter tuple used in Theorem~\ref{thm:ordinary-endpoints}, its value is a rational polynomial of degree at most two in $\pi^{-1}$.
\end{proposition}
\begin{proof}
We evaluate the three pairings in \eqref{lc:eq:schur} using the following integrals and the response in Lemma~\ref{min:responselemma}.
All matrices in this part divide out the common positive factor
\begin{equation}\label{min:eq:normalization}
 C_N^I:=|S^{N-2}|2^{N-2}\mu_N^I,\qquad
 \mu_N^I:=\frac{\Gamma((N-1)/2)\sqrt\pi\,\Gamma((N-4)/2)}{4\Gamma(N-2)}
 \quad(N>4).
\end{equation}
For a harmonic of tangential spherical average square one, use the radial integrals in \eqref{min:moment} and the boundary integrals in Section~\ref{sec:source}.
Every invoked integral is in its convergent range.

A defect $b_a=(1+s)^a\Phi_\ell$ transforms on the equator to
\[
 f_a(\zeta)=\frac{2^{a-1}H_\ell(\zeta^T)}{(1+\zeta_N)^{a+\ell}}.
\]
Let $A=(N-1)/2$, $B=A+\ell$, $2\Lambda=N+2\ell-2$, and define
\[
 K^I_{N,\ell,a,b}=\frac{B(A-a,B)B(A-b,B)}{4\mu_N^I B(B,B)}.
\]
The same Gegenbauer coefficients derived in Section~\ref{lc:sec:moments}, but now with the response of Lemma~\ref{min:responselemma}, give
\begin{equation}\label{min:series}
 \mathcal R^I_\ell(a,b)=K^I_{N,\ell,a,b}
 \sum_{j\in\mathcal I_\ell}
 \frac{(2\Lambda)_j(a+\ell)_j(b+\ell)_j}
 {j!(N+\ell-1-a)_j(N+\ell-1-b)_j}
 \frac{2j+2\Lambda}{2\Lambda}\,r_{j+\ell}.
\end{equation}
The set $\mathcal I_\ell$ omits only $j+\ell=1$; the $j+\ell=0$ term has response zero.
If $\ell=0$ and $a=0$ or $b=0$, the response is zero by the constant-mode observation.
For comparison, the scalar-flat Robin response is $-K_{N,0,a,b}$ in \eqref{lc:eq:response}.

\medskip
\paragraph{Summation of the boundary spectral series.}
For every parameter tuple used here, each parity admits a finite telescoping evaluation, with the term at infinity included as follows.
Separate $j=2k+j_0$ into its two parities, with starting indices $(0,1)$ for $\ell\ge2$, $(2,1)$ for $\ell=1$, and $(2,3)$ for $\ell=0$.
For each retained parameter tuple, the term $t(k)$ admits a rational function $g(k)$ such that
\[
 t(k+1)g(k+1)-t(k)g(k)=t(k).
\]
The code uses Gosper's algorithm \cite{Gosper} to find the rational function $g$, and checks the displayed identity by exact polynomial multiplication.
Thus
\[
 \sum_{k=0}^\infty t(k)=\lim_{K\to\infty}t(K)g(K)-t(0)g(0).
\]
For odd $N$, $g(k)=O(k)$ and $t(k)=O(k^{-[N+1-2(a+b+\ell)]})$, with the bracket strictly greater than one, so the endpoint is zero.
For even $N$, the term is a rational function of $k$.
Its rational antidifference can have a nonzero limit at infinity: the program computes that constant from the leading numerator and denominator coefficients and retains it in the exact sum.

Every final ordinary entry is a polynomial of degree at most two in $\pi^{-1}$ with rational coefficients.
The scalar-flat entries are rational by Proposition~\ref{prop:scalar-flat-response}.
\end{proof}

\subsection{Positive definiteness of the scale-derivative matrices}
The assembly in Section~\ref{sec:ordinary-endpoints} determines the energy matrix, and \eqref{lc:eq:scale} determines its negative scale derivative.
The remaining step is a finite positivity check on all retained blocks.
\begin{theorem}[Computer-assisted positivity]\label{thm:ordinary-endpoints}
For $(\sigma,\kappa)=(0,2)$ the matrix $P$ in \eqref{lc:eq:scale} is positive definite on every nonzero block of the full ordinary Fermi jet space for general boundary in $5\le N\le14$ and umbilic boundary in $5\le N\le21$.
For $(\sigma,\kappa)=(1,0)$ the corresponding ranges are $5\le N\le14$ and $5\le N\le20$.
\end{theorem}
\begin{proof}
Proposition~\ref{prop:polynomial-decomposition}, the integrals \eqref{lc:eq:moment} and \eqref{min:moment}, and the boundary pairings in Propositions~\ref{prop:scalar-flat-response} and~\ref{prop:minimal-response} determine every matrix entry.
For the scalar-flat equation, their entries and $LDL^{\mathsf T}$ pivots are rational.
For the minimal-boundary equation, first store every entry as a rational polynomial in $\pi^{-1}$.
Machin's identity
\[
 \pi=16\arctan(1/5)-4\arctan(1/239)
\]
and alternating Taylor bounds give a rational enclosure narrower than $2^{-256}$.
The program then uses outward-rounded dyadic interval arithmetic with 256 fractional bits to enclose every $LDL^{\mathsf T}$ pivot.
At each division, the preceding pivot interval has a positive lower endpoint.
All resulting pivot lower bounds are strictly positive.

Table~\ref{tab:endpoint-certificates} specifies the complete list of nonempty matrix groups.
The group at a fixed dimension consists of all scalar, vector, and tensor multiplicity spaces allowed by \eqref{crit:eq:fermi}.
The displayed counts can be reconstructed from the basis restrictions and are checked during matrix reconstruction.
The scalar-flat count is $302$ and the minimal-boundary count is $276$.
The ancillary program \texttt{verify\_compactness.py} reconstructs these matrices from the integration and inverse formulas and computes their rational pivots or pivot intervals.
From the submission directory, the command \texttt{python anc/verify\_compactness.py full} checks all $578$ matrices and the independent finite identities.
The shorter \texttt{quick} command checks those identities and six sample matrices.
\end{proof}
\begin{table}[ht]
\centering\small
\begin{tabular}{@{}lccr@{}}\toprule
Boundary class & Dimensions & Blocks per dimension & Total\\ \midrule
General, both equations & $5,6$ & $1$ & $2$\\
 & $7,8$ & $5$ & $10$\\
 & $9,10$ & $10$ & $20$\\
 & $11,12$ & $13$ & $26$\\
 & $13,14$ & $16$ & $32$\\ \midrule
Umbilic, both equations & $7,8$ & $2$ & $4$\\
 & $9,10$ & $6$ & $12$\\
 & $11,12$ & $11$ & $22$\\
 & $13,14$ & $14$ & $28$\\
 & $15,16$ & $17$ & $34$\\
 & $17,18$ & $20$ & $40$\\
 & $19,20$ & $23$ & $46$\\
Umbilic, scalar-flat only & $21$ & $26$ & $26$\\\bottomrule
\end{tabular}
\caption{Positive-definite scale-derivative matrices, grouped by dimension and boundary geometry. The umbilic polynomial spaces in dimensions five and six are empty in the ordinary degree range.}
\label{tab:endpoint-certificates}
\end{table}

\section{Coercivity for arbitrary constant mean curvature}\label{sec:ordinary}
The endpoint calculations give the sign for fixed scalar-flat or minimal-boundary data.
For arbitrary constant mean curvature through dimension eight, only degree-one and degree-two ordinary tensors occur.
We prove their scale inequality directly, keeping the boundary response in the mixed-degree block.

\subsection{The linearized problem and the quadratic form}
Throughout the degree-two calculation, let $N=7$ or $8$; the degree-one conclusion will also be proved for $N=5,6$.
Put $m=N-1$, $y=(x,t)\in\R^m\times[0,\infty)$, and $L=-\kappa\in\R$.
Set
\[
 U=\left(\frac{2}{1+|x|^2+(t-L)^2}\right)^{(N-2)/2}.
\]
The outward normal, energy, and Jacobi form are those of Section~\ref{sec:source}.

We use the corrected quadratic energy $F(k)$ in \eqref{eq:quadratic-action}.
The notation used for its translated and scaled version is
\begin{equation}
  F_h(\xi,\varepsilon):= F(h_{\xi,\varepsilon}),\qquad
 h_{\xi,\varepsilon}(x,t):=h(\xi+\varepsilon x,\varepsilon t).
\end{equation}
Here $\xi\in\R^{N-1}$ and $\varepsilon>0$.
These are the components of the logarithmic metric after translating and dilating coordinates and removing the constant Euclidean scale factor.
Conformal invariance of the energy identifies this definition with the quadratic coefficient evaluated at the bubble $\varepsilon^{-(N-2)/2}U((x-\xi)/\varepsilon,t/\varepsilon)$ in the original coordinates.
In particular, $F_h= F_h(0,1)= F(h)$ below.
A degree-two model is ordinary precisely when $N>6$.
All scale derivatives keep $\kappa$ fixed.
Auxiliary symbols for polynomials and integrals are local to the corresponding calculation.

\medskip
Let $W$ be a tangential algebraic Weyl tensor and $H^W_{ab}(x)=W_{aibj}x_ix_j$.
Let $A,B\in\operatorname{Sym}^2_0(\R^m)$ and
\[
 C_{ab,c}=C_{ba,c},\qquad \sum_a C_{aa,c}=0,\qquad
 C(x)_{ab}=C_{ab,c}x_c.
\]
All normal components of the following tensor are zero:
\begin{equation}\label{ord:eq:jet}
 h=H^W(x)+tA+t^2B+tC(x).
\end{equation}
\begin{theorem}\label{ord:thm:main}
For $N=7,8$ and each fixed $\kappa\in\R$, there is $c_{N,\kappa}>0$ such that the corrected quadratic model of \eqref{ord:eq:jet} satisfies
\[
 -\varepsilon\partial_\varepsilon F_h(0,\varepsilon)
 \ge c_{N,\kappa}\left[
 \varepsilon^2|A|^2+
 \varepsilon^4\bigl(|B|^2+|C|^2+|W|^2\bigr)\right]
 \qquad(\varepsilon>0).
\]
The constants can be chosen uniformly on compact intervals of $\kappa$.
\end{theorem}
We first identify the angular spaces containing the scalar correction.
For \eqref{ord:eq:jet}, $\partial_i\partial_jh_{ij}=0$.
The quadratic coefficient is given by \eqref{lc:eq:rawsource}--\eqref{lc:eq:rawpolar}.
Its scalar sources have tangential harmonic degrees one, two, and three.
Metric stationarity gives orthogonality to the Jacobi kernel, so the inverse can be computed in these angular spaces without changing $F$.

\begin{lemma}\label{ord:lem:positive}
On tangential angular degree $\ell\ge2$, $\B$ is positive.
On degree one, it is nonnegative and its nullspace consists of the tangential translation Jacobi fields.
It has a bounded inverse on the corresponding kernel complements for fixed finite $L$.
\end{lemma}
\begin{proof}
The positive radial factor of a translation field is $\phi=r(1+r^2+(t-L)^2)^{-N/2}$, $r=|x|$.
The ground-state transform gives
\begin{align*}
 \B(wY_\ell,wY_\ell)={}&
 \int_0^\infty\!\int_0^\infty\phi^2|\nabla_{r,t}(w/\phi)|^2r^{N-2}\dd r\dd t\\
 &+[\ell(\ell+N-3)-(N-2)]\int_0^\infty\!\int_0^\infty w^2r^{N-4}\dd r\dd t,
\end{align*}
up to positive angular normalization.
The second coefficient is positive for $\ell\ge2$ and zero for $\ell=1$.
In the latter case, equality makes $w/\phi$ constant.
Fredholm invertibility on the complement follows from the conformal equivalence with a compact round cap.
\end{proof}
For any scalar trial $v_0$ in a positive sourced sector,
\[
  F_h\le\mathcal Q(h)+2c_N^{-1}\int_{\R^N_+} S_hv_0+c_N^{-1}\B(v_0,v_0).
\]
We use this bound only for a single homogeneous degree.
The mixed normal block is computed by an exact response formula.

\subsection{Terms linear and quadratic in the normal coordinate}
Initially take $h=(at+bt^2)A$ with $|A|=1$ and put
\begin{align*}
 Q_A&:=A_{ab}x_ax_b,\qquad \Phi:=\frac{Q_AU}{1+|x|^2+(t-L)^2},\\
 J&:=-\Delta-\frac{N(N+2)}4U^{4/(N-2)},\qquad
 B_0:=\partial_\nu-\frac{N\kappa}2U^{2/(N-2)}.
\end{align*}
Then $J\Phi=2N\Phi/\bigl(1+|x|^2+(t-L)^2\bigr)$, $B_0\Phi=0$, and $S_h=N(N-2)(at+bt^2)\Phi/\bigl(1+|x|^2+(t-L)^2\bigr)$.
Write $z=t-L$ and
\begin{align*}
 g(z)&=a(z/2+L)+b\left(z^2/3+Lz+L^2+\frac1{3(N-2)}\right),\\
 q_0&=\frac b{6(N-2)},\qquad
 v_p=-\frac{N-2}{2}(g+q_0\bigl(1+|x|^2+(t-L)^2\bigr))\Phi,\\
 \delta_{a,b}&=\frac{N-2}{4}a+\frac{N-3}{6}Lb.
\end{align*}
\begin{lemma}\label{lem:normal-particular}
The particular solution $v_p$ satisfies
\[
 Jv_p=-S_h,\qquad B_0v_p=\delta_{a,b}\Phi.
\]
Let $\psi$ solve $J\psi=0$, $B_0\psi=\Phi$ in the degree-two energy space.
Then the exact scalar corrector is $v=v_p-\delta_{a,b}\psi$.
\end{lemma}
\begin{proof}
Differentiate the displayed polynomial multiplier of $\Phi$ at fixed $|x|^2$ and substitute in $J$ and $B_0$.
\end{proof}
Set
\[
 \nu_0=\frac{N-3}{2},\quad c(t)=1+(t-L)^2,\quad c_0=1+L^2,
 \quad w=c_0^{-\nu_0},\quad
 M=\int_{-L}^{\infty}(1+z^2)^{-\nu_0}\dd z,
\]
\[
 \mathfrak c_N=2^{N-2}\pi^{(N-1)/2}\frac{\Gamma((N-3)/2)}{\Gamma(N-2)},\quad
 T=\int_{\partial\R^N_+}\Phi\psi,\quad \theta=\frac{c_N^{-1}T}{\mathfrak c_N}>0.
\]
Here $\mathfrak c_N>0$ is the angular integration constant for the ordinary quadratic model.
\begin{proposition}\label{ord:prop:E}
We have
\[
  F_{(at+bt^2)A}(0,1)=\mathfrak c_N(a,b)E(a,b)^T,
 \qquad E=E^0-\theta \mathbf d\mathbf d^T.
\]
For arbitrary $A,B$, the same matrix acts on each coordinate of $\operatorname{Sym}^2_0(\R^m)$.
Here $\mathbf d=((N-2)/4,L(N-3)/6)^T$, and
\[
 E^0_7=
 \begin{pmatrix}
 -Lw/8&-w(2L^2+1)/15\\
 -w(2L^2+1)/15&2(7L^3w+11Lw-30M)/225
 \end{pmatrix},
\]
\[
 E^0_8=
 \begin{pmatrix}
 -Lw/8&-5w(5L^2+2)/216\\
 -5w(5L^2+2)/216&-5(7L^3w-8Lw+54M)/1944
 \end{pmatrix}.
\]
\end{proposition}
\begin{proof}
Lemma~\ref{lem:normal-particular} and the pairing identity \eqref{lc:eq:schur} give
\[
  F_h=\mathcal Q(h)+c_N^{-1}\int_{\R^N_+} S_hv_p
       +c_N^{-1}\delta_{a,b}\int_{\partial\R^N_+}\Phi v_p-c_N^{-1}\delta_{a,b}^2T.
\]
Tangential beta integration gives
\begin{align*}
 \int_{\R^{N-1}}\Phi^2\dd x
 &=\frac{\mathfrak c_N}{2(N-1)(N-2)}c(t)^{-\nu_0},\\
 \int_{\R^{N-1}}\frac{\Phi^2}{1+|x|^2+(t-L)^2}\dd x
 &=\frac{(N-3)\mathfrak c_N}{4N(N-1)(N-2)}c(t)^{-\nu_0-1},\\
 \int_{\R^{N-1}}\bigl(1+|x|^2+(t-L)^2\bigr)\Phi^2\dd x
 &=\frac{\mathfrak c_N}{(N-5)(N-2)}c(t)^{1-\nu_0}.
\end{align*}
For example, with $I_j:=\int_{-L}^\infty z^j(1+z^2)^{-\nu_0}\dd z$,
\[
 I_0=M,\quad I_1=\frac{c_0w}{N-5},\quad
 I_2=\frac{M-Lc_0w}{N-6}.
\]
For $J_j:=\int_{-L}^{\infty}z^j(1+z^2)^{-\nu_0-1}\dd z$,
\[
 J_0=\frac{(N-4)M+Lw}{N-3},\quad J_1=\frac w{N-3},
 \quad J_{j+2}=I_j-J_j.
\]
For an explicit algebraic reduction, put $f(z)=a(z+L)+b(z+L)^2$, and let $\mathcal I$ and $\mathcal J$ denote polynomial integration using the integrals $I_j$ and $J_j$, respectively.
In units of $\mathfrak c_N$, the raw coefficient, the particular bulk correction, and the particular boundary correction are, respectively,
\[
 \nu_0\mathcal J(f^2)-\tfrac14\mathcal I((f')^2),\qquad -\nu_0\mathcal J(fg)-Nq_0\mathcal I(f),\qquad \delta_{a,b}\left[-\frac{g(-L)w}{N-2} -\frac{2(N-1)q_0c_0w}{(N-2)(N-5)}\right].
\]
Their sum is the quadratic form with matrix $E^0$ in Proposition~\ref{ord:prop:E}; the remaining exact response is $-\theta \delta_{a,b}^2$.
These formulas and the recurrence relations for $I_j$ and $J_j$ determine every entry by rational arithmetic.
Rotation invariance gives the assertion for arbitrary $A,B$.
\end{proof}

\subsubsection{Estimates for the boundary correction}
\begin{lemma}[Bounds for the normal boundary response]\label{lem:normal-response-bounds}
The exact response $\theta$ satisfies \eqref{ord:eq:lower} for every $L\in\R$, \eqref{ord:eq:uppernegative} for $L>0$, and \eqref{ord:eq:upperpositive} for $L\le0$.
\end{lemma}
\begin{proof}
The variational principle for $\psi$ gives
\begin{equation}\label{ord:eq:lower}
 \theta\ge\theta_-:=
 \frac{2w^2}{(N-2)^2((N-4)M+Lw)}.
\end{equation}
This is the trial $\Phi$: $T\ge(\int_{\partial\R^N_+}\Phi^2)^2/\B(\Phi,\Phi)$.
For $L>0$, the positive radial supersolution
\[
 \chi=\frac{\bigl(1+|x|^2+(t-L)^2\bigr)+2}{2L}\Phi
\]
has $B_0\chi=\Phi$ and $J\chi=2(N-2)\Phi/L$.
Comparison in the angular radial equation, justified by Lemma~\ref{ord:lem:positive}, gives
\begin{equation}\label{ord:eq:uppernegative}
 \theta\le\theta_+:=
 \frac{2w}{L(N-2)^2}\left(\frac{(N-1)c_0}{N-5}+1\right).
\end{equation}
The angular harmonic can change sign; comparison here is for its positive radial coefficient, not for the full sign-changing function.

For $L=-K\le0$, put $s=\sqrt{1+K^2}$ and
\[
 C=\frac{8K^2+4NKs+8+4(2K+Ns)^2/(N-2)}{2N},
 \qquad \chi:=\frac{|x|^2+(t-s)^2+C}{2s}\Phi.
\]
The multiplier of $\Phi$ is positive.
Its boundary defect is $\Phi$, and direct differentiation gives
\[
 2s\bigl(1+|x|^2+(t-L)^2\bigr)\,\frac{J\chi}{\Phi}
 =4(N-2)|x|^2+4(N-2)\left(t-\frac{2K+Ns}{N-2}\right)^2\ge0.
\]
Consequently
\begin{equation}\label{ord:eq:upperpositive}
 \theta\le\frac{w}{s(N-2)^2}
       \left(\frac{2(N-1)c_0}{N-5}+C\right).
\end{equation}
All these trials have finite energy for $N>6$.
\end{proof}

\subsubsection{The scale derivative of the quadratic energy}\label{ord:sec:actualscale}
\begin{proposition}[Positivity of the normal scale matrix]\label{prop:normal-scale-positive}
For $N=7,8$ and every fixed $L\in\R$, the matrix $P=-\mathsf DE-E\mathsf D$, where $E$ is given in Proposition~\ref{ord:prop:E} and $\mathsf D=\operatorname{diag}(1,2)$, is positive definite.
Its least eigenvalue has a positive lower bound when $L$ ranges over a compact interval.
\end{proposition}
\begin{proof}
Let $\mathsf D=\operatorname{diag}(1,2)$.
The scale matrix is
\[
 P=-\mathsf D E-E\mathsf D,
\]
because the two tensors have degrees one and two.
Its determinant is a concave quadratic polynomial in $\theta$; its quadratic coefficient is $-\mathbf d_1^2\mathbf d_2^2$.
Also $P_{11}>0$.
For $L\ge0$, this follows from $E_{11}=-Lw/8-\theta(N-2)^2/16$.
For $L=-K<0$, \eqref{ord:eq:lower} and
\[
 \int_K^\infty z(1+z^2)^{-\nu_0-1}\dd z
 =\frac{w}{N-3}
 >K\int_K^\infty(1+z^2)^{-\nu_0-1}\dd z
\]
again give $E_{11}<0$.
By Lemma~\ref{lem:normal-response-bounds}, it suffices to verify the determinant at the two bounds for $\theta$.

\medskip
\paragraph{Positivity of the determinants.}
For $N=7$, $L>0$, the upper endpoint simplifies to
\[
 \det P(\theta_+)=\frac{wc_0}{45L}
       \left(48M-\frac{37L}{1+L^2}\right)>0,
\]
since $M\ge\pi/4>3/4$ and $L/(1+L^2)\le1/2$.
For the lower endpoint, put $M=3/4+X$, $X\ge0$, and define
\[
 \Pi_7(L,X)=540Lc_0^4X^3+9c_0^2Q_2X^2
                    +\frac34c_0Q_1X+\frac1{16}Q_0.
\]
Each $Q_j$ has an explicitly positive decomposition:
\begin{align*}
 Q_2={}&L^3(135L^2-50L+5)+L(265L^2-18L+135)+11,\\
 Q_1={}&L^5(1215L^2-900L+200)+L^3(3445L^2-1224L+150)\\
 &+L(3095L^2-126L+1063)+198,\\
 Q_0={}&L^7(3645L^2-4050L+1200)+L^5(13380L^2-9558L+1800)\\
 &+L^3(16470L^2-6075L+600)+L(9012L^2-476L+2277)+891.
\end{align*}
Every displayed quadratic has positive leading coefficient and negative discriminant.
The exact determinant identity, including its positive prefactor, is
\begin{equation}
 \det P(\theta_-)
 =\frac{16\Pi_7(L,X)}
 {225c_0^2\bigl((12X+9)c_0^2+4L\bigr)^2}>0.
\end{equation}
Indeed $c_0=1+L^2>0$, the squared factor has a strictly positive base for $L>0$ and $X\ge0$, and the displayed decompositions give $\Pi_7(L,X)>0$.

For $N=8$, $L\ge0$, set $x=L/\sqrt{1+L^2}\in[0,1)$.
Then $M=(2+3x-x^3)/3$.
The determinant at \eqref{ord:eq:lower} is
\[
 \frac{5(1-x)^2(1+x)^3R_5(x)}{15552(3x^2-9x+8)^2},
\]
where
\[
 R_5=255x^5-1785x^4+5041x^3-5431x^2+3072.
\]
Its degree-five Bernstein coefficients on $[0,1]$ are
\[
 3072,\ 3072,\ 25289/10,\ 9734/5,\ 7364/5,\ 1152.
\]
For $L>0$, the determinant at the upper bound \eqref{ord:eq:uppernegative} is
\[
 \frac{25(1-x)^2(1+x)^3(7x^2-7x+12)}{972x}>0.
\]

For $N=8$ and $L=-K\le0$, use the improved lower bound obtained from the three scalar trials $\Phi,t\Phi,\bigl(1+|x|^2+(t-L)^2\bigr)\Phi$.
To specify it completely, set
\[
 \varpi_0=\frac1{2(N-1)(N-2)},\quad
 \varpi_1=\frac{N-3}{4N(N-1)(N-2)},\quad
 \varpi_{-1}=\frac1{(N-5)(N-2)},
\]
\[
 T_0=J_0,\qquad T_1=J_1+LJ_0,\qquad
 T_2=J_2+2LJ_1+L^2J_0.
\]
After division by $\mathfrak c_N$, the Gram matrix $G$ is symmetric, and its entries and the load vector $b_\partial$ are as follows:
\begin{align*}
 G_{11}&=2N\varpi_1T_0,\quad G_{12}=2N\varpi_1T_1,\quad G_{13}=2N\varpi_0M,\\
 G_{22}&=2N\varpi_1T_2+\varpi_0M,\quad
 G_{23}=4(N-2)\varpi_0(I_1+LM)-4N\varpi_1T_1,\\
 G_{33}&=4(N-2)\varpi_{-1}(I_2+M)-4N\varpi_0M+2L\varpi_{-1}c_0w.
\end{align*}
The load vector is $b_\partial=(\varpi_0w,0,\varpi_{-1}c_0w)^T$.
These entries follow by applying $J$ and $B_0$ to the three displayed trials and using the integrals in the proof of Proposition~\ref{ord:prop:E}.
Lemma~\ref{ord:lem:positive} makes $G$ positive definite.
The variational bound is $\theta\ge c_N^{-1} b_\partial^TG^{-1}b_\partial$.
With $x:=K/\sqrt{1+K^2}$, substitution of $N=8$, $L=-x/\sqrt{1-x^2}$, $w=(1-x^2)^{5/2}$, and $M=(2-3x+x^3)/3$ gives
\[
 \theta_{-,3}=
 \frac{(1-x)^2(1+x)^3(165x^3+645x^2+1028x+646)}
 {54(55x^3+270x^2+521x+378)}.
\]
The upper bound \eqref{ord:eq:upperpositive} is $\theta_+=(1-x^2)^2(x^2+20x+47)/216$.
Writing $Q_3=55x^3+270x^2+521x+378$, the lower determinant is
\[
 \frac{5(1-x)^3(1+x)^2}{324Q_3^2}
 (24200x^4+161315x^3+438445x^2+579997x+309771)>0.
\]
The upper determinant is
\[
 \frac{5(1-x)^3(1+x)^2}{746496}(-R_7(x)),
\]
where
\[
 R_7=5x^7+85x^6+870x^5+4150x^4+11205x^3-12107x^2-25136x-51264.
\]
The minimum degree-seven Bernstein coefficient of $-R_7$ on $[0,1]$ equals $51264$.
Formula~\eqref{ord:eq:bernsteinformula} gives this value and the remaining two-variable certificates.

Finally, for $N=7$, $L=-K\le0$, use $\upsilon=(K+\sqrt{1+K^2})^{-1}\in(0,1]$.
Then
\[
 L=\frac{\upsilon^2-1}{2\upsilon},\quad w=\frac{16\upsilon^4}{(1+\upsilon^2)^4},\quad
 M=\arctan \upsilon-\frac{\upsilon(1-\upsilon^2)}{(1+\upsilon^2)^2}.
\]
Bound $\arctan \upsilon$ below by its eight-term alternating series and above by adding $\upsilon^{17}/17$.
Write $M=M_-+Z(M_+-M_-)$, $0\le Z\le1$.
At the upper response endpoint,
\[
 \theta_+=\frac{16\upsilon^3(105\upsilon^4+320\upsilon^2+231)}{875(1+\upsilon^2)^5}.
\]
Here is a finite rational certificate that makes the numerator polynomials and their positivity independent of a supplementary program.
Put $v=\upsilon^2$ and define
\begin{align*}
 A(v,Z)&=\sum_{j=0}^7\frac{(-1)^jv^j}{2j+1}
                         +\frac{Zv^8}{17},\\
 B(v,Z)&=\frac{(1+v)^2A(v,Z)+v-1}{v},\\
 H(v,Z)&=3(1+v)^2B(v,Z)+8(v-1),\\
 C(v,Z)&=-16\{7(v-1)^3+44v(v-1)-15v(1+v)^2B(v,Z)\},\\
 T(v)&=105v^2+320v+231.
\end{align*}
The numerator defining $B$ has zero constant term, so $B$ is a polynomial.
Define further
\begin{align*}
 X_-&=2(v-1)H+64v,&
 Y_-&=8(v^2+1)H+128v(v-1),\\
 Z_-&=CH+2048v(v-1)^2,\\
 X_+&=70(v^2-1)+2T,&
 Y_+&=280(v^2+1)(1+v)+4(v-1)T,\\
 Z_+&=35(1+v)C+64(v-1)^2T.
\end{align*}
The two polynomials to be certified are
\[
 \mathcal P_-(v,Z)=\frac{X_-Z_--9Y_-^2}{v^2(1+v)^4},\qquad
 \mathcal P_+(v,Z)=X_+Z_+-9Y_+^2.
\]
Polynomial division in the first expression has zero remainder.
Their bidegrees are $(29,3)$ and $(15,1)$.
Direct substitution into the scale matrix gives the exact identities
\begin{equation}
 \begin{aligned}
 \det P(\theta_-)
   &=\frac{\upsilon^8\mathcal P_-(\upsilon^2,Z)}{225(1+\upsilon^2)^4H(\upsilon^2,Z)^2},\\
 \det P(\theta_+)
   &=\frac{\upsilon^4\mathcal P_+(\upsilon^2,Z)}{275625(1+\upsilon^2)^{10}}.
 \end{aligned}
\end{equation}
The finite coefficient calculation uses the following formula.
If $f(v,Z)=\sum c_{ab}v^aZ^b$ has bidegree $(d,e)$, its tensor-product Bernstein coefficients are
\begin{equation}\label{ord:eq:bernsteinformula}
 b_{ij}(f)=\sum_{a=0}^i\sum_{b=0}^j
 c_{ab}\frac{\binom ia}{\binom da}
        \frac{\binom jb}{\binom eb},
 \qquad 0\le i\le d,\quad0\le j\le e.
\end{equation}
Inserting the displayed rational polynomials gives the following exact certificate values (the one-variable entries are independent of $Z$):
\[
 \begin{array}{c|c|c|c}
 f&(d,e)&\hbox{number of coefficients}&\min_{i,j}b_{ij}(f)\\ \hline
 \mathcal P_-&(29,3)&120&49152/5\\
 \mathcal P_+&(15,1)&32&3599344\\
 R_5&(5,0)&6&1152\\
 -R_7&(7,0)&8&51264
 \end{array}
\]
Thus all four polynomials are strictly positive on $[0,1]^2$, since the Bernstein basis is nonnegative and sums to one.
Formula \eqref{ord:eq:bernsteinformula} with $e=0$ also gives the bound on $-R_7$ above.
All entries in this certificate are determined by finite rational sums and polynomial products displayed here.
The actual lower-end denominator is nonzero: indeed $3M+Lw=\upsilon^3H(\upsilon^2,Z)/(1+\upsilon^2)^4>0$ by the positive integral in \eqref{ord:eq:lower}.
This proves the endpoint signs.

At $L=0$, $\det P$ is increasing and linear in $\theta$, and the lower endpoint suffices.
Thus $P>0$ for both dimensions and every finite $L$.

Continuity of the response and matrix entries proves the compact-parameter assertion.
\end{proof}

\subsection{Tangential dependence of the quadratic metric terms}
\begin{proposition}[Tangential quadratic terms]
For $N=7,8$ and every fixed $L\in\R$, the trace-free hook, symmetric trace-free, and degree-one scalar components of $tC(x)$ are orthogonal for $F$.
The corrected energy is strictly negative on each nonzero component.
The estimates for the scalar components are \eqref{ord:eq:Sbound} and \eqref{ord:eq:Vbound}.
\end{proposition}
\begin{proof}
Put $V_b=C_{ab,a}$ and $K_C=C_{ab,c}C_{ac,b}$.
Substitution of $h=tC(x)$ into \eqref{lc:eq:rawsource}--\eqref{lc:eq:rawpolar} gives
\begin{equation}\label{ord:eq:Craw}
 \frac{\mathcal Q(tC)}{\mathfrak c_N}
 =\int_0^\infty c(t)^{-\nu_0}
 \left[t^2\left(\frac{|C|^2}{4}+\frac{K_C}{2}\right)
       -\frac{c(t)|C|^2}{4(N-5)}\right]\dd t.
\end{equation}
The curvature term here is $-\tfrac14(|C(x)|^2+t^2|C|^2)-\tfrac12t^2|V|^2$.
The kinetic term cancels the $|V|^2$ term and yields \eqref{ord:eq:Craw}.
The source is
\[
 S_{tC}=N(N-2)t\frac{C_{ab,c}x_ax_bx_c}{\bigl(1+|x|^2+(t-L)^2\bigr)^2}U
                -(N-2)t\frac{V\cdot x}{\bigl(1+|x|^2+(t-L)^2\bigr)}U.
\]

The orthogonal decomposition is into totally symmetric trace-free tensors, trace-free hooks, and the vector tensors
\[
 C^V_{ab,c}(v)=
 \frac{m}{(m-1)(m+2)}(\delta_{ac}v_b+\delta_{bc}v_a)
 -\frac{2}{(m-1)(m+2)}\delta_{ab}v_c.
\]
Here the trace-free hook is vector-derived, whereas the totally symmetric part and $C^V(v)$ are scalar-derived, with harmonic degrees three and one.
For this normalization, $C^V_{ab,a}=v_b$, $|C^V|^2=2m|v|^2/((m-1)(m+2))$, and $K_{C^V}=(m-2)|v|^2/((m-1)(m+2))$.

\medskip
\paragraph{Trace-free hook tensors.}
If $C$ is trace-free hook, its cyclic symmetrization and $V$ vanish, $K_C=-|C|^2/2$, and $S_{tC}=0$.
Therefore
\[
  F_{tC}=-\frac{\mathfrak c_N|C|^2}{4(N-5)}
           \int_0^\infty c(t)^{1-\nu_0}\dd t<0.
\]

\medskip
\paragraph{Totally symmetric trace-free tensors.}
Suppose $C$ is totally symmetric and trace-free.
Set $Q_3=C_{abc}x_ax_bx_c$, $\Phi_3=Q_3U/\bigl(1+|x|^2+(t-L)^2\bigr)$.
Then
\[
 J\Phi_3=4N\Phi_3/\bigl(1+|x|^2+(t-L)^2\bigr),\quad B_0\Phi_3=0,\quad
 v_p=-\frac{N-2}{12}(2t+L)\Phi_3,
\]
\[
 Jv_p=-S_{tC},\qquad B_0v_p=\frac{N-2}{6}\Phi_3.
\]
Apply the decomposition \eqref{lc:eq:schur} and the variational principle from Lemma~\ref{ord:lem:positive}, with trial function $\Phi_3$, to obtain
\begin{equation}\label{ord:eq:Sbound}
  F_{tC}\le
 -\frac{\mathfrak c_N|C|^2c_0w}{12(N-5)}
 \left[L+\frac{c_0w}{(N-5)M}\right].
\end{equation}
For example, the particular boundary contribution is $-\mathfrak c_N|C|^2Lc_0w/(12(N-5))$; the Rayleigh correction is at least $\mathfrak c_N|C|^2c_0^2w^2/(12(N-5)^2M)$.

\medskip
\paragraph{The scalar harmonic component of degree one.}
For $C=C^V(v)$, put $\Phi_1=(v\cdot x)U/\bigl(1+|x|^2+(t-L)^2\bigr)$.
This is a translation Jacobi field.
With $z=t-L$, define
\begin{align*}
 b(z)&=-\frac{N-3}{4(N-2)}z-\frac L2,\\
 g(z)&=\frac{z^3}{3}+\frac L2z^2+\frac{N-1}{2(N-2)}z
       -\frac{L^3}{6}+\frac{N-1}{2(N-2)}L,\\
 v_p&=\frac{N-2}{N+1}(g+b\bigl(1+|x|^2+(t-L)^2\bigr))\Phi_1.
\end{align*}
Then $Jv_p=-S_{tC}$ and
\[
 B_0v_p=\frac{N-3}{4(N+1)}\Xi,\qquad
 \Xi=\left(\bigl(1+|x|^2+(t-L)^2\bigr)-\frac{2(N-1)c_0}{N-3}\right)\Phi_1.
\]
The boundary functional $\Xi$ is orthogonal to the translation kernel.
Let $T_\Xi$ be its positive boundary response.
The exact particular coefficient is
\[
 -\frac{\mathfrak c_N|v|^2(N-1)^2Lc_0w}{4(N+1)(N-5)(N-2)^2},
\]
and the response is subtracted with multiplier $(N-1)(N-3)^2/[4(N-2)(N+1)^2]$.
The trials $\bigl(1+|x|^2+(t-L)^2\bigr)\Phi_1,t\Phi_1,t^2\Phi_1$, projected off the kernel if necessary, give
\[
 \frac{T_\Xi}{\mathfrak c_N|v|^2}\ge
 \frac{c_0^2w^2(N-1)(N+1)}{M(N-5)^2(N-3)^2(N-2)}.
\]
For a complete inverse calculation, put $\varpi=(N-3)/(4(N-1)(N-2))$.
Their Gram matrix $G_V$ and load $b_V$, normalized by $\mathfrak c_N|v|^2$, are
\begin{align*}
 (G_V)_{11}&=\frac{(N+2)M-Lw}{(N-1)(N-2)},&
 (G_V)_{12}&=2\varpi J_1,\\
 (G_V)_{13}&=4\varpi(J_2+LJ_1),&
 (G_V)_{22}&=\varpi J_0,\\
 (G_V)_{23}&=2\varpi(J_1+LJ_0),&
 (G_V)_{33}&=4\varpi(J_2+2LJ_1+L^2J_0),
\end{align*}
with the remaining entries determined by symmetry, and
\[
 b_V=\left(\frac{c_0w(N+1)}{(N-5)(N-3)(N-2)},0,0\right)^T.
\]
The entries follow from $\B(f\Phi_1,g\Phi_1)=\int_{\R^N_+}\Phi_1^2\nabla f\cdot\nabla g$.
The three trials are linearly independent modulo the translation kernel, so $G_V$ is positive definite.
Solving $G_Vz=b_V$ and using the recurrence relations for $I_j$ and $J_j$ gives
\[
 b_V^TG_V^{-1}b_V
 =\frac{c_0^2w^2(N-1)(N+1)}{M(N-5)^2(N-3)^2(N-2)}.
\]
This proves the displayed response lower bound.
Hence
\begin{equation}\label{ord:eq:Vbound}
  F_{tC^V(v)}\le
 -\frac{\mathfrak c_N|v|^2(N-1)^2c_0w}{4(N+1)(N-5)(N-2)^2}
 \left[L+\frac{c_0w}{(N-5)M}\right].
\end{equation}

The bracket $L+c_0w/((N-5)M)$ is positive.
This is immediate for $L\ge0$.
For $L=-K<0$, use
\[
 \frac{c_0w}{N-5}
 =\int_K^\infty z(1+z^2)^{-\nu_0}\dd z
 >K\int_K^\infty(1+z^2)^{-\nu_0}\dd z=KM.
\]
\end{proof}

\begin{proof}[Proof of Theorem~\ref{ord:thm:main}]
Tangential inversion separates $tC(x)$ from the even tensors $H^W,tA,t^2B$.
The Weyl representation and $\operatorname{Sym}^2_0$ representation are orthogonal for any invariant quadratic form.
For $H^W\ne0$, its scalar source vanishes and
\[
  F_{H^W}(0,1)=-\frac14\int_{\R^N_+} U^2|\nabla H^W|^2<0.
\]
Within the $C$ space the three irreducible sectors are orthogonal.
The scalar-derived sectors have source harmonic degrees three and one, while the hook sector has zero source; hence the corrected mixed pairings also vanish.
All these degree-two sectors have scale derivative $4 F$.
For the remaining block $tA,t^2B$, Proposition~\ref{prop:normal-scale-positive} gives a positive scale matrix.
Finite-dimensional norm equivalence proves Theorem~\ref{ord:thm:main}.
Continuity of the scalar inverse and the coefficients for finite $L$ yields compact-parameter uniformity.
\end{proof}

\begin{corollary}\label{ord:cor:degreeone}
For $N=5,6$ and each fixed $\kappa\in\R$, the model $h=tA$, with $A\in\operatorname{Sym}^2_0(\R^{N-1})$, satisfies
\[
 -\varepsilon\partial_\varepsilon F_{tA}(0,\varepsilon)
 \ge c_{N,\kappa}\varepsilon^2|A|^2.
\]
The positive constant is uniform for $\kappa$ in a compact interval.
\end{corollary}
\begin{proof}
The formula $F_{tA}/(\mathfrak c_N|A|^2)=-Lw/8-\theta(N-2)^2/16$ and the trial \eqref{ord:eq:lower} remain valid in these dimensions when used on the degree-one sector alone.
The particular solution is a linear multiplier of $\Phi$ and has finite Dirichlet energy for $N>4$; all integrals needed for this formula converge.
None of the degree-two integrals with a factor $(N-5)^{-1}$ is used.
The integral inequality in Section~\ref{ord:sec:actualscale} gives strict negativity of the displayed coefficient also when $L<0$.
Homogeneity gives the asserted factor $\varepsilon^2$, and continuity of the scalar inverse gives compact-parameter uniformity.
\end{proof}

\section{The logarithmic term in the energy expansion}\label{sec:critical}
The ordinary integrals cease to converge when $2d=N-2$.
For a cutoff tensor of this degree, we compute the coefficient of $\log R$ and then combine it with the ordinary inequalities of Theorems~\ref{thm:ordinary-endpoints} and~\ref{ord:thm:main}, together with Corollary~\ref{ord:cor:degreeone}.
The proof has three steps: solve the leading scalar equation at infinity, use Kelvin inversion and a boundary-preserving gauge to identify the quadratic coefficient, and exclude the kernel using the polynomial Fermi conditions.
The final cutoff estimates make this calculation uniform for moving coefficient vectors.

\subsection{The linearized equation at the critical degree}
Use the full Fermi spaces in \eqref{crit:eq:fermi}.
Use the cutoff $\chi_R$ fixed in Section~\ref{sec:preliminaries}.
\begin{theorem}[The logarithmic coefficient]\label{crit:thm:critical}
Let $N=2d+2$, $d\ge1$.
There is a linear map
\[
 T_d:\mathscr J_d\rightarrow C^\infty(S^{N-1}_+;\operatorname{Sym}^2\R^N)
\]
which is injective and independent of $(\sigma,\kappa)$, such that
\begin{equation}\label{crit:eq:residue}
 \lim_{R\to\infty}\frac{F(\chi_Rh)}{\log R}
 =-\frac{2^{\sigma(N-2)}}4\int_{S^{N-1}_+}
       \bigl(|\nabla_S T_dh|^2+d^2|T_dh|^2\bigr)\dd\theta.
\end{equation}
Here derivatives of $T_dh$ are taken componentwise in a fixed Cartesian frame.
In particular, the right-hand side is at most $-c(N)\norm h^2$ for some $c(N)>0$.
For a homogeneous scaled coefficient and a differentiated cutoff, the logarithmic residue of the negative scale derivative is $-2d$ times \eqref{crit:eq:residue}, and hence is strictly positive on $\mathscr J_d\setminus\{0\}$.
\end{theorem}

We first construct the homogeneous scalar correction and the tensor $T_dh$ used in the theorem.
First divide out the leading coefficient $2^{\sigma(N-2)/2}$ of $U$.
Set
\[
 N=2d+2,\qquad w(y)=|y|^{2-N}=r^{-2d}.
\]
For a trace-free tangential tensor $h$, the first scalar source is
\begin{equation}\label{crit:eq:source}
 S_h=h_{ij}\partial_{ij}w+(\partial_i h_{ij})\partial_jw
                    +c_N(\partial_i\partial_jh_{ij})w.
\end{equation}
The sign convention is that the scalar corrector satisfies $-\Delta v=-S_h$.
Because $h$ has degree $d$,
\[
 S_h=r^{-d-2}s_h(\theta).
\]
The linearized boundary mean curvature of $\delta+s h$ is zero: its normal components vanish and $\partial_t\tr h=0$.
Also $\partial_\nu w=0$ on the flat boundary away from the origin.
Therefore the leading scalar boundary condition is Neumann.

\begin{lemma}\label{crit:lem:cone}
There is a unique homogeneous correction $v=r^{-d}\phi_h(\theta)$ satisfying
\[
 \begin{cases}
 -\Delta v=-S_h &\text{in }\R^N_+,\\
 \partial_\nu v=0 &\text{on }\partial\R^N_+\setminus\{0\},
 \end{cases}
\]
and its angular factor solves
\begin{equation}\label{crit:eq:angular}
 \begin{cases}
 (-\Delta_{S^{N-1}}+d^2)\phi_h=-s_h &\text{in }S^{N-1}_+,\\
 \partial_{\nu_S}\phi_h=0 &\text{on }\partial S^{N-1}_+.
 \end{cases}
\end{equation}
\end{lemma}
\begin{proof}
For any real exponent $a$, $\Delta(r^a\phi)=r^{a-2}(\Delta_S\phi+a(a+N-2)\phi)$.
For $a=-d$ and $N-2=2d$, the second coefficient is $-d^2$.
The Neumann quadratic form of the operator in \eqref{crit:eq:angular} is
\[
 \int_{S^{N-1}_+}(|\nabla_S\phi|^2+d^2\phi^2),
\]
which is strictly positive.
Elliptic regularity gives smoothness up to the equator.
Homogeneity turns the Euclidean Neumann condition into the stated spherical condition.
\end{proof}

\subsubsection{The leading term at infinity}
\begin{lemma}\label{lem:critical-leading-term}
Replacing the bubble by its leading homogeneous term changes the critical quadratic densities only by integrable terms.
After division by $2^{\sigma(N-2)}$, the resulting logarithmic coefficient is independent of the curvature pair.
\end{lemma}
\begin{proof}
For either family in \eqref{eq:bubble},
\[
 U=2^{\sigma(N-2)/2} r^{-2d}(1+O_\kappa(r^{-1})),
\]
with the corresponding differentiated estimates.
The interior Jacobi potential is $O_\kappa(r^{-4})$, and its boundary Robin coefficient is $O_\kappa(r^{-2})$.
For $v=r^{-d}\phi$, these potential energies have radial orders $r^{-3}\dd r$ in the interior and $r^{-2}\dd r$ on the boundary.
Neither contributes a logarithm.
A factor $O_\kappa(r^{-1})$ in a critical quadratic metric density changes $r^{-1}\dd r$ to an integrable density.
Thus, after division by $2^{\sigma(N-2)}$, the logarithmic coefficient is independent of the curvature pair.
These bounds are uniform on compact parameter sets on which the bubble denominator remains positive and the Jacobi inverse is uniformly bounded.
\end{proof}

\subsection{Kelvin inversion and the second variation}
Let $I(z)=z/|z|^2$ and write
\[
 w_s=w+s v,
 \qquad \bar g_s=w_s^{4/(N-2)}(\delta+s h+O(s^2)).
\]
The metric $I^*\bar g_0$ is Euclidean.
Denote its first variation by $k$:
\[
 I^*\bar g_s=\delta+s k+O(s^2).
\]
It is homogeneous of degree $-d$.
By conformal covariance and Lemma~\ref{crit:lem:cone},
\begin{equation}\label{crit:eq:constraints}
 R'_\delta(k)=0,\qquad H'_\delta(k)=0.
\end{equation}
At the flat boundary, $k_{aN}=0$.
This follows because inversion preserves the boundary plane and its differential preserves the splitting into normal and tangential directions there.

Seek a vector field $V$ of degree $1-d$ satisfying
\begin{equation}\label{crit:eq:gauge}
 \begin{cases}
 \Delta V_j=\partial_i k_{ij}-\tfrac12\partial_j\tr k &\text{in }\R^N_+,\\
 V_N=0,\quad\partial_NV_a=0 &\text{on }\partial\R^N_+\setminus\{0\}.
 \end{cases}
\end{equation}
Then $K=k-\mathcal L_V\delta$ satisfies $\partial_iK_{ij}-\tfrac12\partial_j\tr K=0$ and $K_{aN}=0$ on the boundary.
Since $V_N=0$ on the plane, the local flow of $V$ preserves the boundary; no normal displacement of the boundary is allowed.

\begin{lemma}\label{crit:lem:gauge}
Equation \eqref{crit:eq:gauge} is solvable for every $d\ge2$.
For $d=1$ and $h\in\mathscr J_1$, it is solvable modulo tangential constant vector fields.
\end{lemma}
\begin{proof}
The angular operator for $V=r^{1-d}v(\theta)$ is $-\Delta_S+d^2-1$.
For tangential Cartesian components, use Neumann conditions; for the normal component, use Dirichlet conditions.
If $d\geq2$, both operators are strictly positive.
If $d=1$, the sole possible obstruction is the constant Neumann mode of a tangential component.
In this case \eqref{crit:eq:fermi} implies $h=tA$, $A$ constant tangential and trace free.
Under $x\mapsto-x$, the corrected inverted tensor has even $k_{ab},k_{NN}$ and odd $k_{aN}$.
Consequently $\partial_i k_{ia}-\tfrac12\partial_a\tr k$ is odd and has zero hemispherical mean.
The normal Dirichlet problem has no kernel.
Tangential constants do not affect $\mathcal L_V\delta$.
\end{proof}

\begin{lemma}\label{crit:lem:tt}
The tensor $K$ satisfies
\[
 \tr K=0,\qquad\partial_iK_{ij}=0,
\]
and on the boundary
\begin{equation}
 K_{aN}=0,\qquad \partial_\nu K_{NN}=0,
 \qquad\partial_\nu\tr_TK=0.
\end{equation}
\end{lemma}
\begin{proof}
The constraints \eqref{crit:eq:constraints} are invariant under a boundary-preserving infinitesimal diffeomorphism.
In Bianchi gauge, $0=R'(K)=-\tfrac12\Delta\tr K$.
Since $K_{aN}=0$ on the boundary, $H'(K)=0$ gives $\partial_\nu\tr_TK=0$.
The normal Bianchi equation gives $\partial_\nu K_{NN}=\partial_\nu\tr_TK=0$.
Thus $\tr K$ is homogeneous of degree $-d$, harmonic, and Neumann.
Its angular factor is annihilated by $-\Delta_S+d^2$ and is zero.
Bianchi gauge now gives divergence freeness.
\end{proof}
Define $T_dh=K|_{S^{N-1}_+}$.
This constructs the linear map on $\mathscr J_d$.
Its injectivity uses the last normal-form condition in \eqref{crit:eq:fermi} and is proved below.

\medskip
\paragraph{The scalar-curvature functional with boundary term.}
For the tensor in Lemma~\ref{crit:lem:tt}, we use the second variation of $\mathcal S_\Omega$ defined in Section~\ref{sec:action}.
Use half-annuli $\Omega_R=\{R^{-1}<|z|<1,\ z_N>0\}$, and define $\mathcal S_{\Omega_R}$ by integrating over its three smooth boundary faces.
There is no corner-angle term in this definition.
To avoid a corner variation convention, the proof below first localizes the metric variation on the smooth half-space, then compares the two functionals.
Spherical faces and endpoint collars are retained as bounded terms.

\begin{lemma}[Second variation at the Euclidean metric]\label{crit:lem:EH}
Let $K$ be homogeneous of degree $-d$, $N=2d+2$, and satisfy Lemma~\ref{crit:lem:tt}.
Then
\begin{equation}\label{crit:eq:EH}
 [s^2]\mathcal S_{\Omega_R}(\delta+sK+O(s^2))
 =-\frac14\int_{\Omega_R}|\partial K|^2\dd z+O(1),
\end{equation}
provided the second-order path term has its natural degree $-2d$.
The $O(1)$ may depend on the fixed angular data but not on $R$.
\end{lemma}
\begin{proof}
Apply the smooth-boundary first-variation formula \eqref{crit:eq:EHfirst} to a compactly supported cutoff of the variation.
Write the second-order path term in the statement as $s^2M$, where $M$ has degree $-2d$.
Choose a smooth radial $\varpi_R$ equal to one on $R^{-1}\leq r\leq1$, supported in $(2R)^{-1}<r<2$, with fixed rescaled profiles on its two transition collars.
Apply \eqref{crit:eq:EHfirst} to
\[
 g_{s,R}=\delta+sL_R+s^2M_R,\qquad
 L_R=\varpi_RK,\quad M_R=\varpi_R^2M.
\]
All variations are compactly supported away from the origin.
The interior and boundary first-variation tensors vanish at the flat metric, so $D\mathcal S_\delta[M_R]=0$ exactly.
Differentiating \eqref{crit:eq:EHfirst} therefore gives
\[
 \left.\partial_s^2\mathcal S(g_{s,R})\right|_{s=0} ={}-\int_{\R^N_+} \langle\Ric'(L_R)-\tfrac12R'(L_R)\delta,L_R\rangle -\int_{\partial\R^N_+} \langle\mathrm{II}'(L_R)-H'_{\tr}(L_R)\delta^T,L_R^T\rangle.
\]
On the central half-annulus $L_R=K$, and the trace-free and divergence-free conditions imply $\Ric'(K)=-\tfrac12\Delta K$, $R'(K)=0$.
On its planar face, $K_{aN}=0$ gives $\mathrm{II}'_{ab}(K)=\tfrac12\partial_\nu K_{ab}$ and $H'_{\tr}(K)=0$.
Integration by parts on the central half-annulus gives the bulk contribution
\[
 -\tfrac12\int_{\Omega_R}|\partial K|^2
 +\tfrac12\int_{\partial\Omega_R}K_{ij}\partial_\nu K_{ij}.
\]
On the plane, the tangential part of the last integral cancels the integral involving the boundary first-variation tensor.
Its mixed part is zero because $K_{aN}=0$; its normal part is zero because $\partial_\nu K_{NN}=0$.

For completeness, every discarded piece is an endpoint term of bounded size.
On a transition collar of radius $a\in\{R^{-1},1\}$,
\[
 |\partial^jL_R|\leq Ca^{-d-j}\ (0\leq j\leq2),\qquad
 |\partial^jM_R|\leq Ca^{-2d-j}\ (0\leq j\leq2).
\]
The bulk second-variation densities are bounded by $C(|L_R||\partial^2L_R|+|\partial L_R|^2)$, of size $Ca^{-2d-2}$ on a set of volume $O(a^N)$.
The planar densities have size $Ca^{-2d-1}$ on a set of area $O(a^{N-1})$.
Both integrals are $O(a^{N-2d-2})=O(1)$.
The same estimates, with $|\partial^2M_R|$ in the bulk and $|\partial M_R|$ on the plane, bound the quadratic coefficient when comparing the localized energy with the central energy.
On a spherical face of radius $a$, the quadratic density of the boundary mean-curvature term and the displayed integration-by-parts flux have size $Ca^{-2d-1}$ and area $O(a^{N-1})$, hence are also $O(1)$.

Thus the coefficient of the compactly supported half-space energy differs from $[s^2]\mathcal S_{\Omega_R}(\delta+sK+s^2M)$ by $O(1)$, and all the cutoff errors in its second variation are $O(1)$ as well.
No first-variation formula on a corner has been used.
Dividing the second derivative by two proves \eqref{crit:eq:EH}.
\end{proof}
Since $K=r^{-d}T_dh$,
\begin{equation}\label{crit:eq:dirichlet}
 \int_{\Omega_R}|\partial K|^2
 =\log R\int_{S^{N-1}_+}
       \bigl(|\nabla_ST_dh|^2+d^2|T_dh|^2\bigr).
\end{equation}

\subsubsection{Invariance of the logarithmic coefficient}
\begin{lemma}\label{lem:critical-invariance}
Kelvin inversion and the boundary-preserving change generated by $V$ do not change the corrected logarithmic coefficient.
For the homogeneous correction, that coefficient is the right-hand side of \eqref{crit:eq:residue}.
\end{lemma}
\begin{proof}
The conformal transformation law gives the exact identity
\[
 \mathcal S_\Omega(u^{4/(N-2)}g)
 =\int_\Omega(c_N^{-1}|\nabla u|_g^2+R_gu^2)\dd v_g
       +2\int_{\partial\Omega}H_{\tr,g}u^2\dd\sigma_g.
\]
When only the flat part of the boundary is included in the original model, the difference is a spherical endpoint term.
Its quadratic critical coefficient is bounded.
Kelvin inversion is an exact change of variables.
The diffeomorphism generated by $V$ preserves the flat boundary.
Moving the spherical boundary by that diffeomorphism changes the quadratic energy by bounded endpoint terms, again by critical homogeneity.
Hence neither operation changes the logarithmic coefficient.
The volume and boundary-area multiplier terms in $\mathcal A$ are integrable at the critical end and do not contribute a logarithm.
Restoring the asymptotic factor $U\sim 2^{\sigma(N-2)/2}w$ multiplies the quadratic energy by $2^{\sigma(N-2)}$.
Equations \eqref{crit:eq:EH} and \eqref{crit:eq:dirichlet} now give the right-hand side of \eqref{crit:eq:residue}.
\end{proof}

\subsection{Strict negativity of the logarithmic coefficient}
\begin{lemma}\label{lc:lem:kernel}
The map $T_d$ in \eqref{crit:eq:residue} is injective for every $d\ge1$.
\end{lemma}
\begin{proof}
If $T_dh=0$, reversing inversion and the scalar correction gives
\[
 h=\mathcal L_X\delta+f\delta,
\]
where $X$ is homogeneous of degree $d+1$, $f$ of degree $d$, and $X_N(x,0)=0$.
Initially they are smooth only away from zero.
They are in fact polynomial.
Indeed, for the Schouten tensor $P_g$,
\[
 P'_\delta(h)=-\tfrac12\nabla^2 f.
\]
The left side is a polynomial of degree $d-2$ (zero when $d=1$), so $f$ is polynomial up to an affine term; homogeneity preserves polynomiality.
The identity $2\partial_i\partial_jX_k=\partial_i(h-f\delta)_{jk}+\partial_j(h-f\delta)_{ik}-\partial_k(h-f\delta)_{ij}$ then shows that $X$ is polynomial too.

Tangential trace and Fermi normal components imply
\[
 f=-2\partial_tX_N,\quad\partial_tX_a=-\partial_aX_N,\quad
 \Div_xX^T=m\partial_tX_N,\quad
 \Delta_xX_N+m\partial_{tt}X_N=0.
\]
On the boundary, let $\Psi=x\cdot X^T$, a polynomial of degree $d+2$.
The radial condition gives
\[
 dX^T+\nabla_x\Psi+fx=0.
\]
Contract with $x$, then substitute in $2\Div_xX^T+mf=0$.
Writing $\Psi=r^{d+2}\psi$ yields
\begin{equation}
 -\Delta_{S^{m-1}}\psi=d^2(m-1)\psi=2d^3\psi,
 \qquad m=2d+1.
\end{equation}
But a polynomial of degree $d+2$ restricts to spherical harmonics of degrees at most $d+2$, whose eigenvalues are at most
\[
 (d+2)(d+m)=(d+2)(3d+1).
\]
For $d\ge3$,
\[
 2d^3-(d+2)(3d+1)=(d+1)(2d^2-5d-2)>0.
\]
Hence $\psi=0$.
When $d=2$, the possible polynomial harmonic degrees are $0,2,4$ with eigenvalues $0,10,28$, not the required $16$.
Thus again $\Psi=0$, and the boundary formulas give $X^T=f=0$ there.
Write $X_N(x,t)=\sum_jt^ja_j(x)$.
Its equation and zero Dirichlet and Neumann data imply
\[
 a_0=a_1=0,\qquad
 a_{j+2}=-\frac{\Delta_xa_j}{m(j+2)(j+1)}.
\]
Thus $X_N=0$.
Now $\partial_tX_a=0$ and the zero tangential boundary data imply $X=0$, hence $h=0$.

For $d=1$, the surviving boundary degree-one harmonics are generated by
\[
 X(y)=|y|^2(b,0)-2(b\cdot x)y,\qquad f=4b\cdot x.
\]
This is a conformal Killing variation with $\mathcal L_X\delta+f\delta=0$.
Subtract it, and the same zero-data argument gives $h=0$.
\end{proof}
\subsection{Comparison with the exact scalar correction}
\begin{proof}[Proof of Theorem~\ref{crit:thm:critical}]
Lemmas~\ref{crit:lem:cone}, \ref{crit:lem:gauge}, and~\ref{crit:lem:tt} construct $T_dh$.
Lemmas~\ref{lem:critical-leading-term}, \ref{crit:lem:EH}, and~\ref{lem:critical-invariance} identify its homogeneous logarithmic coefficient, and Lemma~\ref{lc:lem:kernel} proves injectivity.
By finite-dimensional norm equivalence, the right-hand side of \eqref{crit:eq:residue} is at most $-c\|h\|^2$.
It remains to compare the homogeneous correction with the exact Jacobi correction and to control cutoff differentiation.

The comparison requires a uniform bound on both residuals.
Restore the coefficient $2^{\sigma(N-2)/2}$ of \eqref{eq:bubble}, and choose a fixed radial $\zeta$ equal to zero for $r\leq1$ and one for $r\geq2$.
Set
\[
 \widetilde v_R:=2^{\sigma(N-2)/2}\zeta\chi_Rr^{-d}\phi_h.
\]
The Jacobi interior and boundary operators are
\[
 \mathcal L:=-\Delta-\frac{\sigma N(N+2)}4U^{4/(N-2)},\qquad
 B_J:=\partial_\nu-\frac{N\kappa}2U^{2/(N-2)}.
\]
Write $S_a^U$ for \eqref{crit:eq:source} with $h,w$ replaced by $a,U$.
For $a=\chi_Rh$, its first-order metric boundary source is exactly zero: $a_{Ni}=0$ and $\tr a=0$ give both zero normal variation and zero linearized mean curvature.
Hence the residual of $\widetilde v_R$ is
\[
 f_R=\mathcal L\widetilde v_R+S_{\chi_Rh}^U,
 \qquad g_R=B_J\widetilde v_R.
\]
Constants below include $\norm h$ and may depend on a compact interval of $\kappa$ values.
The differentiated expansion $\partial^j(U-2^{\sigma(N-2)/2}w)=O_\kappa(r^{-2d-1-j})$, $0\leq j\leq2$, and the cone equation imply
\begin{equation}
 \begin{aligned}
 |f_R|&\leq C\mathbf 1_{\{r<2\}}
       +Cr^{-d-3}\mathbf 1_{\{2<r<2R\}}
       +CR^{-d-2}\mathbf 1_{\{R<r<2R\}},\\
 |g_R|&\leq C\mathbf 1_{\{r<2\}}
       +Cr^{-d-2}\mathbf 1_{\{2<r<2R\}}\quad(t=0).
 \end{aligned}
\end{equation}
For the boundary estimate, radial cutoffs have zero normal derivative on the plane and the conical factor is Neumann, so $g_R=-\tfrac{N\kappa}2U^{2/(N-2)}\widetilde v_R$ exactly.
For the interior estimate, the potential term has the faster order $r^{-d-4}$; the $r^{-d-3}$ term is the bubble--cone source difference.
Derivatives of $\chi_R$ in both the scalar and metric source terms have annular order $R^{-d-2}$.

Let $p_*=2N/(N+2)$ and $q_*=2(N-1)/N$, the dual bulk and trace Sobolev exponents.
The annular bulk norm is bounded since $R^{-d-2}R^{N/p_*}=1$.
The remaining tails are uniformly integrable:
\[
 (d+3)p_*>N,\qquad (d+2)q_*>N-1.
\]
Thus
\[
 \norm{f_R}_{L^{p_*}(\R^N_+)}+
 \norm{g_R}_{L^{q_*}(\partial\R^N_+)}\leq C.
\]
The Sobolev inequality and the trace inequality \cite{EscobarTrace} bound the associated functional on $\dot H^1$ uniformly in $R$.
Meanwhile $\norm{\widetilde v_R}_{\dot H^1}=O(\sqrt{\log R}+1)$.
Projection onto the chosen complement subtracts bounded multiples of Jacobi fields: the compact tests defining that projection see a uniformly bounded function on their fixed support.
It does not alter the residual.
Applying the fixed projected inverse now shows that the exact scalar corrector differs from this projected conical cutoff by a function with uniformly bounded Dirichlet norm.
Consequently
\begin{equation}\label{crit:eq:cutasympt}
  F(\chi_Rh)=\mathfrak r(h)\log R+O_\kappa(\sqrt{\log R}+1),
\end{equation}
where $\mathfrak r(h)$ is the right-hand side of \eqref{crit:eq:residue}.

There is also the useful estimate
\begin{equation}\label{crit:eq:cutderivative}
 R\partial_R F(\chi_Rh)=O_\kappa(1).
\end{equation}
The differentiated metric source is supported in the cutoff annulus and has bulk size $O(R^{-d-2})$ and bounded dual norm; its boundary source is still zero.
Its pairing with the conical cutoff is bounded because
\[
 R^{-d-2}R^{-d}R^N=1.
\]
Pairing it with the bounded-norm correction is bounded by duality.
Self-adjointness of the projected Jacobi inverse expresses the derivative of the scalar-correction quadratic term by this source pairing, so no estimate on the derivative of the corrector is being assumed.
The raw quadratic derivative is bounded by the same annular calculation.
This proves \eqref{crit:eq:cutderivative} without differentiating an unspecified error in \eqref{crit:eq:cutasympt}.

For a lower-degree tensor $h_j$, $j<d$, its mixed pairing with $h_d$ is bounded as $R\to\infty$.
Pair the lower source with the conical correction: the radial exponent is $j+d+1-N<-1$.
The remaining scalar pairing is bounded by the bounded inverse and the ordinary lower-degree source norm.
More explicitly, the differentiated degree-$j$ source has annular dual norm $O(R^{j-d})$, while the differentiated critical source has bounded dual norm.
Pair these, respectively, with the conical-plus-bounded critical corrector and the bounded-norm lower-degree corrector.
The conical annular pairing has order $R^{j-d}$; all other pairings are bounded.
The raw pairing and its cutoff derivative satisfy the same bounds.
\end{proof}

\subsection{Cutoff estimates and coercivity of the scale derivative}
\begin{lemma}[Convergence of the cutoff quadratic forms]\label{ord:lem:cutoff}
Let $h_1,\ldots,h_b$ be ordinary homogeneous Fermi tensors of degrees $d_i\le d_N^{\rm ord}$.
Let $E_R$ and $E_\infty$ be the matrices of $F(\chi_R\sum_i a_ih_i)$ and $F(\sum_i a_ih_i)$, respectively.
Both are unnormalized; the positive factors divided out in Section~\ref{sec:ordinary-endpoints} are restored here.
Then
\[
 \|E_R-E_\infty\|+\|R\partial_RE_R\|\le C R^{\max_i d_i+1-N/2}.
\]
The constants are uniform on compact parameter sets with a bounded projected Jacobi inverse.
\end{lemma}
\begin{proof}
The weighted norm \eqref{eq:Nnorm} satisfies
\[
 \N((1-\chi_R)h_i)+\N((R\partial_R\chi_R)h_i)
 \le C R^{(2d_i+2-N)/2},\qquad \N(\chi_Rh_i)\le C.
\]
Indeed, its exterior radial integral is bounded by $C\int_R^\infty r^{2d_i+1-N}\dd r$; cutoff derivatives have the same annular bound.
Apply the bilinear estimate \eqref{eq:quadratic-bounded}.
Differentiation gives $R\partial_R F(\chi_Rh)=2 F(\chi_Rh,(R\partial_R\chi_R)h)$.
This proves both estimates, including all divergence terms created by the cutoff and the full scalar inverse.
\end{proof}

Let $\mathsf D$ be the diagonal matrix of retained homogeneous degrees.
Under dilation the coefficient vector has entries $\varepsilon^dJ_d$, and $R=\rho/\varepsilon$.
Thus the actual scale matrix is
\begin{equation}
 P_R=-\mathsf D E_R-E_R\mathsf D+R\partial_RE_R.
\end{equation}
If $\mathcal P_R(H)$ denotes its quadratic form, then
\begin{equation}
 \mathcal P_R\left(\sum_d\varepsilon^dJ_d\right)
 =-\varepsilon\partial_\varepsilon
  F\left(\chi_{\rho/\varepsilon}\sum_d\varepsilon^dJ_d\right),
\end{equation}
where $\rho$ and $J_d$ are held fixed.
\begin{theorem}[Coercivity of the scale derivative]
\label{thm:localcoercivity}
For every dimension, boundary class, and curvature pair in Theorem~\ref{thm:main}, there are $c>0$ and $R_0$ such that, for $R=\rho/\varepsilon\ge R_0$,
\begin{equation}\label{eq:Qcoercive}
 -\varepsilon\partial_\varepsilon
  F\left(\chi_R\sum_{d=1}^{d_N}\varepsilon^dJ_d\right)
 \ge c\mathscr N(\varepsilon,\rho),
\end{equation}
where
\begin{equation}\label{eq:Jnorm}
 \mathscr N(\varepsilon,\rho):=
 \sum_{2d<N-2}\varepsilon^{2d}\|J_d\|^2
 +\mathbf 1_{\{N\text{ even}\}}\varepsilon^{N-2}
            \log(\rho/\varepsilon)\|J_{d_N}\|^2.
\end{equation}
The coefficient vectors are arbitrary in the full Fermi spaces, or their umbilic subspaces as applicable.
In dimension three, the sum is empty and $\mathscr N=0$.
Constants are uniform on compact subsets of the stated parameter ranges.
\end{theorem}
\begin{proof}
For the two endpoint equations, use Theorem~\ref{thm:ordinary-endpoints}.
For arbitrary finite $\kappa$ in dimensions five through eight, use Theorem~\ref{ord:thm:main} and Corollary~\ref{ord:cor:degreeone}.
Lemma~\ref{ord:lem:cutoff} gives convergence of all ordinary scale matrices.
In even dimensions, Theorem~\ref{crit:thm:critical} supplies the full positive critical logarithmic block.
In an even dimension, the critical calculation and the ordinary cutoff lemma give
\[
 P_R=\begin{pmatrix}
 P_{\rm ord}+o(1)&O(1)\\
 O(1)&A\log R+O(\sqrt{\log R}+1)
 \end{pmatrix},\qquad P_{\rm ord}>0,\quad A>0.
\]
An absent ordinary block is omitted.

If the ordinary and critical lower bounds are $a|X|^2$ and $b\log R\,|Y|^2$, and the cross-block norm is at most $M$, use $2M|X||Y|\le a|X|^2/4+4M^2|Y|^2/a$.
For large $R$, the logarithmic block absorbs the last term and its $O(\sqrt{\log R}+1)$ remainder.
Dimension four has only the critical degree-one block, and odd dimensions have only ordinary blocks.
These arguments apply uniformly to moving coefficient vectors.
\end{proof}

\begin{proposition}\label{prop:local-continuation}
The inequality \eqref{eq:Qcoercive} also holds in the parameter ranges of Corollary~\ref{thm:continuation}.
\end{proposition}
\begin{proof}
At fixed dimension, there are finitely many ordinary blocks, all strictly positive at the relevant endpoint.
Near $(1,0)$, use the family \eqref{eq:bubble}.
Near $(0,2)$, use, for $\sigma\ge0$ small,
\[
 U_\sigma(x,t):=\bigl(\sigma/4+|x|^2+(t+1)^2\bigr)^{-(N-2)/2}.
\]
This family solves \eqref{eq:PDE} on the Euclidean half-space with curvature pair $(\sigma,2)$ and converges to the scalar-flat bubble as $\sigma\to0$.
The source maps and the projected Jacobi inverse are continuous in operator norm, after identifying the complements by \eqref{eq:dual}.
The integrals and their first two scale derivatives are dominated by the same integrable powers in the ordinary range.
Hence the positive ordinary gaps persist in both parameter neighborhoods.
The proof of Theorem~\ref{crit:thm:critical} applies to both families.
For $U_\sigma$, the expansion $U_\sigma(y)=|y|^{2-N}+O(|y|^{1-N})$ as $|y|\to\infty$ shows that its logarithmic coefficient is the right-hand side of \eqref{crit:eq:residue} with factor $1$ in place of $2^{\sigma(N-2)}$.
Near $(1,0)$, the factor remains $2^{N-2}$, so the logarithmic forms agree after division by these positive constants.
The bounded remainders and the absorption in the proof of Theorem~\ref{thm:localcoercivity} are uniform in these neighborhoods.

For $\sigma>0$, multiplication of a solution by $\sigma^{(N-2)/4}$ sends $(\sigma,2)$ to $(1,2/\sqrt\sigma)$.
For the displayed bubbles, it changes the physical scale $\varepsilon$ to $\sqrt\sigma\varepsilon/2$ and multiplies the energy by $\sigma^{(N-2)/2}>0$.
Thus the scale inequality holds for every fixed $\kappa>2/\sqrt{\sigma_0}$, where $\sigma_0>0$ is small enough for the scalar-flat neighborhood.
\end{proof}

\section{The nonlinear correction and reduced energy}\label{sec:nonlinear-family}
We now pass from the corrected quadratic energy to a nonlinear projected solution.
We localize the full conformal Fermi metric, retaining its high-order trace and mean-curvature terms.
Thus the compactly supported logarithmic metric in this section is allowed to have nonzero trace.
Theorem~\ref{thm:tame} and Proposition~\ref{prop:nonlinear} control the scale derivative of the reduced energy; Proposition~\ref{prop:pointwise} supplies the pointwise bounds used in Section~\ref{sec:comparison}.
\subsection{Expansion of the reduced energy}
Throughout this section, $g=\exp(k)$, where $k_{Ni}=0$ and $k$ is smooth and compactly supported.
Put $X(k):=\norm k_{L^\infty}+\norm{\nabla k}_{L^N}$.
The weighted norm $\N$ is defined in \eqref{eq:Nnorm}.
Define the weak equation residual and its first metric variation by
\[
 \mathscr R_k(v)[\psi]:=\frac{c_N}{2}D_v\A_{\exp(k)}(v)[\psi],\qquad
 \Lambda[k]:=\left.\frac{d}{ds}\mathscr R_{s k}(U)\right|_{s=0}.
\]
\begin{lemma}[Metric coefficient estimates]\label{lem:metric-coefficients}
For the compactly supported tensors above with $X(k)$ sufficiently small, the weighted trace estimate is
\begin{equation}\label{eq:weightedtrace}
 \int_{\partial\R^N_+}\frac{U^2|k|^2}{1+|x|}\le C\N(k)^2.
\end{equation}
The first metric derivative and the nonlinear residual satisfy
\begin{equation}\label{eq:coeffest}
 \begin{aligned}
 \norm{\Lambda[k]}_{Y^*}&\le C\N(k),\\
 \norm{D_v\mathscr R_k(U)-D_v\mathscr R_0(U)}_{\E\to\E^*}&\le C(X(k)+X(k)^2),\\
 \norm{\mathscr R_k(U)-\Lambda[k]}_{Y^*}&\le CX(k)\N(k).
 \end{aligned}
\end{equation}
The constants are independent of the support radius.
\end{lemma}
\begin{proof}
To prove \eqref{eq:weightedtrace}, integrate $\partial_t\bigl(U^2|k|^2/(1+|y|)\bigr)$ over the half-space.
The derivative of the weight is bounded by $CU^2/(1+|y|)^2$, and Cauchy--Schwarz bounds the term containing $k\partial_tk$ by $C\N(k)^2$.
To estimate the metric dependence of the equation, integrate the divergence part of scalar curvature together with the mean-curvature term.
In Fermi block form, the normal curvature flux cancels the $2(N-1)H$ term.
The remaining quadratic scalar form is a sum of terms of the types
\[
 A(k)\nabla v\nabla w,\qquad
 A(k)\partial k\,(v\nabla w+w\nabla v),\qquad
 A(k)(\partial k)^2vw.
\]
Here $A(k)$ denotes smooth coefficient tensors, which may differ between terms, for $\norm k_\infty$ small.
The volume and boundary densities are smooth functions of $k$.
The Sobolev, trace, and H\"older inequalities give \eqref{eq:coeffest}.
Integrating the curvature divergence makes these estimates depend only on the displayed norms of $k$ and its first derivatives; an $L^{N/2}$ bound on its second derivatives is unnecessary.
\end{proof}

For general $k$, define $F(k)$ as the coefficient of $s^2$ in $\A_{\exp(s k)}(U+s v(k))$, where $v(k)=-\B_Y^{-1}\Lambda[k]$ is the linear scalar correction.
This agrees with \eqref{eq:quadratic-action} when $k$ is trace-free.
The nonlinear projected equation is
\begin{equation}\label{eq:nonlinear-projected}
 \mathscr R_k(U+w(k))[\psi]=0\quad(\psi\in Y),\qquad w(k)\in Y.
\end{equation}
For its small solution, define $\mathscr A(k):=\A_{\exp(k)}(U+w(k))$.
The following estimate controls the derivative of the difference between $\mathscr A$ and its quadratic expansion.
\begin{theorem}[Derivative estimate for the reduced energy]\label{thm:tame}
For $X(k)$ small, there is a unique small projected branch $w(k)$, with
\[
 \norm{w(k)}_\E\le C\N(k).
\]
Set $\gamma=\min\{1,2/(N-2)\}$.
The expansion $\mathscr A(k)=\A_\delta(U)+ F(k)+\operatorname{Rem}(k)$ satisfies, for every smooth compactly supported direction $\ell$,
\begin{equation}\label{eq:tame}
 |D\operatorname{Rem}(k)[\ell]|
 \le C\left[(X(k)+\N(k)^\gamma)\N(k)\N(\ell)
                  +X(\ell)\N(k)^2\right].
\end{equation}
The constant is independent of the support radius.
It is uniform for $\kappa$ in a compact interval.
\end{theorem}
\begin{proof}
The maps defined by the interior and boundary nonlinearities are $C^{1,\gamma}$ from $\E$ to its dual.
The critical Sobolev and trace inequalities give this assertion directly; the smaller H\"older exponent comes from $q-1=2/(N-2)$.
The projected inverse following Lemma~\ref{lem:model-kernel} and the coefficient bounds of Lemma~\ref{lem:metric-coefficients} give the projected branch by the implicit function theorem, or by a contraction on a ball of radius $C\N(k)$.

Let $v(k)=-\B_Y^{-1}\Lambda[k]$ be the linear correction.
Subtracting its equation from the projected nonlinear equation gives
\begin{equation}\label{eq:wminusz}
 \norm{w(k)-v(k)}_\E
 \le C\bigl(X(k)+\N(k)^\gamma\bigr)\N(k).
\end{equation}
All powers are estimated in the dual Sobolev or trace norm, so the sign of $\kappa$ causes no change in the estimate.

Since $D w(k)[\ell]\in Y$ and the derivative of the energy with respect to the scalar variable vanishes on $Y$, the envelope identity is exact:
\[
 D\mathscr A(k)[\ell]=D_k\A_{\exp(k)}(U+w(k))[\ell].
\]
Expand the right-hand side in $k,w$.
First metric stationarity removes the constant term.
Its quadratic terms are the derivative of the raw metric quadratic coefficient and the pairing of the first scalar residual $\Lambda[\ell]$ with $v(k)$, namely $D F(k)[\ell]$.
Replacing $v(k)$ by $w(k)$ costs at most $C\N(\ell)\norm{w-v}_\E$.

The remaining term types have the following bounds.
Terms linear in $\ell$ and quadratic in $w$ are bounded by $CX(\ell)\norm w_\E^2$.
Terms with one $k$, one $\ell$, and one $w$ are bounded by
\[
 C\{X(k)\N(\ell)+X(\ell)\N(k)\}\norm w_\E.
\]
Terms of metric degree at least three, differentiated in $\ell$, are bounded by
\[
 C\{X(k)\N(k)\N(\ell)+X(\ell)\N(k)^2\}.
\]
These follow from the coefficient bounds in \eqref{eq:coeffest}.
For example, a term with $\partial\ell\,w\nabla w$ is bounded by $\norm{\partial\ell}_N\norm w_{2N/(N-2)}\norm{\nabla w}_2$, and a term $\ell(\partial k)^2U^2$ by $\norm\ell_\infty\N(k)^2$.
Terms from the volume and area densities are treated by the same inequalities and \eqref{eq:weightedtrace}.
The higher scalar-power terms are bounded by these expressions times $1+\norm w_\E^\gamma$.

Combining these bounds with \eqref{eq:wminusz} proves \eqref{eq:tame}.
The constants depend only on the displayed norms and the projected Jacobi inverse at the fixed curvature pair $(\sigma,\kappa)$, so they are independent of the support radius.
\end{proof}
\medskip
\paragraph{The localized metric.}
We apply the expansion to the localized metric at a boundary center $p$.
Let $h_p^\rho:=\chi(|y|/\rho)h_p$, and let $\mathscr N_p$ denote \eqref{eq:Jnorm} with $J_d=J_d(p)$.
Set
\begin{equation}\label{eq:localized-metric}
 k_{p,\varepsilon}(y)=h_p^\rho(\varepsilon y),\qquad R=\rho/\varepsilon,
 \qquad \mathscr A_p^\rho(\varepsilon)=\mathscr A(k_{p,\varepsilon}).
\end{equation}
The physical metric $\exp(h_p^\rho)$ agrees with $g_p$ on $B_\rho^+$ and with $\delta$ outside $B_{2\rho}^+$.

\begin{lemma}[Localized Taylor estimates]\label{lem:localized-tail}
The metric in \eqref{eq:localized-metric} satisfies
\begin{equation}\label{eq:tailN}
 \begin{aligned}
 \N(\chi_Rq_p(\varepsilon\,\cdot))^2+
 \N\bigl(\varepsilon\partial_\varepsilon[\chi_Rq_p(\varepsilon\,\cdot)]\bigr)^2
 &\le C\varepsilon^{N-2}\rho^{\omega_N},\\
 \N(k_{p,\varepsilon})^2+\N(\varepsilon\partial_\varepsilon k_{p,\varepsilon})^2
 &\le C\{\mathscr N_p(\varepsilon,\rho)+\varepsilon^{N-2}\rho^{\omega_N}\},
 \end{aligned}
\end{equation}
where $\omega_N:=2d_N+4-N$ equals $1$ for odd $N$ and $2$ for even $N$.
In addition,
\begin{equation}\label{eq:smallX}
 X(k_{p,\varepsilon})+X(\varepsilon\partial_\varepsilon k_{p,\varepsilon})\le C\rho.
\end{equation}
The constants are uniform for centers in a compact coordinate chart.
\end{lemma}
\begin{proof}
Taylor's theorem and \eqref{eq:weightedtrace} give the first estimate from
\[
 \varepsilon^{2d_N+2}\int_1^{2\rho/\varepsilon}r^{2d_N+3-N}\dd r
 \le C\varepsilon^{N-2}\rho^{2d_N+4-N}.
\]
The coefficient on the right is bounded by a constant times the square of the uniform derivative bound for the Taylor remainder $q_p$.
The second estimate follows by adding the homogeneous terms.
Scale invariance of the $L^N$ gradient norm gives \eqref{eq:smallX}.
\end{proof}

\begin{proposition}[Scale derivative of the reduced energy]\label{prop:nonlinear}
For $\rho>0$ sufficiently small and then $\varepsilon/\rho$ sufficiently small,
\begin{equation}\label{eq:nonlinear}
 -\varepsilon\partial_\varepsilon \mathscr A_p^\rho(\varepsilon)
 \ge c\mathscr N_p(\varepsilon,\rho)-C\varepsilon^{N-2}\rho^{\omega_N}.
\end{equation}
The constants are uniform for centers in a compact coordinate chart.
For a sequence of rescaled metrics converging smoothly to the flat metric, the coefficient of the last term tends to zero.
\end{proposition}
\begin{proof}
Apply Theorem~\ref{thm:tame} with $\ell=\varepsilon\partial_\varepsilon k$.
By Lemma~\ref{lem:localized-tail}, its error is
\[
 C(\rho+\N(k)^\gamma)\{\mathscr N+\varepsilon^{N-2}\rho^{\omega_N}\}.
\]
Choose $\rho$ to absorb the first coefficient of $\mathscr N$; then $\N(k)\to0$ uniformly as $\varepsilon\to0$.
For any $\delta>0$, the coefficient bounds in Lemma~\ref{lem:metric-coefficients}, polarization, \eqref{eq:tailN}, and Young's inequality bound the quadratic pairings involving $q_p$ by $\delta\mathscr N+C_\delta\varepsilon^{N-2}\rho^{\omega_N}$.
Choose $\delta$ small and absorb $\delta\mathscr N$ in \eqref{eq:Qcoercive}.
This proves \eqref{eq:nonlinear}.
The Taylor remainder $q_p$ and the full energy retain the trace defects and boundary mean-curvature terms.
Under smooth convergence to the flat metric, the Taylor-tail derivative bounds tend to zero, proving the last assertion.
\end{proof}

\subsection{Pointwise estimates for the projected solutions}\label{sec:pointwise}
The comparison with an exact solution requires pointwise control at the mass scale.
We obtain this control for the full localized metric $k=k_{p,\varepsilon}$ by solving its projected equation in weighted spaces.
We first specify these spaces and the corresponding inverse estimate.

Fix $0<\alpha<1$, put $s_y=1+|y|$, and set
\[
 B_y=B_{s_y/8}(y)\cap\overline{\R^N_+},\qquad
 \Delta_x=B_{s_x/8}(x)\cap\partial\R^N_+.
\]
For a function on a set $B$ and a length $s$, write
\[
 |f|_{j,\alpha;s,B}
 =\sum_{a=0}^j s^a\|\nabla^a f\|_{L^\infty(B)}
       +s^{j+\alpha}[\nabla^j f]_{\alpha;B}.
\]
Boundary derivatives in this definition are tangential.
Define
\begin{equation}
 \|w\|_{\mathcal C_U^{2,\alpha}}
 :=\sup_{y\in\overline{\R^N_+}}
        U(y)^{-1}|w|_{2,\alpha;s_y,B_y}.
\end{equation}
Weights $U$ and $s$ are comparable to their values at the center on these sets, uniformly for $\kappa$ in a compact interval.
Use shells
\[
 A_0=\{|y|\le2\},\qquad
 A_j=\{2^j\le |y|\le2^{j+1}\}\quad(j\ge1),
\]
intersected with the half-space.
For a bulk source $f$ and boundary source $b$, put
\begin{equation}\label{eq:sourcenorm}
 \begin{aligned}
 \|(f,b)\|_{\mathcal S_\alpha}
 :=\sum_{j\ge0}\bigg\{
 &\sup_{y\in A_j}\frac{s_y^2}{U(y)}
                |f|_{0,\alpha;s_y,B_y}\\
 &+\sup_{x\in A_j\cap\partial\R^N_+}\frac{s_x}{U(x)}
                |b|_{1,\alpha;s_x,\Delta_x}\bigg\}.
 \end{aligned}
\end{equation}
Empty suprema are zero.
We also use
\[
 \|k\|_{\mathcal M_\alpha}
 =\sum_{j\ge0}\sup_{y\in A_j}|k|_{2,\alpha;s_y,B_y}.
\]
The choice of overlapping local balls only changes these norms by fixed factors.

\begin{lemma}\label{lem:weighted}
Let $J:=-\Delta-\tfrac{\sigma N(N+2)}4U^{4/(N-2)}$ and $B_J:=\partial_\nu-\tfrac{N\kappa}2U^{2/(N-2)}$.
For $\|(f,b)\|_{\mathcal S_\alpha}<\infty$ there is a unique $w\in Y$ and vector $d$ satisfying
\[
 \begin{cases}
 Jw=f+\sum_jd_j\theta_j &\text{in }\R^N_+,\\
 B_Jw=b &\text{on }\partial\R^N_+,
 \end{cases}
 \qquad \int_{\R^N_+}\theta_jw=0\quad(0\le j<N).
\]
Equivalently, $\B(w,\varphi)=\int_{\R^N_+} f\varphi+\int_{\partial\R^N_+} b\varphi$ for $\varphi\in Y$.
The solution obeys
\begin{equation}\label{eq:weightedgreen}
 \|w\|_{\mathcal C_U^{2,\alpha}}+\sum_j|d_j|
       \le C\|(f,b)\|_{\mathcal S_\alpha}.
\end{equation}
The constant is independent of source support and locally uniform in $\kappa$.
Its bulk and boundary representation kernels satisfy
\[
 |\mathcal G_\perp(y,z)|
 \le C\bigl(|y-z|^{2-N}+U(y)U(z)\bigr).
\]
\end{lemma}
\begin{proof}
The source norm embeds into the dual energy space.
Indeed, on a shell of radius $s\ge1$ the bulk bound is $Cs^{-N}$ times its shell coefficient, and the boundary bound is $Cs^{1-N}$ times that coefficient.
Their $L^{2N/(N+2)}$ and $L^{2(N-1)/N}$ norms, respectively, are bounded by that coefficient times $s^{-(N-2)/2}$.
Summation and the core estimate prove the assertion.
Thus the inverse of $\B$ on $Y$ gives an energy solution.
Testing against $Z_k$ gives
\[
 d_k=-\int_{\R^N_+} fZ_k-\int_{\partial\R^N_+} bZ_k,
\]
which also bounds the multipliers.

Compactification by $U^{4/(N-2)}\delta$ gives a smooth Robin problem on a spherical cap when $\sigma=1$, and on a Euclidean ball of radius $1/2$ when $\sigma=0$.
Its inverse on the constrained complement has a kernel bounded by $C(\operatorname{dist}^{2-N}+1)$; the finite-rank projection has smooth kernel.
Conformal covariance and comparability of cap distance with chordal distance give the displayed Euclidean kernel bound.
Boundary representation is its boundary trace.
The constraints become smooth functions vanishing near the omitted point, so this argument introduces no singular projection there.

For the zeroth-order bound, split the potential into source radii below $s_y/2$, comparable with $s_y$, and above $2s_y$.
In the first region a bulk shell contributes at most its source coefficient times $U(y)$, since its density is $O(s^{-N})$; a boundary shell with density $O(s^{1-N})$ has the same bound.
On comparable shells, integrate the locally integrable Newton singularity.
On larger shells, the additional ratio is bounded by $C(s_y/s)^{N-2}$.
Summing bounds the solution by $CU(y)\|(f,b)\|_{\mathcal S_\alpha}$.
The smooth finite-rank term has the same estimate.

Apply interior or Robin Schauder estimates on the local balls, using slightly enlarged balls to control the supremum.
After scaling by $s_y$, the Jacobi coefficients have uniform smooth bounds and the source norms are exactly those in \eqref{eq:sourcenorm}.
Compact multiplier sources affect only finitely many core balls.
This proves the derivative and H\"older terms in \eqref{eq:weightedgreen}.
For points whose local ball meets the plane, use a boundary half-ball of comparable radius; a fixed enlargement suffices.
Approximation by compactly supported sources proves the assertion for finite source norm.
All cap inverses, local Schauder constants, and weight comparisons are uniform for $\kappa$ in a compact interval.
\end{proof}

\begin{proposition}\label{prop:pointwise}
For $\rho$ sufficiently small and $\varepsilon/\rho$ sufficiently small, the projected family satisfies
\begin{equation}\label{eq:Vpositive}
 V_{p,\varepsilon}=U+w_{p,\varepsilon}>0,\qquad
 \|w_{p,\varepsilon}\|_{\mathcal C_U^{2,\alpha}}\le C\rho.
\end{equation}
Moreover $V_{p,\varepsilon}\to U$ locally in $C^2$ as $\varepsilon\to0$, and
\begin{equation}\label{eq:residual}
 \mathscr R_{k_{p,\varepsilon}}(V_{p,\varepsilon})[\varphi]
 =\sum_{j=0}^{N-1}c_j(p,\varepsilon)\int_{\R^N_+}\theta_j\varphi.
\end{equation}
Thus the boundary equation holds exactly, and the interior residual is supported in the fixed support of the $\theta_j$.
Constants are uniform in the local ellipticity and derivative bounds specified in Section~\ref{sec:quadratic}.
\end{proposition}
\begin{proof}
Taylor's theorem, the fixed smooth cutoff, and its differentiated bounds give
\[
 \sup_{y\in A_j}|k_{p,\varepsilon}|_{2,\alpha;s_y,B_y}
 \le C\varepsilon 2^j
\]
for $j\ge1$ when $2^j\le C\rho/\varepsilon$, a bound $C\varepsilon$ on the core, and zero above a fixed multiple of the last shell.
Consequently
\begin{equation}
 \|k_{p,\varepsilon}\|_{\mathcal M_\alpha}\le C\rho.
\end{equation}
These bounds include the transition annulus, where each derivative of the cutoff contributes its inverse radius.

Express the weak residual relative to Euclidean bulk and boundary measures.
Write its restriction to $Y$ at $U+w$ as the fixed Jacobi form plus a source pair.
The metric residual at $U$, denoted $\mathcal R_k$, and the linear metric perturbation $\mathcal P_k w$ satisfy
\begin{equation}\label{eq:pointwiseperturbation}
 \begin{aligned}
 \|\mathcal R_k\|_{\mathcal S_\alpha}&\le C\|k\|_{\mathcal M_\alpha},\\
 \|\mathcal P_k w\|_{\mathcal S_\alpha}
 &\le C\|k\|_{\mathcal M_\alpha}
                  \|w\|_{\mathcal C_U^{2,\alpha}}.
 \end{aligned}
\end{equation}
For example, the bulk terms are combinations of $k\partial^2w$, $\partial k\partial w$, and $(\partial^2k+(\partial k)^2)w$ with smooth small-coefficient factors.
Their local scaled norms, multiplied by $s^2/U$, are bounded by the metric shell coefficient times the solution norm.
Boundary terms use one metric derivative and the boundary $C^{1,\alpha}$ norm.
Density factors obey the same estimates.
Summing these local estimates proves \eqref{eq:pointwiseperturbation}.

On the ball $\|w\|_{\mathcal C_U^{2,\alpha}}\le\delta<1/2$ the scalar remainders obey
\[
 \|\mathcal T_k(w)\|_{\mathcal S_\alpha}\le C\delta^2,
 \qquad
 \|\mathcal T_k(w)-\mathcal T_k(\widetilde w)\|_{\mathcal S_\alpha}
 \le C\delta\|w-\widetilde w\|_{\mathcal C_U^{2,\alpha}}.
\]
Indeed, $U+w\ge U/2$ allows Taylor expansion in $w/U$; the bulk shell factors are $s^2U^{p-1}=O(s^{-2})$ and the boundary shell factors are $sU^{q-1}=O(s^{-1})$.
Both are summable, also after the scaled derivatives required in the norms.
The sign of $\kappa$ changes only a bounded coefficient.

Lemma~\ref{lem:weighted} now makes the projected equation a contraction in $Y\cap\mathcal C_U^{2,\alpha}$ on a ball of radius $C\rho$, after decreasing $\rho$.
This proves positivity and \eqref{eq:Vpositive}.
The weighted norm controls the energy norm because $\int_{\R^N_+} U^2s^{-2}<\infty$.
Hence this branch lies in the fixed small energy neighborhood where Theorem~\ref{thm:tame} gives uniqueness, and is the same branch.
The annihilator of $Y$ consists of the displayed interior functionals, giving \eqref{eq:residual}.
Finally $\|w\|_\E\le C\N(k)\to0$, and local elliptic estimates give the stated convergence.
\end{proof}

\section{Blow-up estimates and comparison with Green functions}\label{sec:local-comparison}
We compare the projected family with an isolated simple blow-up and identify the scale derivative of its energy with the regular part of a Green function difference.
This proves the vanishing of the retained metric coefficients and the local sign restriction used to exclude accumulation of blow-up points in Section~\ref{sec:compactness}.
The mass estimate at the end of the section will convert that sign into the global contradiction.

Throughout this section, the dimension and parameters lie in the ranges of Theorem~\ref{thm:localcoercivity} and Proposition~\ref{prop:local-continuation}.
For varying metrics, common constants require uniform finite derivative bounds in the normalized charts.
The estimates allow arbitrary relative rates of metric convergence and bubble concentration.
When $\sigma=1$, $U_{\kappa_i}$ denotes \eqref{eq:bubble} with $\kappa=\kappa_i$ varying in a fixed compact interval.
When $\sigma=0$, we take $\kappa_i=2$ and use the fixed scalar-flat bubble.
\subsection{Estimates near isolated simple blow-up points}
\begin{lemma}\label{lem:fermisimple}
Let $g_i$ be smooth Fermi metrics on a fixed half-ball $B^+_{r_0}$, with $g_i(0)=\delta$, $(g_i)_{NN}=1$, $(g_i)_{aN}=0$, common ellipticity and finite derivative bounds as in Section~\ref{sec:quadratic}.
Let $\kappa_i$ belong to a fixed compact interval, with $\kappa_i=2$ if $\sigma=0$, and let $u_i>0$ solve \eqref{eq:PDE} there, including its boundary equation on the flat disk.
Suppose, for $\varepsilon_i\downarrow0$, that
\[
 u_i(z)\le C_0|z|^{-(N-2)/2}\quad(0<|z|<r_0),\qquad
 \varepsilon_i^{(N-2)/2} u_i(\varepsilon_i y)-U_{\kappa_i}(y)
       \rightarrow0\quad\hbox{in }C^2_{\rm loc}(\overline{\R^N_+}),
\]
and that
\[
 W_i(r):=r^{(N-2)/2}\overline u_i(r),\qquad
 \overline u_i(r):=\frac1{|S^{N-1}_+|}
       \int_{S^{N-1}_+}u_i(r\omega)\,\dd\omega,
\]
has exactly one critical point in $(0,r_0)$.
Then there are fixed $\rho>0$, $L>1$ and $C>1$ such that, for all sufficiently large $i$,
\begin{equation}\label{eq:fermiupper}
 u_i(z)\le C\varepsilon_i^{(N-2)/2}(\varepsilon_i+|z|)^{2-N}
                 \quad (z\in B^+_\rho).
\end{equation}
Let $G_i$ be the positive Green function on $B^+_{2\rho}$, with pole at zero, zero Dirichlet data on the hemispherical boundary, $B_{g_i}G_i=0$ on the flat boundary away from zero, and $|z|^{N-2}G_i(z)\to1$.
Then
\begin{equation}\label{eq:fermilower}
 u_i(z)\ge C^{-1}\varepsilon_i^{(N-2)/2} G_i(z)
       \quad (L\varepsilon_i\le |z|\le\rho).
\end{equation}
The constants depend only on the displayed local bounds, $C_0$, $r_0$ and the compact parameter interval.
In particular, they are uniform for the smoothly varying and rescaled backgrounds used below.
\end{lemma}
\begin{proof}
We use the annular barrier and Green function argument for the classical simple-point estimate; compare \cite[Proposition~3.9]{AW}.
We give the argument in the chosen Fermi coordinates so that the simplicity assumption agrees with Definition~\ref{def:fermisimple}.

Annular Harnack inequalities apply with uniform constants.
Indeed, after dilation of an annulus of radius $r$, the isolation bound controls the interior coefficient $r^2u_i^{p-1}$ and the Robin coefficient $ru_i^{q-1}$.
Consequently
\begin{equation}\label{eq:fermiharnack}
 \max_{S_r^+}u_i\le C\overline u_i(r)
                  \le C\min_{S_r^+}u_i.
\end{equation}
For every $\omega\in S^{N-1}_+$,
\begin{align*}
 r^{(N-2)/2} U(r\omega)=
 \begin{cases}
 (r+r^{-1}+2\omega_N)^{-(N-2)/2}, &(\sigma,\kappa)=(0,2),\\
 \displaystyle\left(\frac{2}{r+(1+\kappa^2)/r+2\kappa\omega_N}\right)^{(N-2)/2}, &\sigma=1.
 \end{cases}
\end{align*}
Its derivative has the sign of $1-r^2$ in the scalar-flat case and $1+\kappa^2-r^2$ when $\sigma=1$.
Choose a fixed, large $L_0$, uniformly for the parameter interval.
Bubble convergence and the single-critical-point hypothesis give $W_i'(r)<0$ for $L_0\varepsilon_i\le r<r_0$.
They also give, by \eqref{eq:fermiharnack},
\begin{equation}
 r^{(N-2)/2} u_i(z)\le C L_0^{-(N-2)/2}
       \quad (r=|z|\ge L_0\varepsilon_i).
\end{equation}
Increasing the fixed $L_0$ makes both scaled nonlinear coefficients arbitrarily small on this annulus.

Fix
\[
 0<\vartheta<\min\{\frac{N-2}{2},2/p,1/q\}.
\]
We first prove the weaker, but integrable, estimate
\begin{equation}\label{eq:fermirough}
 \varepsilon_i^{(N-2)/2} u_i(\varepsilon_i y)
       \le C(1+|y|)^{2-N+\vartheta}
       \quad (|y|\le d/\varepsilon_i)
\end{equation}
for a sufficiently small fixed $d<r_0$.
Write the equation as $\mathcal P_i u_i=0$, $\mathcal Q_i u_i=0$, where
\[
 \mathcal P_i=\Delta_{g_i}-c_NR_{g_i}+
                 \sigma\tfrac{N(N-2)}4u_i^{p-1},\qquad
 \mathcal Q_i=\partial_{\nu_i}+\frac{N-2}{2} H_{g_i}
                         -\frac{N-2}{2}\kappa_i u_i^{q-1}.
\]
For $\mu=\vartheta,N-2-\vartheta$ use
\[
 F_\mu(z)=r^{-\mu}(1-\varpi z_N/r).
\]
Here $0<\varpi<1/2$ is fixed small compared with $\vartheta(N-2-\vartheta)$; next choose $d$ small and $L_0$ large.
The Fermi identities make $\nu_i=-\partial_N$ on the plane.
Since $g_i(z)-\delta=O(r)$ uniformly, direct differentiation gives
\[
 \mathcal P_iF_\mu
 \le\bigl[-\mu(N-2-\mu)+C(\varpi+d)
                         +C L_0^{-2}\bigr]r^{-\mu-2}<0,
\]
\[
 \mathcal Q_iF_\mu
 \ge\bigl[\varpi-Cd-C L_0^{-1}\bigr]r^{-\mu-1}>0
       \quad\hbox{on the plane}.
\]
These estimates include the scalar and mean curvature terms; their constants are uniform.
A positive strict supersolution also supplies the maximum principle on this mixed annulus: divide a putative negative comparison function by $F_\vartheta$.
The transformed interior zeroth-order coefficient is negative and the boundary coefficient is positive, excluding a negative minimum.

Put $m_i=\max_{S_d^+}u_i$.
Comparison at the inner and outer hemispheres, using the bubble estimate at $L_0\varepsilon_i$, now gives
\begin{equation}\label{eq:fermitwopowers}
 u_i(z)\le C\varepsilon_i^{(N-2)/2-\vartheta}r^{2-N+\vartheta}
                       +C m_i(d/r)^\vartheta
       \quad (L_0\varepsilon_i\le r\le d).
\end{equation}
For fixed $0<\theta<1$, monotonicity and \eqref{eq:fermiharnack} imply
\[
 m_i\le C\theta^{(N-2)/2}\overline u_i(\theta d)
 \le C_{d,\theta}\varepsilon_i^{(N-2)/2-\vartheta}
                         +C\theta^{(N-2)/2-\vartheta}m_i.
\]
Choose $\theta$ so the last coefficient is less than $1/2$.
Absorbing it proves $m_i\le C\varepsilon_i^{(N-2)/2-\vartheta}$.
Inserting this into \eqref{eq:fermitwopowers}, and using $2\vartheta<N-2$, proves \eqref{eq:fermirough}; the bounded bubble region follows directly from the assumed convergence.

We next recover the endpoint exponent.
On $B_d^+$, use the background operators
\[
 \mathcal L_i=-\Delta_{g_i}+c_NR_{g_i},
 \qquad \mathcal B_i=\partial_{\nu_i}+\frac{N-2}{2} H_{g_i}.
\]
For small fixed $d$, their mixed Dirichlet--Robin form is coercive by the Poincare and trace inequalities.
Its Green kernel $K_i$ is positive and satisfies $K_i(x,z)\le C|x-z|^{2-N}$; the solution of the homogeneous equation with hemispherical Dirichlet data bounded by $m_i$ is bounded by $Cm_i$.
These uniform local elliptic estimates follow after scaling to a unit half-ball: the zeroth-order interior and boundary coefficients are $O(d^2)$ and $O(d)$, respectively.
The same kernel restricted to the plane represents Robin boundary data.
Thus Green representation and absolute values give, on $B_{d/2}^+$,
\[
 u_i(x)\le Cm_i+C\int_{B_d^+}|x-z|^{2-N}u_i(z)^p\,\dd v_{g_i}
       +C\int_{D_d}|x-z|^{2-N}u_i(z)^q\,\dd S_{g_i}.
\]
In the bubble variables, the two source majorants supplied by \eqref{eq:fermirough} decay as $(1+|y|)^{-N-2+p\vartheta}$ and $(1+|y|)^{-N+q\vartheta}$.
Our choice of $\vartheta$ makes each integrable in its respective dimension.
Splitting the integrals into $|z|<|x|/2$, $|z|\asymp|x|$, and $|z|>2|x|$ gives
\begin{equation}\label{eq:fermiendpointpre}
 u_i(x)\le Cm_i+C\varepsilon_i^{(N-2)/2}(\varepsilon_i+|x|)^{2-N}
                      \quad (|x|\le d/2).
\end{equation}
For $|x|\ge\varepsilon_i$, divide the middle and outer contributions by $\varepsilon_i^{(N-2)/2}|x|^{2-N}$.
Their bounds are, respectively,
\[
 C(\varepsilon_i/|x|)^{2-p\vartheta}\quad\hbox{in the bulk},\qquad
 C(\varepsilon_i/|x|)^{1-q\vartheta}\quad\hbox{on the boundary}.
\]
Using monotonicity and Harnack once more, now with a fixed $\theta<1/2$, yields
\[
 m_i\le C\theta^{(N-2)/2}\overline u_i(\theta d)
       \le C\theta^{(N-2)/2} m_i+C_{d,\theta}\varepsilon_i^{(N-2)/2}.
\]
Choose $\theta$ small and absorb.
Equation \eqref{eq:fermiendpointpre} proves the upper bound on $B_{d/2}^+$.

For the lower bound, we keep track of the signed boundary source so that the argument also applies when $\kappa_i<0$.
Use the background Green representation on $B_{d/2}^+$, whose homogeneous extension of the positive outer data is nonnegative.
Its kernel is denoted again by $K_i$.
Set $\rho=d/4$ and let $G_i$ be as in the statement.
The normalization of a boundary pole gives
\begin{equation}\label{eq:fermipolekernel}
 K_i(x,0)=\frac{G_i(x)}{(N-2)|S^{N-1}_+|},\qquad
 G_i(x)\asymp |x|^{2-N}\quad (|x|\le\rho).
\end{equation}
Local boundary estimates away from the pole also give
\[
 |K_i(x,z)-K_i(x,0)|\le C|z|\,|x|^{1-N}
                    \quad (|z|\le |x|/2, |x|\le\rho).
\]
These statements follow from the normalized half-space fundamental solution and local elliptic estimates; all constants are uniform.

The total signed source of the flat bubble is strictly positive:
\begin{equation}\label{eq:fermisignedcharge}
 \sigma\frac{N(N-2)}4\int_{\R^N_+}U^p
       +\frac{N-2}{2}\kappa\int_{\partial\R^N_+}U^q
       =(N-2)2^{\sigma(N-2)/2}|S^{N-1}_+|>0.
\end{equation}
Indeed, integration of its equation and boundary equation over a half-ball identifies the left side with the flux through its hemisphere; $U=2^{\sigma(N-2)/2}r^{2-N}+O(r^{1-N})$ gives the displayed limit.
This calculation is valid for either sign of $\kappa$.

Choose a large fixed $A$.
Bubble convergence implies that the signed source in $|z|\le A\varepsilon_i$, divided by $\varepsilon_i^{(N-2)/2}$, differs by $o(1)$ from the corresponding truncated expression in \eqref{eq:fermisignedcharge} with $\kappa=\kappa_i$.
The omitted absolute source is uniformly small as $A\to\infty$, by \eqref{eq:fermirough}.
More precisely, for $r=|x|\ge L\varepsilon_i$, $L\ge2A$, the absolute potential from $|z|>A\varepsilon_i$ is at most
\[
 C\varepsilon_i^{(N-2)/2} r^{2-N}
 \left(A^{-2+p\vartheta}+A^{-1+q\vartheta}
       +L^{-2+p\vartheta}+L^{-1+q\vartheta}\right).
\]
This is the same three-region integral estimate used above, now with the inner region truncated at $A\varepsilon_i$.
By \eqref{eq:fermipolekernel}, replacing the core kernel by $K_i(x,0)$ has error at most $C(A/L)\varepsilon_i^{(N-2)/2} r^{2-N}$.
First choose $A$ so the source tails are small compared with the positive constant in \eqref{eq:fermisignedcharge}, then choose $L$ large compared with $A$, and finally take $i$ large.
The positive core contribution dominates all these errors.
Green representation, including its nonnegative homogeneous term, gives $u_i(x)\ge c\varepsilon_i^{(N-2)/2} G_i(x)$ for $L\varepsilon_i\le |x|\le\rho$, proving \eqref{eq:fermilower}.
\end{proof}
\subsection{Comparison with the projected solutions}\label{sec:comparison}
We first assume that a boundary-type blow-up is isolated simple in the fixed conformal Fermi coordinates of Definition~\ref{def:fermisimple}.
This assumption will be removed in Section~\ref{sec:selection}.
Lemma~\ref{lem:fermisimple} first supplies the sharp pointwise estimates.
\begin{lemma}[Choice of the center and scale]\label{lem:center-scale}
For $V_i$ as in \eqref{eq:Vpositive}, the boundary centers and scales can be adjusted by $o(\varepsilon_i)$ and a factor tending to one so that \eqref{eq:modulate} holds for the rescaled profiles $v_i$.
The resulting profiles satisfy \eqref{eq:sharpdecay}, and the two-sided bounds in Lemma~\ref{lem:fermisimple} remain valid on a smaller fixed half-ball.
\end{lemma}
\begin{proof}
Write, in conformal Fermi coordinates,
\[
 v_i(y)=\varepsilon_i^{(N-2)/2} u_i(\Psi_{p_i}(\varepsilon_i y)),\qquad
 v_i\to U\quad\hbox{locally in }C^2,
\]
where the notation incorporates the conformal change of the solution.
The scales are comparable to $(\max u_i)^{-2/(N-2)}$.
For a small fixed $\rho$, let $V_i:=V_{p_i,\varepsilon_i}$ be the projected solution from \eqref{eq:Vpositive}.
Impose
\begin{equation}\label{eq:modulate}
 \int_{\R^N_+}\theta_j(v_i-V_i)=0\qquad(0\le j<N).
\end{equation}
These equations are just $\int_{\R^N_+}\theta_j(v_i-U)=0$, since $w_i\in Y$.
Their Jacobian in the tangential center and logarithmic scale variables is the nonsingular matrix in \eqref{eq:dual}, up to $o(1)$.
The implicit function theorem solves \eqref{eq:modulate} for the centers and scales; the smooth dependence of the conformal Fermi charts contributes only $o(1)$ to this Jacobian.
The centers move by $o(\varepsilon_i)$ and the scale ratio tends to one, so the old and new quantities $\varepsilon_i+|z|$ are uniformly comparable.
The smooth transition maps and positive conformal factors preserve \eqref{eq:fermiupper}.
On a smaller fixed half-ball, the local Green functions are comparable to $|z|^{2-N}$ outside the bubble neighborhood, so \eqref{eq:fermilower} also remains valid.
Together with Proposition~\ref{prop:pointwise}, these estimates imply
\begin{equation}\label{eq:sharpdecay}
 0<v_i(y),V_i(y)\le C(1+|y|)^{2-N}\quad
 \text{for }|y|<\rho/\varepsilon_i.
\end{equation}
\end{proof}
Put $e_i=v_i-V_i$ and $R_i=\rho/\varepsilon_i$.
Because the two metrics agree exactly on this domain, subtracting their equations gives
\begin{equation}\label{eq:diffeq}
 \mathcal L_i e_i=-\sum_jc_{ij}\theta_j\quad\hbox{weakly in }B_{R_i}^+,
\end{equation}
with the homogeneous secant Robin condition on the flat boundary.
The weak convention absorbs the volume density in the right-hand side.
Here $\mathcal L_i$ is the actual linear conformal operator minus the secant derivatives of the interior and boundary nonlinearities.
In particular, it converges locally to the Jacobi operator.
Since the projected model uses the full localized metric, the right-hand side of \eqref{eq:diffeq} consists exactly of its compactly supported multipliers.

\begin{lemma}\label{lem:expanding}
For the profiles obtained in Lemma~\ref{lem:center-scale},
\begin{equation}\label{eq:expanding}
 \|e_i\|_{L^\infty(B_{R_i}^+)}+\sum_j|c_{ij}|
 \le C\|e_i\|_{L^\infty(S_{R_i}^+)}.
\end{equation}
The constant is uniform as $R_i\to\infty$, after $\rho$ has been fixed sufficiently small.
Consequently,
\begin{equation}\label{eq:massscaleerror}
 \|e_i\|_\infty+\sum_j|c_{ij}|\le C_\rho\varepsilon_i^{N-2}.
\end{equation}
More precisely, take any boundary or interior ball of radius $s/8$ with $s=1+|y|$, whose fixed enlargement avoids the outer spherical boundary.
Then
\[
 |e_i|_{2,\alpha;s,B_{s/8}(y)\cap\overline{\R^N_+}}
 \le C\left(\|e_i\|_{L^\infty(B_{R_i}^+)}+\sum_j|c_{ij}|\right).
\]
The multipliers affect only core balls.
The constants are uniform in the local ellipticity and derivative bounds of Section~\ref{sec:quadratic}, for fixed isolation radius and sharp-decay constant; they are independent of $R_i$.
In particular, the assertion applies to smoothly flattening backgrounds with the same bounds.
\end{lemma}
\begin{proof}
We prove a uniform exterior estimate and then use Jacobi nondegeneracy to control the solution on the core.
The principal metric is $O(\rho)$ close to the Euclidean metric, with $r|\partial g_i|\le C\rho$.
In the exterior of a fixed large ball, the absolute values of the zero-order interior and boundary coefficients are bounded respectively by
\[
 C\{(1+r)^{-4}+\varepsilon_i^2\},\qquad
 C\{(1+r)^{-2}+\varepsilon_i\}.
\]
The terms involving $\varepsilon_i$ are the rescaled scalar and mean curvatures of the original smooth metric.
Against the mixed Dirichlet--Neumann Green kernel on an annular half-ball, their potential norms are bounded by
\begin{equation}
 C(R_0^{-2}+\rho^2),\qquad C(R_0^{-1}+\rho),
\end{equation}
on $R_0<r<R_i$.
For example, the tail integrals at a fixed interior point have powers $r^{-3}\dd r$ and $r^{-2}\dd r$, while a bounded rescaled interior potential has norm $O(\varepsilon_i^2R_i^2)$ and a bounded boundary potential has norm $O(\varepsilon_iR_i)$.
Splitting into comparable annuli proves the same supremum estimates for all points.
The Green bounds for the close metric follow by local flattening and the uniform ellipticity estimates; after scaling they have constants independent of $R_i$.

Choose $R_0$ large, then $\rho$ small.
We obtain the exterior estimate from the mixed Green representation on the annulus.
First invert the metric Laplacian with Dirichlet data on the two spherical boundaries and homogeneous Neumann data on the flat part.
Its Green kernel has the uniform Newton bound.
This can also be seen by even reflection of the block metric across the plane, followed by the Dirichlet Green estimate for a uniformly elliptic divergence operator.
Let $\vartheta\in((N-2)/2,N-2)$.
The radial function $(R_0/r)^\vartheta$ is a supersolution for the metric Laplacian, because its strict Euclidean term $\vartheta(N-2-\vartheta)r^{-\vartheta-2}$ dominates the $O(\rho)$ coefficient errors, and its normal derivative is zero on the plane.
Thus the harmonic extension of bounded inner data and zero outer data is bounded by $C(R_0/r)^\vartheta$.
The harmonic extension of bounded outer data and zero inner data is bounded by the size of the outer data.

Treat the interior and Robin potentials as sources in this Green representation.
Their integral operator has norm at most
\[
 C_\vartheta(R_0^{-1}+\rho+\rho^2)
\]
both in the unweighted supremum norm and in the weighted norm $\sup (r/R_0)^\vartheta|e|$.
Here is the precise weighted estimate.
If $G_i^A$ is the mixed Laplacian kernel on the annulus, $h_\vartheta(y)=(R_0/|y|)^\vartheta$, and $q_i,b_i$ are the absolute interior and boundary potentials, then
\begin{equation}
 \begin{aligned}
 \sup_y h_\vartheta(y)^{-1}\bigg\{
 &\int_{B_{R_i}^+\setminus B_{R_0}^+}G_i^A(y,z)q_i(z)h_\vartheta(z)\dd z\\
 &+\int_{\{z_N=0,\ R_0<|z|<R_i\}}G_i^A(y,z)b_i(z)h_\vartheta(z)\dd S_z\bigg\}
 \le C_\vartheta(R_0^{-2}+R_0^{-1}+\rho^2+\rho).
 \end{aligned}
\end{equation}
The same estimate holds with $h_\vartheta=1$.
To verify it, use the Newton bound and split $|z|<|y|/2$, $|z|\asymp|y|$, and $|z|>2|y|$.
For example, the inner bulk tail is bounded by
\[
 C|y|^{2-N+\vartheta}
       \int_{R_0}^{|y|/2}s^{N-5-\vartheta}\dd s
 \le C_\vartheta R_0^{-2};
\]
at the zero exponent use the bounded function $a^{-2}\log a$, $a\ge1$.
The boundary tail has $s^{N-4-\vartheta}$ in place of $s^{N-5-\vartheta}$ and is $O(R_0^{-1})$.
The bounds for the two remaining regions follow by integrating the Newton singularity on a comparable ball and the powers $s^{-3-\vartheta}$ and $s^{-2-\vartheta}$ on the outer region.
The constant geometric terms give $C_\vartheta\varepsilon_i^2R_i^2$ and $C_\vartheta\varepsilon_iR_i$.
The only inner-end restriction for these estimates is $\vartheta<N-2$.
All terms are small with the stated choices.

Apply the resulting Neumann-series inverse separately to harmonic extension of inner and outer spherical data, using the weighted norm for inner data and the unweighted norm for outer data.
This yields
\begin{equation}\label{eq:annular}
 |e_i(y)|\le
 C(R_0/r)^\vartheta\sup_{S_{R_0}^+}|e_i|
 +C\sup_{S_{R_i}^+}|e_i|\qquad(R_0<r<R_i).
\end{equation}
The multipliers are supported inside $R_0$ and do not appear in this exterior estimate.
Absolute values were used in the integral-operator bounds, so the estimate is independent of the sign of $\kappa$.
Choose a fixed smooth compactly supported cutoff $\zeta_1$ equal to one on the supports of all $\theta_j$.
Then $\int_{\R^N_+}\zeta_1Z_k\theta_j=\delta_{kj}$ by \eqref{eq:dual}.
Testing \eqref{eq:diffeq} against $\zeta_1Z_k$ and transferring derivatives to this fixed test function yields $\sum_j|c_{ij}|\le C\|e_i\|_\infty$.
If \eqref{eq:expanding} fails, divide $e_i,c_i$ by $\|e_i\|_\infty$ and pass to a local limit.
The outer boundary value tends to zero.
The limit $e$ satisfies the Jacobi equation with forcing $-\sum c_j\theta_j$, the homogeneous Jacobi Robin condition, and the compact constraints.
Estimate \eqref{eq:annular} gives $e=O(r^{-\vartheta})$; elliptic estimates give its derivative decay.
Since $\vartheta>(N-2)/2$, $e\in\E$.
Testing against $Z_j$ eliminates all $c_j$.
Jacobi nondegeneracy and the constraints give $e=0$.
The same annular estimate now excludes escape of the normalized maximum to infinity, contradicting its value one.
This proves \eqref{eq:expanding}.

By \eqref{eq:sharpdecay}, the outer boundary value is $O(R_i^{2-N})=O_\rho(\varepsilon_i^{N-2})$, proving \eqref{eq:massscaleerror}.
Scaled Schauder estimates applied to \eqref{eq:diffeq} give the derivative assertions, away from the outer corner if necessary.
\end{proof}

\subsection{The scale derivative and the Green function difference}
For the profiles of Lemma~\ref{lem:center-scale}, let $g_i^\rho$ denote the physical localized metric and define
\[
 V_i^{\rm phys}(x)=\varepsilon_i^{-(N-2)/2}V_i(x/\varepsilon_i).
\]
Let $G_\rho$ be the Green function of the limiting localized metric on the whole half-space, vanishing at infinity and normalized by $\lim_{x\to0}|x|^{N-2}G_\rho(x)=1$.
Smallness of the localized conformal quadratic form ensures existence and positivity.
\begin{lemma}[Outer Green limits]\label{lem:outer-limits}
After passing to a subsequence,
\begin{equation}\label{eq:outerlimits}
 \varepsilon_i^{-(N-2)/2}u_i\rightarrow 2^{\sigma(N-2)/2} G,\qquad
 \varepsilon_i^{-(N-2)/2}V_i^{\rm phys}\rightarrow 2^{\sigma(N-2)/2} G_\rho
 \quad\hbox{locally away from }0.
\end{equation}
The functions $G$ and $G_\rho$ satisfy the same local homogeneous conformal Robin equation, each has singular coefficient one, and $G-G_\rho$ extends smoothly through the origin.
\end{lemma}
\begin{proof}
The isolated-simple estimate and local elliptic estimates give subsequential convergence.
The singular coefficients in \eqref{eq:outerlimits} agree because the nonlinear charges of both families converge to the charge of the same normalized bubble.
In Green representation, the rescaled interior and boundary nonlinear charges converge to those of $U$: their radial tails are bounded by $Cr^{-3}\dd r$ and $Cr^{-2}\dd r$ by \eqref{eq:sharpdecay}.
The compact multiplier charge of the projected model tends to zero.
Integrating the exact bubble equation gives the total charge
\[
 (N-2)2^{\sigma(N-2)/2}|S^{N-1}_+|.
\]
This includes the signed boundary charge when $\kappa<0$.
It is the flux of $2^{\sigma(N-2)/2}r^{2-N}$, and hence proves \eqref{eq:outerlimits} with the displayed normalization.
For the model, the global bound $V_i\le CU$ and decay at infinity identify the unique decaying Green function.

The difference $G-G_\rho$ extends smoothly through the origin.
One may see this from equality of their local distributional source and elliptic removability, or directly from the bounded physical difference furnished by \eqref{eq:massscaleerror}.
\end{proof}
Put
\begin{equation}\label{eq:bdef}
 b_\rho:=(G-G_\rho)(0).
\end{equation}

\begin{lemma}\label{lem:constant}
Along a subsequence for which the limits in Lemma~\ref{lem:outer-limits} hold,
\[
 E_i=\varepsilon_i^{2-N}(v_i-V_i)\rightarrow E,\qquad
 \varepsilon_i^{2-N}c_{ij}\rightarrow\bar c_j
\]
locally in $C^2$.
The limit is bounded, satisfies
\begin{equation}\label{eq:Eeq}
 \begin{cases}
 JE=-\sum_j\bar c_j\theta_j &\text{in }\R^N_+,\\
 B_JE=0 &\text{on }\partial\R^N_+,
 \end{cases}
\end{equation}
and has the limit
\begin{equation}\label{eq:Econstant}
 \lim_{|y|\to\infty}E(y)=2^{\sigma(N-2)/2}b_\rho.
\end{equation}
\end{lemma}
\begin{proof}
Lemma~\ref{lem:expanding} gives boundedness and local convergence.
To identify the limit at infinity, we compare the inner and outer limits through a Green representation on a fixed physical half-ball.
In physical coordinates, set
\[
 H_i(x)=\varepsilon_i^{-(N-2)/2}(u_i(x)-V_i^{\rm phys}(x))=E_i(x/\varepsilon_i).
\]
It is uniformly bounded on the fixed half-ball $B_\rho^+$.
Represent $H_i$ using the Green kernel of the linear conformal Laplacian on this half-ball, with zero Dirichlet condition on the spherical part and homogeneous conformal Robin condition on the flat part.
This mixed operator is coercive when $\rho$ is small.
The harmonic extension of its outer boundary values converges to $2^{\sigma(N-2)/2}(G-G_\rho)$, since that is the unique regular solution with the limiting outer values.
At $x=\varepsilon_i y$, its limit is therefore $2^{\sigma(N-2)/2}b_\rho$.

The remaining terms in the representation are the interior secant-potential source, the boundary secant-potential source, and the compact multiplier source.
After scaling, their coefficients are bounded by $C(1+r)^{-4}$, $C(1+r)^{-2}$, and a fixed compactly supported function, respectively.
The scaled Green kernels converge locally to the Neumann Newton kernel $\mathcal G_N$ of the half-space and obey its uniform pointwise bound.
For fixed $y$, the tails beyond radius $T>2|y|+2$ are bounded, respectively, by
\[
 C\int_T^\infty r^{-3}\dd r=O(T^{-2}),\qquad
 C\int_T^\infty r^{-2}\dd r=O(T^{-1}).
\]
These tail bounds justify convergence of the Green representation over the expanding domains.
The resulting representation is
\begin{align*}
 E(y)=2^{\sigma(N-2)/2}b_\rho
 &+\int_{\R^N_+}\mathcal G_N(y,z)
 \left[\frac{\sigma N(N+2)}4U(z)^{4/(N-2)}E(z)-\sum_j\bar c_j\theta_j(z)\right]\dd z\\
 &+\int_{\partial\R^N_+}\mathcal G_N(y,z)
 \frac{N\kappa}2U(z)^{2/(N-2)}E(z)\dd z.
\end{align*}
Every potential on the right tends to zero at infinity.
This follows by splitting at $|z|=|y|/2$ and $2|y|$: the interior potential is $O(r^{-2})$, with the usual $O(r^{2-N})$ or logarithmic modification in dimensions three and four; the boundary potential is $O(r^{-1})$, with $O(r^{-1}\log r)$ allowed in dimension three.
All tend to zero.
Their derivatives satisfy the corresponding estimates.
This proves \eqref{eq:Econstant} and justifies the Wronskian limit in \eqref{eq:Wronskian}.
\end{proof}

\begin{theorem}\label{thm:matching}
For each fixed sufficiently small $\rho$, with $b_\rho$ defined in \eqref{eq:bdef},
\begin{equation}\label{eq:matching}
 \lim_{i\to\infty}\varepsilon_i^{2-N}
 \left[\varepsilon\partial_\varepsilon \mathscr A_{p_i}^\rho(\varepsilon)\right]_{\varepsilon=\varepsilon_i}
 =C_N b_\rho,
 \qquad C_N=2(N-1)(N-2)2^{\sigma(N-2)}|S^{N-1}|>0.
\end{equation}
The limit is taken along any subsequence on which the outer Green limit has been fixed.
\end{theorem}
\begin{proof}
On the flat boundary, the Jacobi Robin terms in Green's identity cancel.
At infinity,
\[
 Z_0=-\frac{N-2}{2} 2^{\sigma(N-2)/2} r^{2-N}+O(r^{1-N}),\qquad
 \partial_rZ_0=\frac{(N-2)^2}{2}2^{\sigma(N-2)/2}r^{1-N}+O(r^{-N}).
\]
Using \eqref{eq:Eeq} and \eqref{eq:dual}, Green's identity gives
\begin{equation}\label{eq:Wronskian}
 -\bar c_0=\lim_{R\to\infty}\int_{S_R^+}
 (E\partial_rZ_0-Z_0\partial_rE)\dd\sigma
 =\frac{(N-2)^2}{2}2^{\sigma(N-2)}|S^{N-1}_+|b_\rho.
\end{equation}
Lemma~\ref{lem:constant} gives $E-2^{\sigma(N-2)/2}b_\rho=o(1)$ and $r\nabla E=o(1)$, which remove the other flux terms.

There is also an exact identity for the derivative of the reduced energy.
Keeping the physical metric $g_{p_i}^\rho$ fixed, the logarithmic derivative of its physical profile, pulled back to $y$ coordinates, is
\[
 \varepsilon\partial_\varepsilon V_{p_i,\varepsilon}
       -(\frac{N-2}{2}+y\cdot\nabla)V_{p_i,\varepsilon}.
\]
The first term belongs to $Y$, since $\int_{\R^N_+}\theta_j(V_{p_i,\varepsilon}-U)=0$ for every $\varepsilon$ and the duals are fixed in the bubble variables.
Pairing the physical residual with this derivative and using scale invariance of the energy yields
\begin{equation}\label{eq:naturalexact}
 \varepsilon\partial_\varepsilon \mathscr A_{p_i}^\rho(\varepsilon)
 =-2c_N^{-1}\sum_jc_j(p_i,\varepsilon)
 \int_{\R^N_+}\theta_j(\frac{N-2}{2}+y\cdot\nabla)V_{p_i,\varepsilon}.
\end{equation}
The derivative keeps the chosen center and its localized background metric fixed.
Local convergence of $V_i$ to $U$ makes the last integral $\delta_{j0}+o(1)$.
Combine \eqref{eq:naturalexact} with \eqref{eq:Wronskian}.
The coefficient is
\[
 c_N^{-1}(N-2)^2\,2^{\sigma(N-2)}|S^{N-1}_+|
 =2(N-1)(N-2)2^{\sigma(N-2)}|S^{N-1}|.
\]
This proves \eqref{eq:matching} with its sign and normalization.
\end{proof}

\begin{corollary}\label{cor:jets}
At an isolated simple boundary blow-up point $p$, all retained normalized jets vanish:
\[
 J_d(p)=0\quad(1\le d\le d_N),\qquad
 g_p=\delta+O(r^{d_N+1}).
\]
For each fixed small $\rho$,
\begin{equation}\label{eq:bupper}
 b_\rho\le C\rho^{\omega_N}.
\end{equation}
\end{corollary}
\begin{proof}
Combine \eqref{eq:nonlinear} and \eqref{eq:matching}:
\[
 0\le\limsup_i c\varepsilon_i^{2-N}\mathscr N_{p_i}(\varepsilon_i,\rho)
 \le -C_N b_\rho+C\rho^{\omega_N}.
\]
This proves \eqref{eq:bupper} and boundedness of the normalized jet expression.
Its ordinary weights are negative powers of $\varepsilon_i$, and its critical weight is $\log(\rho/\varepsilon_i)\to\infty$.
Thus every retained coefficient tends to zero.
Smooth dependence of the normalized coordinates on $p_i\to p$ identifies the limit with $J_d(p)$.
For $N=3$, there is no retained jet, and the last metric assertion follows from smoothness in Fermi coordinates.
The argument allows arbitrary rates of convergence of $p_i$ to $p$.
\end{proof}

\begin{proposition}[Local sign restriction for a flat limiting metric]
\label{prop:flatsign}
Fix a dimension, boundary class, and curvature pair covered by Theorem~\ref{thm:localcoercivity} or Proposition~\ref{prop:local-continuation}, and a half-ball $B_{2r_*}^+$.
Let $g_i$ be smooth metrics there, written in the conformal Fermi normalization \eqref{eq:fermi-normalization}, and suppose
\[
 g_i=\exp(h_i),\qquad h_i\rightarrow0
       \quad\hbox{in }C^{K+2,\alpha}(\overline{B_{2r_*}^+}).
\]
The centers are the origin; moving centers are included by first pulling back their normalized charts.
Let $u_i>0$ solve \eqref{eq:PDE} for $g_i$ on this half-ball, with its boundary equation on the flat part.
Suppose the origin is an isolated simple boundary-type blow-up point in the sense of Definition~\ref{def:fermisimple}, with uniform isolation constants.
Choose its center and bubble scale $\varepsilon_i\to0$ as in \eqref{eq:modulate}, so that
\[
 \varepsilon_i^{(N-2)/2} u_i(\varepsilon_i y)\rightarrow U(y)
       \quad\hbox{locally in }C^2(\overline{\R^N_+}).
\]
Write the chosen center as the origin in the displayed coordinates.
After reducing $r_*$ if needed, Lemma~\ref{lem:fermisimple} and the center and scale comparison in Section~\ref{sec:comparison} give
\begin{equation}\label{eq:flatsharp}
 u_i(x)\le C\varepsilon_i^{(N-2)/2}(\varepsilon_i+|x|)^{2-N}
       \qquad (|x|<r_*),
\end{equation}
with $C$ independent of $i$.
Every subsequential outer limit has the form
\begin{equation}\label{eq:flatouter}
 2^{-\sigma(N-2)/2}\varepsilon_i^{-(N-2)/2}u_i\rightarrow
 G(x)=|x|^{2-N}+b(x)
       \quad\hbox{locally in }C^2
       (\overline{B_{r_*}^+}\setminus\{0\}),
\end{equation}
where $b$ is smooth and harmonic across the origin, has homogeneous Neumann condition on the plane, and satisfies
\[
 b(0)\le0.
\]
In particular, this conclusion requires no relation between the rate of metric flattening and the rate $\varepsilon_i\to0$.
\end{proposition}
\begin{proof}
All statements are local, so fix once and for all $0<2\rho<r_*$.
Put
\[
 \tau_i=\|h_i\|_{C^{K+2,\alpha}(\overline{B_{2r_*}^+})}\rightarrow0,
 \qquad h_i=\sum_{d=1}^{d_N} J_{i,d}+q_i,
 \qquad d_N=\lfloor(N-2)/2\rfloor.
\]
Taylor's theorem gives $|\partial^j q_i(x)|\le C\tau_i|x|^{d_N+1-j}$ for $0\le j\le d_N+1$; higher derivatives have the corresponding uniform bounds.
Consequently the tail estimate \eqref{eq:tailN}, for this sequence, has the explicit coefficient
\[
 \N(\chi_{\rho/\varepsilon}q_i(\varepsilon\,\cdot))^2+
 \N\!\left(\varepsilon\partial_\varepsilon
       [\chi_{\rho/\varepsilon}q_i(\varepsilon\,\cdot)]\right)^2
 \le C\tau_i^2\varepsilon^{N-2}\rho^{\omega_N}.
\]
The small-coefficient bound is
\[
 X(k_{i,\varepsilon})+X(\varepsilon\partial_\varepsilon k_{i,\varepsilon})
 \le C\tau_i\rho,
\]
where $k_{i,\varepsilon}(y)=\chi(\varepsilon|y|/\rho)h_i(\varepsilon y)$.
The constants in the differentiated expansion and the finite-dimensional quadratic estimates can be fixed uniformly for all large $i$.
Repeating the absorption in Proposition~\ref{prop:nonlinear}, including the mixed jet--tail terms, therefore gives
\begin{equation}\label{eq:flatnonlinear}
 -\varepsilon\partial_\varepsilon \mathscr A_i^\rho(\varepsilon)
 \ge c\mathscr N_i(\varepsilon,\rho)
       -C\tau_i^2\varepsilon^{N-2}\rho^{\omega_N}
\end{equation}
for large $i$ and small $\varepsilon/\rho$.
For $N=3$, the jet expression is zero and the same estimate follows from the tail bound alone.

The estimate \eqref{eq:flatsharp} and local elliptic estimates give outer subsequential convergence.
The charge argument following \eqref{eq:outerlimits} is uniform in this sequence: the interior and boundary nonlinear charges have integrable rescaled tails, and their limits are the charges of $U$.
Thus the singular coefficient is exactly one in \eqref{eq:flatouter}.
The limiting operator is the flat Laplacian with homogeneous Neumann condition.
Its equal-source removability statement, or the local Neumann Green representation, shows that the remainder $b$ extends smoothly through zero.
Even reflection makes it an ordinary harmonic function there.

To apply the matching argument uniformly in the varying metrics, construct the full projected model for each $h_i^\rho$, and compare it with $u_i$ on the fixed physical half-ball $B_\rho^+$.
The constants in Lemma~\ref{lem:expanding} are uniform because \eqref{eq:flatsharp} is uniform and $g_i\to\delta$ smoothly.
The normalized physical difference is therefore bounded there.
The outer limit of the projected model is $2^{\sigma(N-2)/2}|x|^{2-N}$: its limiting localized metric is flat, its singular coefficient is one, and the global decay estimate selects the decaying Neumann Green function.
In the physical Green representation of Lemma~\ref{lem:constant}, the harmonic extensions of the outer difference data now converge to $2^{\sigma(N-2)/2} b(x)$.
The same uniform bulk and boundary tail bounds, $O(T^{-2})$ and $O(T^{-1})$, identify the constant at infinity of the normalized inner difference as $2^{\sigma(N-2)/2} b(0)$.
The Wronskian and exact scale-derivative identities are unchanged.
Hence
\begin{equation}\label{eq:flatmatching}
 \lim_i\varepsilon_i^{2-N}
       [\varepsilon\partial_\varepsilon \mathscr A_i^\rho(\varepsilon)]_{\varepsilon=\varepsilon_i}
       =C_N b(0).
\end{equation}
Combine \eqref{eq:flatnonlinear} and \eqref{eq:flatmatching}, keeping $\rho$ fixed throughout:
\[
 0\le\limsup_i c\varepsilon_i^{2-N}\mathscr N_i(\varepsilon_i,\rho)
       \le -C_N b(0)+
          \lim_i C\tau_i^2\rho^{\omega_N}=-C_N b(0).
\]
Since $C_N>0$, this proves the assertion for arbitrary relative rates of $\tau_i\to0$ and $\varepsilon_i\to0$.
\end{proof}

\subsection{The mass of the localized Green metric}
The local sign bounds the regular difference between the outer Green limit and the localized Green function.
To relate this difference to the global Green mass, we first show that the mass contributed by the localized background tends to zero with its localization radius.

\begin{lemma}\label{lem:auxmass}
Suppose that, in conformal Fermi coordinates,
\[
 g=\delta+O(r^{d_N+1}),\qquad d_N=\lfloor(N-2)/2\rfloor,
\]
with uniform differentiated bounds through order $K+2$ as in Section~\ref{sec:quadratic}.
Let $g_\rho$ be its full cutoff metric, equal to $g$ in $B_\rho^+$ and to $\delta$ outside $B_{2\rho}^+$, and let $G_\rho$ have singular coefficient one and decay at infinity.
Complete the Green metric $G_\rho^{4/(N-2)}g_\rho$ on the punctured half-space by adding the regular point at physical infinity.
This point becomes $z=0$ under inversion; the end at infinity in $z$ coordinates is the end arising from the pole.
The completed scalar-flat minimal-boundary metric satisfies
\begin{equation}\label{eq:auxmass}
 |m(G_\rho^{4/(N-2)}g_\rho)|\le C\rho^{\omega_N},
 \qquad \omega_N=2d_N+4-N>0.
\end{equation}
It is asymptotically flat of every order
\[
 \frac{N-2}{2}<s<\min\{d_N+1,N-2\},
\]
with the right endpoint understood as a limiting exponent.
There is a nonempty interval of such orders in every dimension under consideration.
\end{lemma}
\begin{proof}
First invert using $w=r^{2-N}$, before solving the scalar correction.
Set
\[
 \bar g_\rho=I^*(w^{4/(N-2)}g_\rho),\qquad I(z)=z/|z|^2.
\]
It extends smoothly across $z=0$, equals $\delta$ for $|z|<(2\rho)^{-1}$, and obeys
\[
 |\partial^j(\bar g_\rho-\delta)|\le C|z|^{-d_N-1-j}
 \quad(|z|>\rho^{-1}),
\]
with the corresponding cutoff-annulus bounds.
Consequently,
\begin{equation}\label{eq:invertedenergy}
 \begin{aligned}
 \|\bar g_\rho-\delta\|_\infty+
 \|\partial\bar g_\rho\|_{L^N}&\le C\rho^{d_N+1},\\
 \|\partial\bar g_\rho\|_{L^2}^2&\le C\rho^{2d_N+4-N}.
 \end{aligned}
\end{equation}
All mixed normal--tangential coefficients vanish on the boundary plane; this is preserved by inversion.

Solve
\[
 L_{\bar g_\rho}\phi=0,\qquad B_{\bar g_\rho}\phi=0,
 \qquad \phi-1\in\dot H^1(\R^N_+).
\]
The quadratic conformal form is coercive by the first line of \eqref{eq:invertedenergy}.
Its value on the constant function one, tested against $\psi\in\dot H^1$, is a bounded functional of norm at most $C\|\partial\bar g_\rho\|_2$: integrate the scalar-curvature divergence together with the mean-curvature term, and bound the remaining quadratic coefficient by $\|\partial\bar g_\rho\|_N\|\partial\bar g_\rho\|_2\|\psi\|_{2N/(N-2)}$.
Lax--Milgram gives
\begin{equation}\label{eq:phienergy}
 \|\nabla(\phi-1)\|_2\le C\|\partial\bar g_\rho\|_2.
\end{equation}
We give a uniform weighted version of the potential argument for this correction.
Put $a=d_N+1$, $L=(2\rho)^{-1}$, and $z=L\zeta$.
The coefficient matrix of $\bar g_\rho(L\zeta)$ is $\delta+m_L(\zeta)$, where $m_L=0$ for $|\zeta|<1$ and
\[
 |\partial^j m_L(\zeta)|
 \le CL^{-a}(1+|\zeta|)^{-a-j}\quad(0\le j\le2),
\]
with the corresponding local scaled $C^{2,\alpha}$ bound.
This includes the fixed transition annulus in $\zeta$ coordinates.
Choose
\[
 (N-2)/2<s<\min(a,N-2).
\]
Using the local balls and scaled norms of Section~\ref{sec:pointwise}, define
\begin{align*}
 \|\psi\|_{2,\alpha;s}
 &=\sup_\zeta(1+|\zeta|)^s
       |\psi|_{2,\alpha;1+|\zeta|,B_\zeta},\\
 \|(F,b)\|_{\mathcal P_s}
 &=\sup_\zeta(1+|\zeta|)^{s+2}
       |F|_{0,\alpha;1+|\zeta|,B_\zeta}\\
 &\quad+\sup_{\xi\in\partial\R^N_+}(1+|\xi|)^{s+1}
       |b|_{1,\alpha;1+|\xi|,\Delta_\xi}.
\end{align*}
The Neumann Newton inverse obeys
\begin{equation}\label{eq:decayingNewton}
 \|\mathcal G_N(F,b)\|_{2,\alpha;s}
       \le C_s\|(F,b)\|_{\mathcal P_s},\qquad 0<s<N-2.
\end{equation}
For its supremum estimate, on inner source radii use $r^{2-N}\int_0^r t^{N-1}(1+t)^{-s-2}\dd t\le C_s(1+r)^{-s}$; the boundary integral has the same large-radius exponent.
The outer integral is bounded by $C_s\int_r^\infty t^{-s-1}\dd t$.
Comparable radii use the integrable Newton singularity.
Local Schauder estimates give the derivatives and H\"older seminorms.
These calculations explain both restrictions on $s$.

Write the rescaled conformal equation for $\psi=\phi(L\,\cdot)-1$ as a flat Neumann equation with its coefficient errors on the right.
The source at $\psi=0$ has norm at most $CL^{-a}$ in $\mathcal P_s$: scalar curvature is $O(L^{-a}(1+|\zeta|)^{-a-2})$ and mean curvature is $O(L^{-a}(1+|\xi|)^{-a-1})$.
The linear perturbation maps $\|\psi\|_{2,\alpha;s}$ to the source space with norm at most $CL^{-a}$, including principal, first-order, curvature, and boundary terms.
Equation~\eqref{eq:decayingNewton} therefore gives a convergent Neumann series and
\begin{equation}
 \|\phi(L\,\cdot)-1\|_{2,\alpha;s}\le C_s L^{-a}.
\end{equation}
This solution is in $\dot H^1$ because $2s>N-2$, so uniqueness identifies it with the energy solution above.
In particular $\|\phi-1\|_\infty\le C\rho^{d_N+1}$, and $\phi-1$ has the asserted differentiated decay at every order strictly below $\min(d_N+1,N-2)$.
The strict restriction on the decay order absorbs any logarithm at the endpoint.
The constants are uniform in the local ellipticity and derivative bounds of Section~\ref{sec:quadratic}.
Positivity follows from the supremum bound after decreasing $\rho$.
Thus
\[
 \hat g_\rho=\phi^{4/(N-2)}\bar g_\rho
\]
is scalar flat, has minimal boundary, and has the asserted asymptotic order.
Near $z=0$ the background is exactly flat and the equation has homogeneous Neumann data, so boundary regularity (or even reflection) gives a smooth positive extension across that point.
The extended manifold has no omitted finite-distance point and is complete: its only infinite-distance end is asymptotically a Euclidean half-space.
By inversion and uniqueness, $\hat g_\rho=I^*(G_\rho^{4/(N-2)}g_\rho)$.

We now estimate the mass by the squared gradient of the metric.
Equations \eqref{eq:invertedenergy} and \eqref{eq:phienergy} give
\[
 \|\hat g_\rho-\delta\|_\infty=o(1),\qquad
 \int_{\R^N_+}|\partial\hat g_\rho|^2\le C\rho^{\omega_N}.
\]
Use the scalar-curvature divergence identity
\[
 R_{\hat g}\sqrt{\det\hat g}
 =\partial_i W^i+Q(\hat g,\partial\hat g),\qquad
 W^i=\sqrt{\det\hat g}
 (\hat g^{jk}\Gamma^i_{jk}-\hat g^{ij}\Gamma^k_{jk}),
\]
where $|Q|\le C|\partial\hat g|^2$.
On the plane $\hat g_{aN}=0$, a direct substitution gives $W^N=2H_{\mathrm{tr},\hat g}\sqrt{\det\hat g^T}$.
The plane contribution therefore vanishes exactly, since the boundary is minimal.
At infinity, $W$ differs from the linear ADM flux by $O(|\hat g-\delta||\partial\hat g|)$; its integrated difference tends to zero because the asymptotic order exceeds $(N-2)/2$.
The noncompact-boundary corner correction also vanishes in these block coordinates.
Integrating the divergence identity, with $R_{\hat g}=0$, yields
\[
 |m(\hat g_\rho)|\le C\int_{\R^N_+}|\partial\hat g_\rho|^2.
\]
This proves \eqref{eq:auxmass}; the divergence identity accounts for the full metric contribution to the mass.
\end{proof}

\section{Blow-up points and proof of compactness}\label{sec:compactness}
We apply the blow-up arguments of \cite[Section~8]{KMS}, \cite[Propositions~5.5--5.6]{AW}, and \cite[Section~6]{KMW}, using the local estimates and sign restriction of Section~\ref{sec:local-comparison}.
We record the changes needed for the present boundary conditions and dimension ranges.
\subsection{Isolated simple blow-up points}\label{sec:selection}
Use the local maxima $P_i=\{a_{i,j}\}$ and scales $\varepsilon_{i,j}$ of Proposition~\ref{prop:basiclocal}.
The distance estimate is \eqref{eq:selectionbound}; distances are measured between the actual maxima until their $O(\varepsilon_{i,j})$ offsets from the boundary centers become negligible.

\begin{proposition}\label{prop:simple}
Every isolated boundary-type blow-up point is simple in the sense of Definition~\ref{def:fermisimple}.
This remains true for sequences of smooth backgrounds with uniform local derivative bounds, including the rescaled backgrounds below.
The corresponding statement holds at interior points using Proposition~\ref{prop:interiorlocal}.
\end{proposition}
\begin{proof}
The proof is a minor modification of \cite[Lemma~8.2]{KMS}; the boundary versions are \cite[Proposition~5.5]{AW} and \cite[Lemma~6.2]{KMW}.
Here we use the average in Definition~\ref{def:fermisimple} and the sign restriction in Proposition~\ref{prop:flatsign}.
If the point were not simple, the first critical radius $r_i$ after the bubble maximum would satisfy $r_i\to0$ and $\varepsilon_i/r_i\to0$.
After dilation by $r_i$, the origin is isolated simple on every fixed half-ball of radius less than one.
Lemma~\ref{lem:fermisimple} and annular Harnack inequalities yield a normalized limit $G(y)=|y|^{2-N}+b(y)$, harmonic with homogeneous Neumann data, up to the sphere of radius one.
The identity for the rescaled average is exact, so the critical-radius condition gives
\[
 0=\left.\partial_r\bigl(r^{(N-2)/2}\overline G(r)\bigr)\right|_{r=1}
   =\frac{N-2}{2}\bigl(b(0)-1\bigr).
\]
Here the spherical mean of $b$ equals $b(0)$ by even reflection.
This contradicts Proposition~\ref{prop:flatsign}, applied on a smaller half-ball where the rescaled point is already simple.
The interior argument uses full spherical averages and Proposition~\ref{prop:interiorlocal}.
\end{proof}

\begin{lemma}[Comparison of peaks and distance to the boundary]\label{lem:peak-comparison}
At distinct isolated simple peaks in a rescaled domain, the bubble scales satisfy \eqref{eq:heightcomparison} whenever fixed annuli around them can be joined by Harnack chains avoiding the other peaks.
An interior selected maximum cannot satisfy \eqref{eq:forbiddenoffset}, where $\varepsilon_i$ is its bubble scale and $s_i$ is a lower bound for its distance to the other selected maxima.
\end{lemma}
\begin{proof}
For isolated simple peaks with rescaled scales $\delta_{i,j}\to0$, the two-sided Green estimates on fixed annuli $A_j$ give
\begin{equation}
 c_a\delta_{i,j}^{(N-2)/2}\le u_i\le C_a\delta_{i,j}^{(N-2)/2}
 \quad\text{on }A_j,\qquad j=1,2.
\end{equation}
Harnack chains avoiding the locally finite set of peaks imply
\begin{equation}\label{eq:heightcomparison}
 C^{-1}\le\delta_{i,1}^{(N-2)/2}/\delta_{i,2}^{(N-2)/2}\le C.
\end{equation}
Use Lemma~\ref{lem:fermisimple} and Section~\ref{sec:comparison} at boundary-type peaks, and \cite[Propositions~2.3--2.5]{KMS} at interior peaks.

Also, an interior maximum $a_i$ cannot satisfy
\begin{equation}\label{eq:forbiddenoffset}
 \varepsilon_i\ll d_i:=d_g(a_i,\partial M)\ll s_i,\qquad d_i\to0,
\end{equation}
where $s_i$ is a lower bound for its distance to every other selected point.
Indeed, rescaling by $d_i$ about the nearest boundary point leaves one isolated interior peak at normal height one and sends all other peaks out of compact sets.
Proposition~\ref{prop:simple}, the Green estimates, and Harnack inequalities produce a positive harmonic limit on the punctured half-space with homogeneous Neumann data.
Its even reflection has two equal positive poles.
Comparison with their Dirichlet Green functions on expanding balls makes the regular term at either pole positive, contradicting Proposition~\ref{prop:interiorlocal}.
\end{proof}

\begin{proposition}[Separation of the selected local maxima]\label{prop:separate}
For sufficiently accurate initial bubble approximations and sufficiently large $L$, the mutual distances between distinct selected local maxima whose heights diverge are bounded below by a positive constant independent of the solution.
Consequently, the number of blow-up peaks is uniformly bounded and every blow-up point is isolated simple.
\end{proposition}
\begin{proof}
Follow \cite[Proposition~8.3]{KMS}, using Lemma~\ref{lem:peak-comparison} and Propositions~\ref{prop:flatsign} and~\ref{prop:interiorlocal} in place of the corresponding local estimates and sign restriction.
The boundary counterparts are \cite[Proposition~5.6]{AW} and \cite[Proposition~6.3]{KMW}.
We indicate the boundary adjustments.

If the least distance $s_i$ tends to zero, rescale by $s_i$ about a pair realizing it.
The limiting distance of this pair is one, the selected set is locally finite, and the rescaled domains exhaust the whole space or a half-space.
The bounded-height alternatives are excluded exactly as in \cite[Proposition~8.3]{KMS}: Harnack inequalities contradict a bounded-height peak next to a diverging one, while local boundedness yields a single classified bubble, incompatible with two disjoint selected bubble neighborhoods.
Thus both chosen heights diverge.
A maximum approaching the limiting boundary either lies within $O(\varepsilon_i)$ of its boundary center or satisfies the excluded condition \eqref{eq:forbiddenoffset}.
In the former case, replacing it by the boundary center changes the rescaled position by $o(1)$.
Every diverging peak is therefore isolated simple by Proposition~\ref{prop:simple}.

By \eqref{eq:heightcomparison}, normalization at one peak produces a positive harmonic limit with positive residues at both chosen poles and homogeneous Neumann data when a boundary is present.
Reflect across the plane in that case.
Comparison with the Dirichlet Green functions of the two poles on expanding balls shows that the regular coefficient at either pole is positive; any further poles contribute nonnegative terms.
This contradicts Proposition~\ref{prop:flatsign} for a boundary pole or Proposition~\ref{prop:interiorlocal} for an interior pole.
Uniform separation follows, and compactness of $M$ bounds the number of peaks.
Equation \eqref{eq:selectionbound} and Proposition~\ref{prop:simple} complete the proof.
\end{proof}

\subsection{Proofs of the compactness results}
\begin{proof}[Proof of Theorem~\ref{thm:main}]
Suppose the supremum of positive solutions is unbounded.
Propositions~\ref{prop:basiclocal}, \ref{prop:simple}, and~\ref{prop:separate} leave finitely many, uniformly separated, isolated simple blow-up points.
The two-sided Green estimates and Harnack inequalities yield, after normalization at one point, the positive outer limit
\begin{equation}\label{eq:globalG}
 G=G_p+\sum_{\ell\ne p}a_\ell G_{p_\ell},\qquad a_\ell>0.
\end{equation}
Each Green function has singular coefficient one at its pole; positive conformal type excludes a remaining global solution of the homogeneous equation, as in \cite[Section~6]{AW}.

At a boundary pole, Corollary~\ref{cor:jets} gives $g_p=\delta+O(r^{d_N+1})$ and $b_\rho\le C\rho^{\omega_N}$.
Put $b_\rho^p:=(G_p-G_\rho)(p)$.
Equation \eqref{eq:globalG} and positivity imply $b_\rho^p\le b_\rho$.
Lemma~\ref{lem:auxmass} and the regular difference between $G_p$ and $G_\rho$ show that the complete scalar-flat Green metric has minimal boundary and an asymptotically half-space end of order greater than $(N-2)/2$.
Lemma~\ref{lem:masschange} now gives
\begin{align*}
 m(G_p^{4/(N-2)}g)
 &=m(G_\rho^{4/(N-2)}g_\rho)+4(N-1)|S^{N-1}_+|b_\rho^p\\
 &\le C\rho^{\omega_N}.
\end{align*}
Letting $\rho\downarrow0$ and applying Theorem~\ref{thm:positive-mass-full}, we obtain zero mass and Euclidean half-space rigidity.
Conformal compactification contradicts the exclusion of the round hemisphere.

At an interior pole, Proposition~\ref{prop:interiormass} and Theorem~\ref{thm:positive-mass-full} likewise force the Green metric to be Euclidean space.
Its remaining compact minimal boundary is nonempty, so this is impossible.
The intermediate boundary-distance case was excluded in \eqref{eq:forbiddenoffset}.
Thus $\sup u$ is uniformly bounded.

The global Harnack inequality gives $\sup u\le C\inf u$.
Testing \eqref{eq:PDE} against $u$, positive conformal type and the Sobolev and trace inequalities prevent $\sup u$ from tending to zero.
Hence $\inf u\ge c>0$, and Schauder estimates and bootstrap prove the stated bounds and compactness in every fixed $C^k(M)$; compare \cite[Section~6]{AW}.
The uniform local estimates and the positive coefficient bounds on compact parameter sets give the asserted uniformity in $\kappa$.
\end{proof}

\begin{proof}[Proof of Corollary~\ref{thm:continuation}]
Use Proposition~\ref{prop:local-continuation} in place of Theorem~\ref{thm:localcoercivity} in the proof of Theorem~\ref{thm:main}.
The interior estimates remain applicable because $N\le21$, and the boundary estimates are uniform on compact parameter subintervals.
The rescaling from $(\sigma,2)$ to $(1,2/\sqrt\sigma)$ is already included in Proposition~\ref{prop:local-continuation}.
\end{proof}


\begin{thebibliography}{99}
\bibitem{AlmarazExistence}
S.~Almaraz,
An existence theorem of conformal scalar-flat metrics on manifolds with boundary,
\emph{Pacific J. Math.} \textbf{248} (2010), 1--22.

\bibitem{AlmarazCompactness}
\bysame,
A compactness theorem for scalar-flat metrics on manifolds with boundary,
\emph{Calc. Var. Partial Differential Equations} \textbf{41} (2011), 341--386.

\bibitem{AlmarazBlowup}
\bysame,
Blow-up phenomena for scalar-flat metrics on manifolds with boundary,
\emph{J. Differential Equations} \textbf{251} (2011), 1813--1840.

\bibitem{ABdL}
S.~Almaraz, E.~Barbosa and L.~Lopes de Lima,
A positive mass theorem for asymptotically flat manifolds with a non-compact boundary,
\emph{Comm. Anal. Geom.} \textbf{24} (2016), 673--715.

\bibitem{AQW}
S.~Almaraz, O.~S.~de Queiroz and S.~Wang,
A compactness theorem for scalar-flat metrics on 3-manifolds with boundary,
\emph{J. Funct. Anal.} \textbf{277} (2019), 2092--2116.

\bibitem{AW}
S.~Almaraz and S.~Wang,
A compactness theorem for conformal metrics with constant scalar curvature and constant boundary mean curvature in dimension three,
\emph{Calc. Var. Partial Differential Equations} \textbf{64} (2025), Article No.~35, 24 pp.

\bibitem{Bartnik}
R.~Bartnik,
The mass of an asymptotically flat manifold,
\emph{Comm. Pure Appl. Math.} \textbf{39} (1986), 661--693.

\bibitem{BrendleChen}
S.~Brendle and S.-Y.~S.~Chen,
An existence theorem for the Yamabe problem on manifolds with boundary,
\emph{J. Eur. Math. Soc.} \textbf{16} (2014), 991--1016.

\bibitem{BW}
S.~Brendle and Y.~Wang,
A dimension descent scheme for the positive mass theorem in arbitrary dimension,
preprint (2026), \href{https://arxiv.org/abs/2604.08473v2}{arXiv:2604.08473v2}.

\bibitem{ChenScalar}
S.-Y.~S.~Chen,
Conformal deformation to scalar flat metrics with constant mean curvature on the boundary in higher dimensions,
preprint (2009), \href{https://arxiv.org/abs/0912.1302}{arXiv:0912.1302}.

\bibitem{ChenRuanSun}
X.~Chen, Y.~Ruan and L.~Sun,
The Han--Li conjecture in constant scalar curvature and constant boundary mean curvature problem on compact manifolds,
\emph{Adv. Math.} \textbf{358} (2019), Article No.~106854, 56 pp.

\bibitem{ChenSun}
X.~Chen and L.~Sun,
Existence of conformal metrics with constant scalar curvature and constant boundary mean curvature on compact manifolds,
\emph{Commun. Contemp. Math.} \textbf{21} (2019), Article No.~1850021, 51 pp.

\bibitem{Cherrier}
P.~Cherrier,
Probl\`emes de Neumann non lin\'eaires sur les vari\'et\'es riemanniennes,
\emph{J. Funct. Anal.} \textbf{57} (1984), 154--206.

\bibitem{CM}
S.~Cordero-Misteli,
Compactness and noncompactness for the positive Yamabe--Escobar problem with minimal boundary,
preprint (2026), \href{https://arxiv.org/abs/2610.08759v1}{arXiv:2610.08759v1}.

\bibitem{DK}
M.~M.~Disconzi and M.~A.~Khuri,
Compactness and non-compactness for the Yamabe problem on manifolds with boundary,
\emph{J. Reine Angew. Math.} \textbf{724} (2017), 145--201.

\bibitem{EscobarTrace}
J.~F.~Escobar,
Sharp constant in a Sobolev trace inequality,
\emph{Indiana Univ. Math. J.} \textbf{37} (1988), 687--698.

\bibitem{EscobarJDG}
\bysame,
The Yamabe problem on manifolds with boundary,
\emph{J. Differential Geom.} \textbf{35} (1992), 21--84.

\bibitem{EscobarAnnals}
\bysame,
Conformal deformation of a Riemannian metric to a scalar flat metric with constant mean curvature on the boundary,
\emph{Ann. of Math.} (2) \textbf{136} (1992), 1--50; Addendum, \textbf{139} (1994), 749--750.

\bibitem{Escobar1996}
\bysame,
Conformal deformation of a Riemannian metric to a constant scalar curvature metric with constant mean curvature on the boundary,
\emph{Indiana Univ. Math. J.} \textbf{45} (1996), 917--943.

\bibitem{FelliAhmedou}
V.~Felli and M.~Ould Ahmedou,
Compactness results in conformal deformations of Riemannian metrics on manifolds with boundaries,
\emph{Math. Z.} \textbf{244} (2003), 175--210.

\bibitem{FelliAhmedou2005}
\bysame,
A geometric equation with critical nonlinearity on the boundary,
\emph{Pacific J. Math.} \textbf{218} (2005), 75--99.

\bibitem{GhimentiMicheletti2020}
M.~Ghimenti and A.~M.~Micheletti,
Compactness for conformal scalar-flat metrics on umbilic boundary manifolds,
\emph{Nonlinear Anal.} \textbf{200} (2020), Article No.~111992, 30 pp.

\bibitem{GhimentiMicheletti2021}
\bysame,
A compactness result for scalar-flat metrics on low dimensional manifolds with umbilic boundary,
\emph{Calc. Var. Partial Differential Equations} \textbf{60} (2021), Article No.~119, 24 pp.

\bibitem{PartI}
L.~Gong, S.~Kim, M.~Musso and J.~Wei,
On critical dimensions for compactness in the boundary Yamabe problem, I,
preprint (2026), \href{https://arxiv.org/abs/2609.39945v1}{arXiv:2609.39945v1}.

\bibitem{Gosper}
R.~W.~Gosper, Jr.,
Decision procedure for indefinite hypergeometric summation,
\emph{Proc. Natl. Acad. Sci. USA} \textbf{75} (1978), 40--42.

\bibitem{HanLi}
Z.-C.~Han and Y.~Y.~Li,
The Yamabe problem on manifolds with boundary: existence and compactness results,
\emph{Duke Math. J.} \textbf{99} (1999), 489--542.

\bibitem{HanLi2000}
\bysame,
The existence of conformal metrics with constant scalar curvature and constant boundary mean curvature,
\emph{Comm. Anal. Geom.} \textbf{8} (2000), 809--869.

\bibitem{KMS}
M.~A.~Khuri, F.~C.~Marques and R.~M.~Schoen,
A compactness theorem for the Yamabe problem,
\emph{J. Differential Geom.} \textbf{81} (2009), 143--196.

\bibitem{KWW}
M.~A.~Khuri, J.~Wang and J.~Wang,
Riemannian positive mass theorem in all dimensions in the presence of low-codimension singularities,
preprint (2026), \href{https://arxiv.org/abs/2606.23529v2}{arXiv:2606.23529v2}.

\bibitem{KMW}
S.~Kim, M.~Musso and J.~Wei,
Compactness of scalar-flat conformal metrics on low-dimensional manifolds with constant mean curvature on boundary,
\emph{Ann. Inst. H. Poincar\'e C Anal. Non Lin\'eaire} \textbf{38} (2021), 1763--1793.

\bibitem{LiZhu}
Y.~Y.~Li and M.~Zhu,
Uniqueness theorems through the method of moving spheres,
\emph{Duke Math. J.} \textbf{80} (1995), 383--417.

\bibitem{Marques2005}
F.~C.~Marques,
Existence results for the Yamabe problem on manifolds with boundary,
\emph{Indiana Univ. Math. J.} \textbf{54} (2005), 1599--1620.

\bibitem{Marques2007}
\bysame,
Conformal deformations to scalar-flat metrics with constant mean curvature on the boundary,
\emph{Comm. Anal. Geom.} \textbf{15} (2007), 381--405.

\bibitem{MayerNdiayeMinimal}
M.~Mayer and C.~B.~Ndiaye,
Proof of the remaining cases of the Yamabe boundary problem,
preprint (2015), \href{https://arxiv.org/abs/1505.06114}{arXiv:1505.06114}.

\bibitem{MayerNdiayeScalar}
\bysame,
Barycenter technique and the Riemann mapping problem of Cherrier--Escobar,
\emph{J. Differential Geom.} \textbf{107} (2017), 519--560.

\bibitem{Miao}
P.~Miao,
Positive mass theorem on manifolds admitting corners along a hypersurface,
\emph{Adv. Theor. Math. Phys.} \textbf{6} (2002), 1163--1182.

\bibitem{DLMF}
National Institute of Standards and Technology,
\emph{NIST Digital Library of Mathematical Functions},
\url{https://dlmf.nist.gov/}, accessed October 7, 2026.

\bibitem{SchoenYau}
R.~Schoen and S.-T.~Yau,
On the proof of the positive mass conjecture in general relativity,
\emph{Comm. Math. Phys.} \textbf{65} (1979), 45--76.

\bibitem{Witten}
E.~Witten,
A new proof of the positive energy theorem,
\emph{Comm. Math. Phys.} \textbf{80} (1981), 381--402.

\end{thebibliography}
\end{document}